\documentclass[10pt,a4paper]{amsart}
\usepackage[T1]{fontenc} 
\usepackage{amsmath,amsthm,amssymb}
\numberwithin{equation}{section}
\usepackage{times}
\usepackage{bbm}
\usepackage{enumitem}
\usepackage[left=1.15in,right=1.15in,top=1.15in,bottom=1.15in]{geometry} 
\usepackage{calc}
\usepackage{ifthen,tabularx,graphicx,multirow}
\graphicspath{ %
	{Figures/} %
	{Figures/Zoom} %
}
\usepackage{rotating}
\usepackage{hyperref}
\usepackage{tikz}

\usetikzlibrary{tikzmark}
\usepackage{ etoolbox}
\usepackage{tcolorbox}
\usepackage{makecell}
\usepackage{algorithm2e}
\usepackage{cite}
\usepackage{tikz-cd}
\usepackage{tikz,tikz-cd}
\usepackage{enumitem}
\usepackage{changepage}
\numberwithin{equation}{section}

\newtheorem{lemma}{Lemma}
\newtheorem{proposition}[lemma]{Proposition}
\newtheorem{theorem}[lemma]{Theorem}
\numberwithin{lemma}{section}
\newtheorem{corollary}[lemma]{Corollary}
\newtheorem{definition}[lemma]{Definition}

\theoremstyle{definition}
\newtheorem{remark}[lemma]{Remark}

\theoremstyle{definition}

\theoremstyle{definition}

\theoremstyle{definition}

\newcommand*{\MA}[1]{{\color{magenta}#1}}

\DefineNamedColor{named}{JungleGreen} {cmyk}{0.99,0,0.52,0}

\usepackage{array}
\usepackage{booktabs}

\usepackage[foot]{amsaddr}

\begin{document}
\RestyleAlgo{boxruled}

\title[Computing geometric features of spectra]{On the computation of geometric features of spectra of linear operators on Hilbert spaces}

\author{Matthew J. Colbrook}
\address{DAMTP, Centre for Mathematical Sciences, University of Cambridge CB3 0WA,
United Kingdom}
\email{m.colbrook@damtp.cam.ac.uk}

\begin{abstract}
Computing spectra is a central problem in computational mathematics with an abundance of applications throughout the sciences. However, in many applications gaining an approximation of the spectrum is not enough. Often it is vital to determine geometric features of spectra such as Lebesgue measure, capacity or fractal dimensions, different types of spectral radii and numerical ranges, or to detect essential spectral gaps and the corresponding failure of the finite section method. Despite new results on computing spectra and the substantial interest in these geometric problems, there remain no general methods able to compute such geometric features of spectra of infinite-dimensional operators. We provide the first algorithms for the computation of many of these longstanding problems (including the above). As demonstrated with computational examples, the new algorithms yield a library of new methods. Recent progress in computational spectral problems in infinite dimensions has led to the Solvability Complexity Index (SCI) hierarchy, which classifies the difficulty of computational problems. These results reveal that infinite-dimensional spectral problems yield an intricate infinite classification theory determining which spectral problems can be solved and with which type of algorithm. This is very much related to S. Smale's comprehensive program on the foundations of computational mathematics initiated in the 1980s. We classify the computation of geometric features of spectra in the SCI hierarchy, allowing us to precisely determine the boundaries of what computers can achieve (in any model of computation) and prove that our algorithms are optimal. We also provide a new universal technique for establishing lower bounds in the SCI hierarchy, which both greatly simplifies previous SCI arguments and allows new, formerly unattainable, classifications.
\end{abstract}




\maketitle


\tableofcontents

\newpage
\section{Introduction}
\label{intro}

This paper resolves open computational spectral problems related to geometric features of spectra of operators. In other words, we consider the following problem:

\vspace{2mm}

\begin{adjustwidth}{+5mm}{+5mm}
\noindent{}\textit{Are there algorithms that given a bounded}\footnote{Many of our algorithms can also be extended to unbounded operators.} \textit{operator $A\in\mathcal{B}(l^2(\mathbb{N}))$, approximate key geometric features (e.g., spectral gaps, notions of sizes and capacity, measures, topological features such as fractal dimensions, etc.) of the set $\mathrm{Sp}(A)$ from a matrix representation of $A$?}
\end{adjustwidth}

\vspace{2mm}

\noindent{}To answer this question, we use the newly established Solvability Complexity Index (SCI) hierarchy \cite{ben2015can,hansen2011solvability,colbrook2020foundations}, a classification tool that determines the boundaries of what is computationally possible. Classifying spectral problems and providing a library of optimal algorithms\footnote{For precise notions of algorithm, see \S \ref{bigHth}.} remains largely uncharted territory in the foundations of computational mathematics. In exploring this territory, there will, necessarily, have to be many different types of algorithms, as different structures on the various classes of operators and different spectral properties require different techniques.

A famous example of the above question is the almost Mathieu operator on $l^2(\mathbb{Z})$ (see \S \ref{jhnwagtwqrfg}):
$$
(H_{\alpha}x)_n=x_{n-1}+x_{n+1}+2\lambda\cos(2\pi n\alpha)x_n,
$$
which induces the Hofstadter butterfly \cite{hofstadter1976energy}. The almost Mathieu operator plays an important role in physics \cite{last2005spectral}, arising in the study of the quantum Hall effect \cite{thouless1982quantized}, and has become a laboratory for exploring the spectral properties of ergodic Schr\"odinger operators \cite{jitomirskayacritical}. When $\alpha$ is irrational, the Lebesgue measure of the spectrum is $4\left|1-\left|\lambda\right|\right|$. This formula was conjectured based on the numerical work of S. Aubry \& G. Andr{\'e} \cite{aubry1980analyticity} and became one of B. Simon's problems for the 21st century \cite{simon2000schrodinger}. It was later proven by A. Avila \& R. Krikorian \cite{avila2006}. Similarly, M. Kac's ``Ten Martini Problem'', that the spectrum is a Cantor set for all irrational $\alpha$ and $\lambda>0$, was conjectured by M. Azbel \cite{azbel1964energy} and also became one of B. Simon's problems. This problem attracted a host of numerical and analytical work (see the summary in \cite{last2005spectral}), before being proven by A. Avila \& S. Jitomirskaya \cite{avila2009ten}. In both of these examples, we see a crucial interplay between computation, conjecture, and mathematical proof. The above geometric features of spectra play an important role in the physics of the underlying quantum system \cite{han1994critical,ketzmerick1992slow,ketzmerick1997determines,sire1989electronic}. The almost Mathieu operator is by no means unique in this regard and there is a growing literature on computational studies of geometric features of spectra in diverse areas of physics \cite{dean2013hofstadter,geim2013van,hunt2013massive,ponomarenko2013cloning,naumis2017electronic,roman2014spectral,bandres2016topological,tanese2014fractal,kohmoto1987critical,levi2011disorder,rivera2018fractal,new2001diffractive,luitz2017ergodic,torres2015dynamics}.

However, there is a current lack of \textit{rigorous computational theory} and \textit{convergence analysis}, and no known algorithms can tackle general cases. Moreover, the foundations of computation (i.e., what is and what is not computationally possible) for computing geometric features of spectra are almost entirely unexplored. We solve these open problems and others by providing algorithms that compute geometric features of spectra and by classifying the computational problems in the SCI hierarchy.

\subsubsection*{\textbf{The SCI hierarchy:}}
The SCI hierarchy has recently been used to resolve the problem of computing spectra of general bounded operators in infinite dimensions \cite{ben2015can,hansen2011solvability}, and is now being used to explore the foundations of computation in many diverse areas of mathematics \cite{opt_big,boche2019solvability,webb2021spectra,ben2021universal_periodic,Jonathan_res,benartzi2020computing,ben2022universal_inverse,rosler2019solvability,rosler2021computing_rough,colbrook2021PNAS,becker2020computing,colbrook2021computing,colb1,colb2,colbrookPSEUDO,colbrook3,firenet_SIAM_NEWS,colbrook2022computing,CRAS_SCI}.\footnote{For related work on practical infinite-dimensional numerical linear algebra, see \cite{webb2021spectra,colbrook2021computingSIREV,Olver_Townsend_Proceedings, Olver_SIAM_Rev, Olver_code1, Olver_code2,gilles2019continuous,olver2009gmres,horning2019feast}, and for rigorous data-driven algorithms for spectral properties of Koopman operators (operators on infinite-dimensional spaces that globally linearise non-linear dynamical systems), see \cite{colbrook2021rigorous,ResDMD_JFM,colbrook2022mpedmd}.} Whilst for some classes of operators one can compute spectra with error control \cite{colb1,colbrook3,colbrook_IMA_LT}, a potentially surprising consequence is that, for general operators, one needs several successive limits to compute the spectrum. Since traditional approaches are dominated by techniques based on one limit, this explains why many computational spectral problems remain unsolved and opens the door to an infinite classification theory. Moreover, this phenomenon is not just restricted to spectral problems but is shared by other areas of computational mathematics. An example is S. Smale's problem of root-finding of polynomials with rational maps \cite{smale_question}, which also requires several successive limits as established by C. McMullen \cite{McMullen1, mcmullen1988braiding} and P. Doyle \& C. McMullen \cite{Doyle_McMullen}. These results can be expressed in terms of the SCI hierarchy \cite{ben2015can}, which generalises S. Smale's seminal work \cite{Smale2, Smale_Acta_Numerica}  with L. Blum, F. Cucker, M. Shub \cite{BSS_Machine,BCSS,Cucker_AH_real}, and his program on the foundations of scientific computing and the existence of algorithms. Many other problems in the foundations of computations, such as the work by S. Weinberger \cite{Weinberger}, can also be viewed in the context of the SCI hierarchy.

The SCI hierarchy is further motivated by computer-assisted proofs. Computer-assisted proofs are rapidly becoming an essential part of modern mathematics \cite{gowers} and, perhaps surprisingly, non-computable problems can be used in computer-assisted proofs. Examples include the recent proof of Kepler's conjecture (Hilbert's 18th problem) \cite{Hales_Annals, hales_Pi} on optimal packings of $3$-spheres, led by T. Hales, and the Dirac--Schwinger conjecture on the asymptotic behaviour of ground states of certain Schr\"{o}dinger operators, proven in a series of papers by C. Fefferman and L. Seco \cite{fefferman1990, fefferman1992, fefferman1993aperiodicity,  fefferman1994, fefferman1994_2, fefferman1995, fefferman1996interval, fefferman1996, fefferman1997}. Both of these proofs rely on computing non-computable problems. This apparent paradox can be explained by the SCI hierarchy (the $\Sigma^A_1$ and $\Pi_1^A$ classes described below become available for computer-assisted proofs); Hales, Fefferman and Seco implicitly prove $\Sigma^A_1$ classifications in the SCI hierarchy in their papers. Some of the problems we consider also lie in $\Sigma^A_1\cup\Pi_1^A$, meaning that they can be used for computer-assisted proofs.

\subsubsection*{\textbf{The problems addressed in this paper:}} The algorithms we provide are sharp in the SCI hierarchy, meaning that they realise the boundaries of what computers can achieve. Table \ref{result_table} provides a summary of the main SCI classifications of this paper. The main theorems are contained in \S \ref{main_results_section}, including further motivations and classifications for different classes of operators. We provide resolutions to the following problems:

\begin{enumerate}[leftmargin=0.4cm,rightmargin=0cm]
  \item[(i)] \textit{Computing spectral radii, essential spectral radii, polynomial operator norms and capacity of spectra.} The spectral radius is perhaps the most basic geometric property of spectra and arises in stability analysis. We show that computing the spectral radius is high up in the SCI hierarchy for non-normal operators. In fact, it has the same classification in the SCI hierarchy for general bounded operators as that of computing the spectrum itself. Classifications are given for different types of operators (e.g., known column decay, control on resolvent norms) and also for the essential spectral radius. In many cases, the problem of computing polynomial operator norms is easier in the sense of SCI hierarchy. We also consider the problem of computing the logarithmic capacity of the spectrum, following the work of P. Halmos \cite{halmos1971capacity}, which has applications in orthogonal polynomials, approximation theory and when studying the convergence of Krylov methods (see, for example, the work of O. Nevanlinna \cite{nevanlinna2012convergence,nevanlinna1995hessenberg,nevanlinna1990linear} and  U. Miekkala \& O. Nevanlinna \cite{miekkala1996iterative}).
	\item[(ii)] \textit{Computing essential numerical ranges, gaps in essential spectra, and determining whether spectral pollution occurs on sets}. We provide classification results for the essential numerical range, which also hold in the case of unbounded operators. In connection with computing spectra, there has been a substantial effort in studying the finite section method and locating gaps in essential spectra of operators (see the discussion in \S \ref{fs_fails}). When using the finite section method to approximate spectra of self-adjoint operators, spurious eigenvalues, known as spectral pollution, can occur anywhere within these gaps. Paradoxically, we show that determining if spectral pollution occurs on a given set is strictly harder in the sense of the SCI hierarchy than computing the spectrum itself. Hence, computing a failure flag for the finite section method is, in a certain sense, strictly harder than solving the original problem for which it was designed. Moreover, we establish the SCI of detecting gaps in essential spectra of self-adjoint operators, a problem that arises in areas such as perturbation theory and defect models.
	\item[(iii)] \textit{Computing Lebesgue measure of spectra and pseudospectra, and determining if the spectrum is Lebesgue null}. An important property of the spectrum is its Lebesgue measure, with recent progress in the field of Schr{\"o}dinger operators with random or almost periodic potentials \cite{puig2004cantor,avila2006,avila2009ten,avila2007simplicity,beckus2013spectrum}. If the spectrum of an operator is  Lebesgue null, then this implies the absence of absolutely continuous spectra\footnote{For algorithms that compute spectral measures and decompositions, see \cite{colbrook2021computing,colbrook2021computingSIREV,colbrook2022specsolve} and their recent physical applications in \cite{johnstone2021bulk,colbrook2021computingTI}.}, which is related to transport properties if the operator represents a Hamiltonian. Whilst results are known for specific one-dimensional examples such as the almost Mathieu operator \cite{avila2006} or the Fibonacci Hamiltonian \cite{sutHo1989singular}, very little is known in the general case or higher dimensions. This is reflected by the difficulty of performing rigorous numerical studies, despite many examples studied in the physics literature (see the references in \cite{avila2017spectral,benza1991band,sire1989electronic}). We provide the first algorithms for computing the Lebesgue measure of spectra and pseudospectra, and determining if the spectrum is Lebesgue null, for many different classes of operators.
\item[(iv)] \textit{Computing fractal dimensions of spectra.} Fractal dimensions of spectra are important in many applications. For example, in quantum mechanics, they lead to upper bounds on the spreading of wavepackets and are related to time-dependent quantities associated with wave functions \cite{han1994critical,ketzmerick1992slow,ketzmerick1997determines}. Fractal spectra appear in a wide variety of contexts, such as exciting new results in multilayer materials (e.g., bilayer graphene) \cite{dean2013hofstadter,geim2013van,hunt2013massive,ponomarenko2013cloning}, strained materials \cite{naumis2017electronic,roman2014spectral} or quasicrystals \cite{bandres2016topological,tanese2014fractal,kohmoto1987critical,levi2011disorder}. Another well-studied area where fractal spectral properties appear is optics \cite{rivera2018fractal,new2001diffractive}, following the analytical and numerical work of M. Berry and coauthors \cite{berry2001fractal,berry2001theory,berry2004physics}. Despite the physical importance of fractal dimensions, analytical results are known only for a limited number of specific models. Moreover, there are currently no algorithms for computing fractal dimensions of spectra for general operators, or even tridiagonal self-adjoint operators. We provide the first algorithms for computing the box-counting and Hausdorff dimensions of spectra for many different classes of operators.
\end{enumerate}

\subsubsection*{\textbf{Contributions to the SCI hierarchy itself:}}

Our final contribution is a new tool to prove lower bounds (impossibility results) in the SCI hierarchy. This is crucial for some of the classifications of the above problems, and holds regardless of the model of computation. We show that for a certain special class of combinatorial problems, the SCI hierarchy is equivalent to the Baire hierarchy from descriptive set theory (this equivalence does not hold in general). By embedding these combinatorial problems into spectral problems\footnote{This technique is not restricted to spectral problems - it can be adapted to other scenarios.}, we provide the first technique for dealing with problems that have SCI greater than three, and also greatly simplify the proofs of results lower down in the SCI hierarchy. However, it should be stressed that this is not a paper on descriptive set theory or mathematical logic. Our discussion is entirely self-contained and written for a wide audience from a primarily computational background.

\subsubsection*{\textbf{Outline of paper:}}

In \S \ref{sec:rev_new_sec2}, we provide a brief summary of the SCI hierarchy and define the classes of operators for the interpretation of Table \ref{result_table} and the main results. A detailed discussion of the SCI hierarchy is delayed until \S \ref{SCI_Hierarchy}. In \S \ref{main_results_section}, we summarise our main results on the classification of computational spectral problems. Computational examples are then given in \S \ref{num_test}. For example, we provide numerical evidence that a portion of the spectrum of the graphical Laplacian on an infinite Penrose tile is Lebesgue null and fractal, with a fractal dimension of approximately 0.8, and that the whole spectrum has a logarithmic capacity of approximately $2.26$. Mathematical preliminaries, including definitions of the SCI hierarchy and the new tool to provide lower bounds in the SCI hierarchy, are presented in \S \ref{bigHth}. Proofs of our results are given in \S \ref{first_set_proofsdfhedh}--\S\ref{pf_frac}. To make the paper self-contained, we include a short appendix on the results/algorithms of \cite{colb1}, which are used in some of our proofs. Pseudocode for many of the new algorithms is provided in Appendix \ref{compute_rout_appendix}. We use \hspace{-3mm}$\qed$ to denote the end of a proof and $\boxtimes$ to denote the end of a remark.

\begin{table}
\begin{center}
\addtolength{\leftskip} {-1.6cm}
    \addtolength{\rightskip}{-2cm}

\hspace{-0.5cm}\begin{tabular}{|p{7.2cm}|p{6.7cm}|l|}
\hline
 \rule{0pt}{1.1em}\textbf{Description of Problem}&\textbf{SCI Hierarchy Classification}&\textbf{Theorem}\\
 \hline
 \hline

\rule{0pt}{1.1em}Computing the spectral radius.&\makecell[l]{Varies. e.g., Normal operators: $\in\Sigma_1^A$, $\not\in\Delta_1^G$,\\Controlled resolvent: $\in\Sigma_2^A$, $\not\in\Delta_2^G$,\\General bounded operators: $\in\Pi_3^A$, $\not\in\Delta_3^G$}&\ref{spec_rad_thm}\\

\hline

\rule{0pt}{1.1em}Computing the essential spectral radius.&\makecell[l]{Varies. e.g., Most classes: $\in\Pi_2^A$, $\not\in\Delta_2^G$,\\General bounded operators: $\in\Pi_3^A$, $\not\in\Delta_3^G$}&\ref{spec_rad_thm2}\\

\hline

\rule{0pt}{1.1em}Computing polynomial operator norms.&\makecell[l]{With bounded dispersion: $\in\Sigma_1^A$, $\not\in\Delta_1^G$\\
Without bounded dispersion: $\in\Sigma_2^A$, $\not\in\Delta_2^G$}&\ref{spec_rad_thm3}\\

\hline

\rule{0pt}{1.1em}Computing the capacity of the spectrum.&\makecell[l]{With bounded dispersion: $\in\Pi_2^A$, $\not\in\Delta_2^G$\\
Without bounded dispersion: $\in\Pi_3^A$, $\not\in\Delta_3^G$}&\ref{spec_rad_thm3}\\

\hline

\rule{0pt}{1.1em}Computing gaps in the essential spectrum.&\makecell[l]{$\in\Sigma_3^A$, $\not\in\Delta_3^G$}&\ref{spec_poll_hard}\\

\hline
\rule{0pt}{1.1em}Computing the essential numerical range.&\makecell[l]{$\in\Pi_2^A$, $\not\in\Delta_2^G$}&\ref{spec_poll_hard}\\

\hline
\rule{0pt}{1.1em}Determining if spectral pollution can occur on a set
(i.e., failure of finite section method).&\makecell[l]{$\in\Sigma_3^A$, $\not\in\Delta_3^G$}&\ref{spec_poll_hard}\\

\hline

\rule{0pt}{1.1em}Computing the Lebesgue measure of the spectrum.&\makecell[l]{Varies. e.g.,\\Bounded dispersion: $\in\Pi_2^A$, $\notin\Delta_2^G$,\\Self-adjoint and general bounded: $\in\Pi_3^A$, $\notin\Delta_3^G$}&\ref{Leb_1}\\

\hline

\rule{0pt}{1.1em}Computing the Lebesgue measure of the pseudospectrum.&\makecell[l]{Varies. e.g.,\\Bounded dispersion: $\in\Sigma_1^A$, $\notin\Delta_1^G$,\\Self-adjoint and general bounded: $\in\Sigma_2^A$, $\notin\Delta_2^G$}&  \ref{Leb_2}\\

\hline

\rule{0pt}{1.1em}Determining if the Lebesgue measure of the spectrum is zero.&\makecell[l]{Varies. e.g.,\\Bounded dispersion: $\in\Pi_3^A$, $\notin\Delta_3^G$,\\Self-adjoint and general bounded: $\in\Pi_4^A$, $\notin\Delta_4^G$}&\ref{Leb_3}\\

\hline

\rule{0pt}{1.1em}Computing the box-counting dimension of the spectrum (when it exists).&\makecell[l]{Varies. e.g.,\\Bounded dispersion:$\in\Pi_2^A$, $\notin\Delta_2^G$,\\Self-adjoint: $\in\Pi_3^A$, $\notin\Delta_3^G$}&\ref{fractal_theoremhhhjh}\\

\hline

\rule{0pt}{1.1em}Computing the Hausdorff dimension of the spectrum.&\makecell[l]{Varies. e.g.,\\Bounded dispersion:$\in\Sigma_3^A$, $\notin\Delta_3^G$,\\Self-adjoint: $\in\Sigma_4^A$, $\notin\Delta_4^G$}&\ref{fractal_theoremhhhjh}\\

\hline

\end{tabular}
\vspace{3mm}
\caption{Summary of the main results for the readable information $\Lambda_1$ consisting of matrix values. The main theorems contain classifications for different classes of operators (see Table \ref{definition_rev_table2} for the operator classes).}
\vspace{-3mm}
\label{result_table}
\end{center}
\end{table}

\section{Essentials of the SCI Hierarchy and Preliminary Definitions}
\label{sec:rev_new_sec2}

\subsection{A brief introduction to the SCI Hierarchy}
\label{sec:SCI_brief_intro_FF}

\subsubsection{Description of the SCI Hierarchy}

First, we define a computational problem. The basic objects of a computational problem are:
\begin{itemize}
	\item $\Omega$, called the \emph{domain},
	\item $\Lambda$, a set of complex-valued functions on $\Omega$, called the \emph{evaluation set},
	\item $(\mathcal{M},d)$, a metric space,
	\item $\Xi:\Omega\to \mathcal{M}$ the \emph{problem function}.
\end{itemize}
The set $\Omega$ is the set of objects that give rise to our computational problems, the goal being to compute the problem function $\Xi : \Omega\to \mathcal{M}$. The set $\Lambda$ is the collection of functions that provide us with the information we are allowed to read as input to the algorithm. This leads to the following definition:

\begin{definition}[Computational problem]\label{def:comp_prob}
Given a domain $\Omega$; an evaluation set $\Lambda$, such that for any $A_1, A_2 \in \Omega$, $A_1 = A_2$ if and only if $f(A_1) = f(A_2)$ for all $f \in \Lambda$; a metric space $\mathcal{M}$; and a problem function $\Xi:\Omega\to\mathcal{M}$, we call the collection $\{\Xi,\Omega,\mathcal{M},\Lambda\}$ a computational problem.
\end{definition}

The definition of a computational problem is deliberately general. The SCI of a computational problem is the smallest number of successive limits needed to compute the solution to the problem. We call a corresponding suitably indexed family of algorithms a `\textit{tower of algorithms}'. In addition, we will use finer notions of error control. For example, consider the case that $(\mathcal{M},d)$ is the space of non-empty compact subsets of $\mathbb{C}$, equipped with the Hausdorff metric. Then, the SCI hierarchy \cite{colbrook2020foundations,ben2015can} can be described as follows.

\vspace{1mm}

{\bf The SCI hierarchy:}
Given a collection $\mathcal{C}$ of computational problems,
\begin{itemize}
\item[(i)] $\Delta^{\alpha}_0 = \Pi^{\alpha}_0 = \Sigma^{\alpha}_0$ is the set of problems that can be computed in finite time (the SCI $=0$). In other words, there exists an algorithm $\Gamma$ such that $\Gamma(A)=\Xi(A)$ for all $A\in\Omega$.
\item[(ii)] $\Delta^{\alpha}_1$ is the set of problems that can be computed using one limit (the SCI $=1$) with control of the error, i.e., $\exists$ a sequence of algorithms $\{\Gamma_n\}$ such that $d(\Gamma_n(A), \Xi(A)) \leq 2^{-n}, \, \forall A \in \Omega$.
\item[(iii)] $\Sigma^{\alpha}_1$: We have $\Delta^{\alpha}_1 \subset \Sigma^{\alpha}_1 \subset \Delta^{\alpha}_2 $ and $\Sigma^{\alpha}_1$ is the set of problems for which there exists a sequence of algorithms $\{\Gamma_n\}$ such that for every $A \in \Omega$ we have $\Gamma_n(A) \rightarrow \Xi(A)$ as $n \rightarrow \infty$. Moreover, $\sup_{z\in\Gamma_n(A)}\mathrm{dist}(z,\Xi(A))\leq 2^{-n}$, where $\mathrm{dist}(x,S)$ denotes the Euclidean distance of $x$ to $S$.
\item[(iv)] $\Pi^{\alpha}_1$: We have $\Delta^{\alpha}_1 \subset \Pi^{\alpha}_1 \subset \Delta^{\alpha}_2 $ and $\Pi^{\alpha}_1$ is the set of problems for which there exists a sequence of algorithms $\{\Gamma_n\}$ such that for every $A \in \Omega$ we have $\Gamma_n(A) \rightarrow \Xi(A)$ as $n \rightarrow \infty$. Moreover, $\sup_{z\in\Xi(A)}\mathrm{dist}(z,\Gamma_n(A))\leq 2^{-n}$.
\item[(v)] $\Delta^{\alpha}_2$ is the set of problems that can be computed using one limit (SCI $=1$) without error control, i.e., $\exists$ a sequence of algorithms $\{\Gamma_n\}$ such that $\lim_{n\rightarrow \infty}\Gamma_n(A) = \Xi(A), \, \forall A \in \Omega$.
\item[(vi)] $\Delta^{\alpha}_{m+1}$, for $m \in \mathbb{N}$, is the set of problems that can be computed by using $m$ successive limits, (SCI $\leq m$), i.e., $\exists$ a family of algorithms $\{\Gamma_{n_m, \ldots, n_1}\}$ such that 
$$
\lim_{n_m \rightarrow\infty}\cdots \lim_{n_1\rightarrow\infty}\Gamma_{n_m,\ldots, n_1}(A) = \Xi(A), \quad\, \forall A \in \Omega.
$$
\item[(vii)] $\Sigma^{\alpha}_{m}$ is the set of problems in $\Delta^{\alpha}_{m+1}$ such that, letting $\Gamma_{n_m}(A)=\lim_{n_{m-1} \rightarrow\infty}\cdots \lim_{n_1\rightarrow\infty}\Gamma_{n_m,\ldots, n_1}(A)$, $\sup_{z\in \Gamma_{n_m}(A)}\mathrm{dist}(z,\Xi(A))\leq 2^{-n_m}$. In other words, computing the $m$th limit is a $\Sigma^{\alpha}_1$ problem.
\item[(viii)] $\Pi^{\alpha}_{m}$ is the set of problems in $\Delta^{\alpha}_{m+1}$ such that $\sup_{z\in \Xi(A)}\mathrm{dist}(z,\Gamma_{n_m}(A))\leq 2^{-n_m}$. In other words, computing the $m$th limit is a $\Pi^{\alpha}_1$ problem.
\end{itemize}
 
Schematically, the SCI hierarchy can be viewed in the following way:

\begin{equation}\label{SCI_hierarchy_hfhf}
\begin{tikzpicture}[baseline=(current  bounding  box.center)]
  \matrix (m) [matrix of math nodes,row sep=1.2em,column sep=1.5em] {
  \Pi_0^{\alpha}   &                    & \MA{\Pi_1^{\alpha}} &    &  \Pi_2^{\alpha}&  & {}\\
  \Delta_0^{\alpha}&  \Delta_1^{\alpha} & \Sigma_1^{\alpha}\cup\Pi_1^{\alpha} & \Delta_2^{\alpha}&      \Sigma_2^{\alpha}\cup\Pi_2^{\alpha} & \Delta_3^{\alpha}& \cdots\\
	\Sigma_0^{\alpha}&                    & \MA{\Sigma_1^{\alpha}} & &  \Sigma_2^{\alpha}&  &{} \\
  };
 \path[-stealth, auto] (m-1-1) edge[draw=none]
                                    node [sloped, auto=false,
                                     allow upside down] {$=$} (m-2-1)
																		(m-3-1) edge[draw=none]
                                    node [sloped, auto=false,
                                     allow upside down] {$=$} (m-2-1)
																		
																		(m-2-2) edge[draw=none]
                                    node [sloped, auto=false,
                                     allow upside down] {$\subsetneq$} (m-2-3)
																		(m-2-3) edge[draw=none]
                                    node [sloped, auto=false,
                                     allow upside down] {$\subsetneq$} (m-2-4)
																		(m-2-4) edge[draw=none]
                                    node [sloped, auto=false,
                                     allow upside down] {$\subsetneq$} (m-2-5)
																		(m-2-5) edge[draw=none]
                                    node [sloped, auto=false,
                                     allow upside down] {$\subsetneq$} (m-2-6)
																		(m-2-6) edge[draw=none]
                                    node [sloped, auto=false,
                                     allow upside down] {$\subsetneq$} (m-2-7)

												(m-2-1) edge[draw=none]
                                    node [sloped, auto=false,
                                     allow upside down] {$\subsetneq$} (m-2-2)
											 (m-2-2) edge[draw=none]
                                    node [sloped, auto=false,
                                     allow upside down] {$\subsetneq$} (m-1-3)
											(m-2-2) edge[draw=none]
                                    node [sloped, auto=false,
                                     allow upside down] {$\subsetneq$} (m-3-3)
											 (m-1-3) edge[draw=none]
                                    node [sloped, auto=false,
                                     allow upside down] {$\subsetneq$} (m-2-4)
																		(m-3-3) edge[draw=none]
                                    node [sloped, auto=false,
                                     allow upside down] {$\subsetneq$} (m-2-4)
																		(m-2-4) edge[draw=none]
                                    node [sloped, auto=false,
                                     allow upside down] {$\subsetneq$} (m-1-5)
											(m-2-4) edge[draw=none]
                                    node [sloped, auto=false,
                                     allow upside down] {$\subsetneq$} (m-3-5)
																		(m-1-5) edge[draw=none]
                                    node [sloped, auto=false,
                                     allow upside down] {$\subsetneq$} (m-2-6)
																		(m-3-5) edge[draw=none]
                                    node [sloped, auto=false,
                                     allow upside down] {$\subsetneq$} (m-2-6)
											(m-2-6) edge[draw=none]
                                    node [sloped, auto=false,
                                     allow upside down] {$\subsetneq$} (m-1-7)
																		(m-2-6) edge[draw=none]
                                    node [sloped, auto=false,
                                     allow upside down] {$\subsetneq$} (m-3-7);
																		
\end{tikzpicture}
\end{equation}
A visual demonstration of these classes is shown in Figure \ref{sigma_meaning}. For the description for decision problems, see \S \ref{SCI_Hierarchy}. The $\Sigma_1^{\alpha}$ and $\Pi_1^{\alpha}$ classes become crucial in computer-assisted proofs (see below).

\begin{figure}
\centering
\includegraphics[width=0.7\textwidth,trim={0mm 1mm 0mm 1mm},clip]{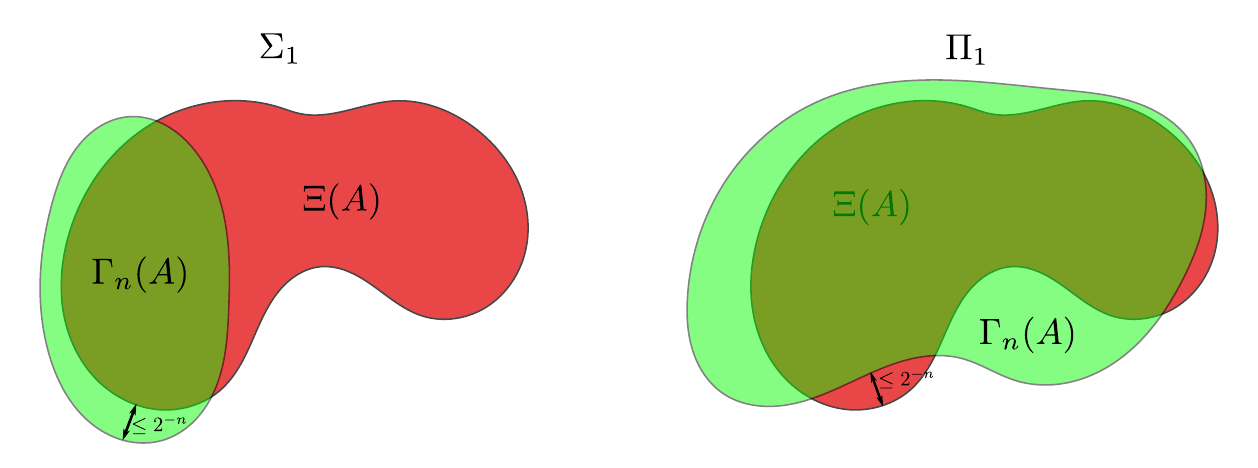}
\caption{Meaning of $\Sigma_1$ and $\Pi_1$ convergence for problem function $\Xi$ computed in the Hausdorff metric. The red areas represent $\Xi(A)$, whereas the green areas represent the output of the algorithm $\Gamma_n(A)$. $\Sigma_1$ convergence means convergence as $n\rightarrow\infty$ but each output point in $\Gamma_n(A)$ is at most distance $2^{-n}$ from $\Xi(A)$. Similarly, in the case of $\Pi_1$, we have convergence as $n\rightarrow\infty$ but any point in $\Xi(A)$ is at most distance $2^{-n}$ from $\Gamma_n(A)$.}
\label{sigma_meaning}
\end{figure}

\begin{remark}[Computability, not complexity]
It is important to note that (despite its name) the SCI hierarchy is a hierarchy for classifying computability, not complexity. Most computational spectral problems of interest are $\notin\Delta_1$ in the SCI hierarchy, and complexity theory only makes sense for problems in $\Delta_1$. Hence, it is impossible to build a complexity theory for most infinite-dimensional spectral problems. The scientific community computes with non-computable problems ($\notin\Delta_1$) on a daily basis (e.g., in quantum mechanics). This also happens in high-profile computer-assisted proofs (see below). \hfill$\boxtimes$
\end{remark}

\subsubsection{The model of computation $\alpha$}
The $\alpha$ in the superscript indicates the model of computation, which is described in \S \ref{SCI_Hierarchy}. For $\alpha = G$, the underlying algorithm is general (see Definition \ref{Gen_alg}) and can use any tools at its disposal. The reader may think of a Blum--Shub--Smale (BSS) machine or a Turing machine with access to any oracle, although a general algorithm is even more powerful. However, for $\alpha = A$ this means that only arithmetic operations and comparisons are allowed. In particular, if rational inputs are considered, the algorithm is a Turing machine, and in the case of real inputs, a BSS machine. Hence, a result of the form
$$
\notin \Delta_k^G \text{ is stronger than } \notin \Delta_k^A.
$$  
Indeed, a $\notin \Delta_k^G$ result is universal and holds for any model of computation. Moreover, 
$$
\in \Delta_k^A \text{ is stronger than } \in \Delta_k^G,
$$  
and similarly for the $\Pi_k$ and $\Sigma_k$ classes. In this paper we prove lower bounds for $\alpha = G$ and upper bounds for $\alpha = A$, thus obtaining the strongest results. Remark \ref{fdjojosr} discusses further how the model of computation is of less importance in infinite dimensions.

\subsubsection{Computer-assisted proofs}

The class of problems $\Delta_1^A$ are precisely those that are computable according to Turing's definition of computability (i.e., there exists an algorithm such that for any $\epsilon > 0$ the algorithm can produce an $\epsilon$-accurate output). However, most infinite-dimensional spectral problems are $\notin \Delta_1^A.$ The simplest example is the problem of computing spectra of infinite diagonal matrices. Very few interesting infinite-dimensional spectral problems are actually in $\Delta_1^A$, and most of the literature on spectral computations provides algorithms that yield $\Delta_2^A$ classification results. Such algorithms converge, but may not provide error control. In many cases, error control is impossible.

Problems not in $\Delta_1^A$ are a daily occurrence in the sciences due to suggestive numerical simulations or evidence based on experiments. However, in the field of computer-assisted proofs, this is not possible, since only $100\%$ rigour is accepted. Nevertheless, there are many examples of famous conjectures that have been proven using computational problems that do not lie in $\Delta_1^A$. For example, the proof of Kepler's conjecture (Hilbert's 18th problem) \cite{Hales_Annals, hales_Pi} relies on decision problems that are not in $\Delta_1^A$ \cite{opt_big}. Another example is C. Fefferman and L. Seco's proof of the Dirac--Schwinger conjecture on the asymptotics of ground states of certain Schr\"odinger operators \cite{fefferman1990, fefferman1992, fefferman1993aperiodicity,  fefferman1994, fefferman1994_2, fefferman1995, fefferman1996interval, fefferman1996, fefferman1997}. The reason for this apparent paradox is that the $\Sigma^A_1$ and $\Pi^A_1$ classes are larger than $\Delta^A_1$, but can still be used in computer-assisted proofs. Both of the above examples implicitly prove $\Sigma^A_1$ classifications. For example, suppose we have a computational spectral problem that lies in $\Sigma^A_1$. This means that there is an algorithm that will converge and never provide incorrect output, up to a user-specified error bound. Thus, conjectures about operators never having spectra in a certain area (a common problem in many problems of stability analysis, for example) could be disproved by a computer-assisted proof. Recent results using computer-assisted proofs in spectral theory include \cite{brown2010,bogli2014guaranteed}.

\subsection{Evaluation sets and domains}

Throughout this paper, unless otherwise specified, $A$ will be a bounded operator acting on the canonical Hilbert space $l^2(\mathbb{N})$ (we define $\Omega_{\mathrm{B}}:=\mathcal{B}(l^2(\mathbb{N}))$), and realised as a matrix with respect to the canonical basis. However, the results of this paper extend to general separable Hilbert spaces $\mathcal{H}$ through a choice of orthonormal basis $e_1,e_2,\ldots$ if one can compute the matrix values of the operators with respect to this basis (see the discussion of the evaluation sets below). For example, we can treat operators naturally defined on lattices such as $\mathbb{Z}^d$, or more generally on graphs. Such operators are abundant in mathematical physics. Below we give the evaluation sets and classes of operators treated in this paper. For convenience, this information is summarised in Tables \ref{definition_rev_table} and \ref{definition_rev_table2}.

\begin{table}
\begin{center}
\addtolength{\leftskip} {-1.6cm}
    \addtolength{\rightskip}{-2cm}
\hspace{-0.5cm}\begin{tabular}{|l|l|l|}
\hline
 \rule{0pt}{1.1em}\textbf{Evaluation set}&\textbf{Information available to algorithm}&\textbf{Meaning}\\
 \hline
 \hline
\rule{0pt}{1.1em}$\Lambda_1$&$f^1_{i,j}: A\mapsto \langle Ae_j,e_i\rangle$ for $i,j\in\mathbb{N}$& matrix entries of $A$\\
\hline
\rule{0pt}{1.1em}$\Lambda_2$&\makecell[l]{$f^1_{i,j}: A\mapsto \langle Ae_j,e_i\rangle$ for $i,j\in\mathbb{N}$\\
$f^2_{i,j}: A\mapsto \langle Ae_j,Ae_i\rangle$ for $i,j\in\mathbb{N}$\\
$f^3_{i,j}: A\mapsto \langle A^*e_j,A^*e_i\rangle$ for $i,j\in\mathbb{N}$}&matrix entries of $A$, $A^*A$, and $AA^*$\\
\hline
\end{tabular}
\vspace{3mm}
\caption{Summary of evaluation sets used in this paper.}
\vspace{-3mm}
\label{definition_rev_table}
\end{center}
\end{table}

\begin{table}
\begin{center}
\addtolength{\leftskip} {-1.6cm}
    \addtolength{\rightskip}{-2cm}

\hspace{-0.5cm}\begin{tabular}{|l|l|}
\hline
 \rule{0pt}{1.1em}\textbf{Domain (class of operators)}&\textbf{Meaning}\\
 \hline
 \hline

\rule{0pt}{1.1em}$\Omega_{\mathrm{B}}$&bounded operators\\
\hline

\rule{0pt}{1.1em}$\Omega_\mathrm{N}$&bounded normal operators\\
\hline

\rule{0pt}{1.1em}$\Omega_{\mathrm{SA}}$&bounded self-adjoint operators\\
\hline

\rule{0pt}{1.1em}$\Omega_\mathrm{D}$&bounded self-adjoint diagonal operators\\
\hline

\rule{0pt}{1.1em}$\Omega_f$&bounded operators with known $f$ satisfying \eqref{bd_disp2}\\
\hline

\rule{0pt}{1.1em}$\Omega_g$&bounded operators with known $g$ satisfying \eqref{control_resolvent}\\
\hline

\end{tabular}

\vspace{3mm}
\caption{Summary of classes of operators treated in this paper. Note that all considered classes lie in $\Omega_\mathrm{B}$, and that $\Omega_\mathrm{D}\subset\Omega_\mathrm{SA}\subset\Omega_\mathrm{N}$, $\Omega_\mathrm{D}\subset \Omega_f$, and $\Omega_\mathrm{N}\subset\Omega_{g:x\mapsto x}$.}
\vspace{-3mm}
\label{definition_rev_table2}
\end{center}
\end{table}

\textbf{Evaluation sets:} We consider two natural sets of information that our algorithms can read. The first, $\Lambda_1$, provides the entries of the matrix representation of $A$ with respect to the canonical basis $\{e_i\}_{i\in \mathbb{N}}$:
$$
\Lambda_1=\{f^1_{i,j}: A\mapsto \langle Ae_j,e_i\rangle| i,j\in\mathbb{N}\}.
$$
The second, $\Lambda_2$, appends $\Lambda_1$ with the entries of the matrix representations of $A^*A$ and $AA^*$ with respect to the canonical basis $\{e_i\}_{i\in \mathbb{N}}$:
$$
\Lambda_2=\Lambda_1\cup\{f^2_{i,j}: A\mapsto \langle Ae_j,Ae_i\rangle, f^3_{i,j}: A\mapsto \langle A^*e_j,A^*e_i\rangle| i,j\in\mathbb{N}\}.
$$
We include $\Lambda_2$ since it is natural for problems posed in variational form, and can often be evaluated through numerical integration. When considering classes with functions $f$ (and $\{c_n\}$) and $g$ as in \eqref{bd_disp2} and \eqref{control_resolvent} below, we will add these to the relevant evaluation set (evaluating $g$ at rational points) and with an abuse of notation still use the notation $\Lambda_i$. A small selection of the problems also require additional information, such as when testing if a set intersects a spectral set, but any changes to $\Lambda_i$ will be pointed out where appropriate. 

\textbf{Classes of operators:} Let $\Omega_\mathrm{N}$ denote the class of normal operators in $\Omega_{\mathrm{B}}$, $\Omega_{\mathrm{SA}}$ denote the class of self-adjoint operators in $\Omega_\mathrm{N}$, and $\Omega_\mathrm{D}$ denote the class of self-adjoint diagonal operators in $\Omega_{\mathrm{SA}}$. For $f:\mathbb{N}\rightarrow\mathbb{N}$, $f(n)\geq n+1$ define
\begin{equation}
D_{f,n}(A) :=\max\left\{{\left\|(I-P_{f(n)})AP_n\right\|},{\left\|P_nA(I-P_{f(n)})\right\|}\right\},
\label{bd_disp2}
\end{equation}
where $P_m$ is the orthogonal projection onto $\mathrm{span}\{e_1,\ldots, e_m\}$. Given such an $f$, we assume access to an estimate $D_{f,n}(A) \leq c_n(A)\in\mathbb{Q}_{\geq0}$, where $c_n \rightarrow 0$ as $n \rightarrow \infty$. We let $\Omega_f$ denote the class of bounded operators with known function $f$ and $\{c_n\}$.\footnote{Sometimes the sequence $\{c_n\}$ is not needed and we will explicitly mention when this is the case.} As a special case, if we know our matrix is sparse with finitely many non-zero entries in each column and row (and we know the positions of the non-zero entries) then we know an $f$ with $c_n=0$. Let $g:\mathbb{R}_{+}\rightarrow\mathbb{R}_{+}$ be a strictly increasing, continuous function that vanishes only at $0$ with $\lim_{x\rightarrow\infty}g(x)=\infty$. Let $\Omega_g$ be the class of bounded operators with
\begin{equation}
\label{control_resolvent}
{g(\mathrm{dist}(z,\mathrm{Sp}(A)))}\leq\left\|R(z,A)\right\|^{-1},
\end{equation}
for $z\in\mathbb{C}$, where $R(z,A)=(A-zI)^{-1}$. A simple compactness argument shows that such a $g$ always exists for any given $A\in\Omega_{\mathrm{B}}$. However, the classification of spectral problems in the SCI hierarchy generally depends on whether one knows an estimate for $g$ or not. For example, in the self-adjoint and normal cases, $g(x) = x$ is the trivial choice of $g$. Operators with $g(x) = x$ are known as $G_1$ and include the well studied class of hyponormal operators (operators with $A^*A-AA^*\geq0$) \cite{putnam1979operators}. A common assumption is that
$$
\|R(z,A)\|\leq \frac{C}{\mathrm{dist}(z,\mathrm{Sp}(A))} \quad \forall z\notin\mathrm{Sp}(A),
$$
for some constant $C$, which is equivalent to $A\in\Omega_g$ with $g(x)=x/C$. For example, if $A$ is similar to a normal operator with a similarity transformation $S$ that has bounded condition number $\kappa(S)$, we can take $C=\kappa(S)$. Other examples with non-linear $g$ include perturbations of self-adjoint operators \cite[e.g., Theorem 7.7.1]{gil2003operator}. More generally, one can view the function $g$ as a measure of stability of the spectrum of $A$ through the formula
\begin{equation}
\mathrm{Sp}_{\epsilon}(A):=\mathrm{Sp}(A)\cup\{z\notin\mathrm{Sp}(A):\left\|R(z,A)\right\|\geq1/\epsilon\}=\bigcup_{B\in\Omega_{\mathrm{B}},\left\|B\right\|\leq\epsilon}\mathrm{Sp}(A+B),
\end{equation}
where $\mathrm{Sp}_{\epsilon}(A)$ denotes the ($\epsilon$-)pseudospectrum of $A$ \cite{trefethen2005spectra}. The function $g$ is held fixed for a given class $\Omega_g$ and a smaller $g$ leads to a larger class of operators $\Omega_g$.

\section{Main Results: The Foundations of Computing Geometric Features of Spectra}
\label{main_results_section}

Our results classify computing geometric features of spectra in the SCI hierarchy. In other words, we are concerned with the foundations of computation for geometric features of spectra. There are two aspects of this classification: proving impossibility results (lower bounds), where we make use of the tools developed in \S \ref{bigHth} and Theorem \ref{DST_main}, and proving upper bounds through the construction of algorithms. This ensures that our algorithms realise the boundary of what computers can achieve in spectral computations. We have included routines for some of the main algorithms in Appendix \ref{compute_rout_appendix} and computational examples in \S \ref{num_test}.

\begin{remark}[Bounding the operator norm] The proofs of lower bounds make clear that all classifications still hold if we replace the respective sub-class $\Omega\subset\Omega_{\mathrm{B}}$ by the restriction to operators in $\Omega$ having operator norm at most $M\in\mathbb{R}_{>0}$, adding such a value $M$ (constant function) to the evaluation set $\Lambda$.\hfill$\boxtimes$
\end{remark}

\begin{remark}[Computing the resolvent norm]
Some of the algorithms are built on the local approximation of the functions (or similar functions) defined by
$$
\gamma_{n}(z;A)=\min\{\sigma_{\mathrm{inf}}((A-zI){|_{P_n\mathcal{H}}}),\sigma_{\mathrm{inf}}((A^*-\bar{z}I){|_{P_n\mathcal{H}}})\},
$$
where $\sigma_{\mathrm{inf}}$ denotes the smallest singular value or injection modulus:
$$
\sigma_{\mathrm{inf}}(T)=\inf\{\|Tv\|:\|v\|=1\}.
$$
The functions $\gamma_{n}$ converge to the resolvent norm $\|R(z,A)\|^{-1}$ uniformly on compact subsets of $\mathbb{C}$ from above as $n\rightarrow\infty$. This idea was crucial in \cite{colb1,colbrook3} to compute spectra with $\Sigma_1^A$ error control for a large class of operators. A theme of some of our proofs, especially those concerning Lebesgue measure and fractal dimensions, is the extension of these ideas to compute geometric properties of the spectrum.\hfill$\boxtimes$
\end{remark}

\subsection{Spectral radii}
\label{sec:rev:spec_rad}

We begin with a very simple geometric feature of the spectrum. The spectral radius, $r(A)$, of a bounded operator $A$ is the supremum of the absolute values of members of the spectrum, which is attained. Spectral radii commonly appear in applications involving stability analysis. We set $\Xi_r(A):=r(A)$ and make the following initial observations:
\begin{itemize}
	\item[(i)] One can easily show that the computational problem of the operator norm of any $A\in\Omega_{\mathrm{B}}$ lies in $\Sigma_1^A$. Hence, since $r(A)\leq \left\|A\right\|$, we can easily get an upper bound for $\Xi_r(A)$ in one limit. Of course, if $A$ is not normal, this upper bound may not agree with $\Xi_r(A)$.
	\item[(ii)] If an operator lies in $\Omega_g$ with $g(x)=x$, then the convex hull of the spectrum is equal to the closure of the numerical range (recall that the numerical range is $\{\langle Ax,x\rangle:\|x\|=1\}$) \cite{orland1964class}. Such operators are known as convexoid and the problem of computing $\Xi_r(A)$ for such operators lies in $\Sigma_1^A$.
	\item[(iii)] In light of Gelfand's famous formula $\Xi_r(A)=\lim_{n\rightarrow\infty}\|A^n\|^{\frac{1}{n}}$, one might expect that the computation of $\Xi_r(A)$ is strictly easier in the sense of the SCI hierarchy than that of the spectrum.
\end{itemize}

The following shows that the intuition in (iii) is misguided in general, and only occurs if an operator is convexoid as in (ii). Computing $\Xi_r(A)$ is just as hard as computing the spectrum for the class $\Omega_{\mathrm{B}}$. Controlling the resolvent via a function $g$ as in \eqref{control_resolvent} makes the problem easier in the sense of SCI hierarchy than the general class $\Omega_{\mathrm{B}}$, but is not sufficient to reduce the SCI of the problem to $1$.

\begin{theorem}
\label{spec_rad_thm}
Let $g:\mathbb{R}_{+}\rightarrow\mathbb{R}_{+}$ be a strictly increasing, continuous function that vanishes only at $0$ with $\lim_{x\rightarrow\infty}g(x)=\infty$. In addition, suppose that $g(x)\leq (1-\delta)x$ for some $\delta\in(0,1)$. Then:
\begin{align*}
&\Delta^G_1 \not\owns  \{\Xi_r,\Omega_\mathrm{D},\Lambda_1\} \in \Sigma^A_1,\quad&&\Delta^G_1 \not\owns  \{\Xi_r,\Omega_\mathrm{N},\Lambda_1\} \in \Sigma^A_1,\quad&&\Delta^G_1 \not\owns  \{\Xi_r,\Omega_f\cap\Omega_g,\Lambda_1\} \in \Sigma^A_1,\\
&\Delta^G_2 \not\owns  \{\Xi_r,\Omega_g,\Lambda_1\} \in \Sigma^A_2,\quad&&\Delta^G_2 \not\owns  \{\Xi_r,\Omega_f,\Lambda_1\} \in \Pi^A_2,\quad &&\Delta^G_3 \not\owns \{\Xi_r,\Omega_{\mathrm{B}},\Lambda_1\} \in \Pi^A_3.
\end{align*}
When considering the evaluation set $\Lambda_2$, the only changes are the following classifications:
\begin{align*}
&\Delta^G_1 \not\owns  \{\Xi_r,\Omega_g,\Lambda_2\} \in \Sigma^A_1,\quad&&\Delta^G_2 \not\owns \{\Xi_r,\Omega_{\mathrm{B}},\Lambda_2\} \in \Pi^A_2.
\end{align*} 
\end{theorem}

\begin{remark}
The $\Pi_2^A$ algorithm for $\{\Xi_r,\Omega_f\}$ does not need a null sequence $\{c_n\}$ bounding the dispersion, $D_{f,n}(A)\leq c_n$, to be sharp in the SCI hierarchy since this is absorbed in the first limit.\hfill$\boxtimes$
\end{remark}

\begin{remark}
The proofs of the lower bounds in Theorem \ref{spec_rad_thm} for $\Omega_g$ require $g$ with the stated additional property and $\delta>0$. In particular, the lower bound does not cover the smaller class of $G_1$ operators.\hfill$\boxtimes$
\end{remark}

\subsection{Essential spectral radii}
\label{sec:rev:ess_spec_rad}

Next, we consider the essential spectral radius. Define the essential spectrum of $A\in\Omega_{\mathrm{B}}$ as
$$
\mathrm{Sp}_{\mathrm{ess}}(A)=\bigcap_{B\in\Omega_K}\mathrm{Sp}(A+B),
$$
where $\Omega_K$ denotes the class of compact operators. The essential spectral radius, $\Xi_{er}(A)$, is simply the supremum of the absolute values over $\mathrm{Sp}_{\mathrm{ess}}(A)$.

\begin{theorem}
\label{spec_rad_thm2}
We have the following classifications for $i=1,2$:
\begin{align*}
&\Delta^G_2 \not\owns  \{\Xi_{er},\Omega_\mathrm{D},\Lambda_i\} \in \Pi^A_2,\quad&&\Delta^G_2 \not\owns  \{\Xi_{er},\Omega_\mathrm{N},\Lambda_i\} \in \Pi^A_2,\quad&&\Delta^G_2 \not\owns  \{\Xi_{er},\Omega_f,\Lambda_i\} \in \Pi^A_2.
\end{align*}
Whereas, for general operators,
\begin{align*}
\Delta^G_3 \not\owns \{\Xi_{er},\Omega_{\mathrm{B}},\Lambda_1\} \in \Pi^A_3,\quad \Delta^G_2 \not\owns \{\Xi_{er},\Omega_{\mathrm{B}},\Lambda_2\} \in \Pi^A_2.
\end{align*}
\end{theorem}

\subsection{Capacity and polynomial operator norms}
\label{sec:rev:CAP}

Given a polynomial $p$ of degree at least two\footnote{We fix the polynomial $p$ for the strongest possible negative results. However, the existence of the towers of algorithms also holds when considering the polynomial $p$ itself as an input.}, we consider the problem of computing $\Xi_{r,p}=\|p(A)\|$ and the capacity of the spectrum defined by
$$
\Xi_{cap}(A)=\inf_{\text{monic polynomial }p}\|p(A)\|^{\frac{1}{\mathrm{deg}(p)}}=\lim_{d\rightarrow\infty}\inf\left\{\|p(A)\|^{\frac{1}{d}}:\text{monic polynomial }p,\mathrm{deg}(p)=d\right\}.
$$
A theorem of Halmos shows that this definition of capacity agrees with the usual potential-theoretic definition of capacity of the set $\mathrm{Sp}(A)$ \cite{halmos1971capacity}. Roughly speaking, the capacity measures the ability of $\mathrm{Sp}(A)$ to hold electrical charge. We will also see some other measures of size in \S \ref{Leb_sec} and \S \ref{frac_dims_sec}. The capacity of the spectrum is of particular interest in Krylov methods where, for instance, it is related to the speed of convergence\footnote{This is an idealisation since the capacity studies operator norms while true Krylov processes look at $p(A)x$ with one or several vectors $x$. However, from local spectral theory (e.g., \cite{muller1992local}) it follows that, generically, the asymptotic speeds are the same.} \cite{nevanlinna2012convergence,nevanlinna1995hessenberg,muller2007spectral,nevanlinna1990linear,miekkala1996iterative}. The capacity is also an important object in local spectral theory \cite{aiena2007fredholm,laursen2000introduction,muller2007spectral}, and related work \cite{nevanlinna2009computing,burke2006characterizations} includes methods for computing the polynomially convex hull of an operator. The following theorem provides the relevant SCI classifications.

\begin{theorem}
\label{spec_rad_thm3}
We have the following classifications for $i=1,2$ and $\hat\Omega=\Omega_\mathrm{D}$ or $\Omega_f$:
\begin{align*}
&\Delta^G_1 \not\owns  \{\Xi_{r,p},\hat\Omega,\Lambda_i\} \in \Sigma^A_1,\quad&&\Delta^G_2 \not\owns  \{\Xi_{cap},\hat\Omega,\Lambda_i\} \in \Pi^A_2.
\end{align*}
For $\tilde\Omega=\Omega_\mathrm{N},\Omega_g$ or $\Omega_{\mathrm{B}}$,
\begin{align*}
&\Delta^G_2 \not\owns  \{\Xi_{r,p},\tilde\Omega,\Lambda_1\} \in \Sigma^A_2,\quad&&\Delta^G_3 \not\owns  \{\Xi_{cap},\tilde\Omega,\Lambda_1\} \in \Pi^A_3\\
&\Delta^G_1 \not\owns  \{\Xi_{r,p},\tilde\Omega,\Lambda_2\} \in \Sigma^A_1,\quad&&\Delta^G_2 \not\owns  \{\Xi_{cap},\tilde\Omega,\Lambda_2\} \in \Pi^A_2.
\end{align*}
\end{theorem}

The proof shows these problems have the same classifications for $\Omega_\mathrm{SA}$ as $\Omega_\mathrm{N}$. Somewhat surprising is the result that the computation of $\|p(A)\|$ requires two successive limits for self-adjoint operators. The proof shows that one reason for this is spectral pollution associated with finite section methods.

\subsection{Essential numerical range, gaps in essential spectra and detecting failure of finite section}
\label{fs_fails}

We now consider geometric features of spectra that are related to the finite section method, the most intensely studied computational method of approximating spectra \cite{Albrecht_Fields, Bottcher_pseu,Bottcher_book, bottcher2006analysis}.\footnote{W. Arveson  \cite{Arveson_cnum_lin94, Arveson_noncommute93,Arveson_role_of94,Arveson_Improper93,  Arveson_discrete91} and N. Brown \cite{brown2007quasi, Brown_2006, Brown_Memoars} pioneered spectral computations from the point of view of $C^*$-algebras, both for the general spectral computation problem as well as for Schr\"{o}dinger operators. This combination can be traced back to the work of A. B{\"o}ttcher \& B. Silberman \cite{Albrecht1983}. Arveson also considered spectral computation in terms of densities, which is related to Szeg\"{o}'s work \cite{Szego} on finite section approximations.} The basic form of the finite section method approximates the spectrum of $A$ by $\mathrm{Sp}(P_nA|_{P_n\mathcal{H}})$, where $\{P_m\}$ is a sequence of finite-dimensional projections converging strongly to the identity as $m\rightarrow\infty$. The computation is often done with finite element, finite difference or spectral methods by discretising the operator on a suitable finite-dimensional space \cite{rappaz1997spectral,boffi2000problem,Boffi2,Annalisa2,zhao2007spurious,Snorre1,klaus1981point,lewin2014spurious}. Even when $A$ is self-adjoint, spurious eigenvalues, that have nothing to do with $\mathrm{Sp}(A)$, can accumulate anywhere within gaps of the essential spectrum as $n\rightarrow\infty$.\footnote{Even when the finite section method converges, it typically only yields $\Delta_2^A$ classifications in the SCI hierarchy \cite{Arieh2,bottcher_grudsky_iserles_2011,Arieh_IMA, brunner2011}.} This is known as \textit{spectral pollution}. More precisely, the essential numerical range of $A\in\Omega_{\mathrm{B}}$ is defined as
\begin{equation}
\label{ess_num_RR_deff}
W_e(A)=\bigcap_{B\in\Omega_K}\overline{W(A+B)},
\end{equation}
where $W(A)=\{\langle Ax,x\rangle:\|x\|=1\}$ is the usual numerical range.\footnote{If $A$ is hyponormal, then $W_e(A)$ is the convex hull of the essential spectrum \cite{salinas1972operators}.} We recall the following two theorems.

\begin{theorem}[Pokrzywa \cite{Pokrzywa_79}]\label{polish}
Let $A \in \mathcal{B}(\mathcal{H})$ and let $\{P_n\}$ be a sequence of finite-dimensional projections converging strongly to the identity. Suppose that $S \subset W_e(A).$ Then there exists a sequence
$\{Q_n\}$ of finite-dimensional projections such that $P_n < Q_n$ (so $Q_n \rightarrow I$ strongly) and 
\[
d_{\mathrm{H}}(\mathrm{Sp}(P_nA |_{P_n\mathcal{H}}) \cup S, \mathrm{Sp}(Q_nA |_{Q_n\mathcal{H}})) \rightarrow 0,
 \quad \text{as }n \rightarrow \infty, 
\]
where $d_{\mathrm{H}}$ denotes the Hausdorff distance.
\end{theorem}

\begin{theorem}[Pokrzywa \cite{Pokrzywa_79}]\label{t:conv}
Let $A \in \mathcal{B}(\mathcal{H})$ and let $\{P_n\}$ be a sequence of finite-dimensional projections converging strongly to the identity. If $\lambda \notin W_e(A)$ then $\lambda \in\mathrm{Sp}(A)$ if and only if 
$$
\mathrm{dist}(\lambda,\mathrm{Sp}( P_nA|_{P_n\mathcal{H}})) \rightarrow 0,
\quad\text{as } n \rightarrow \infty.
$$ 
\end{theorem}

Theorems \ref{polish} and \ref{t:conv} show that spectral pollution is confined to the essential numerical range and can be arbitrarily bad in $W_e(A)\backslash \mathrm{Sp}(A)$.\footnote{In the non-normal case it is possible for finite sections to not capture all of the spectrum - parts of the spectrum may be unattainable. This is distinct from spectral pollution. Theorem \ref{polish} says that, up to a different choice of projections, this can be avoided on $W_e(A)$.} For self-adjoint operators, the gaps in the essential spectrum correspond exactly to $W_e(A)\backslash \mathrm{Sp}(A)$. As a result, there has been considerable attention towards methods that detect gaps in essential spectra and eigenvalues within these gaps \cite{carmona2012spectral,lewin2014spurious,boffi2000problem,Shargorodsky1}, as well as studying the precise nature of spectral pollution \cite{rappaz1997spectral,lewin2010spectral,Marletta_pollution,Marletta_Spec_gaps}.

A consequence of the main result of this section, Theorem \ref{spec_poll_hard}, is that detecting these gaps is strictly harder in the sense of the SCI hierarchy than computing the spectrum for self-adjoint operators (which was classified in \cite{colb1,colbrook3,ben2015can}). We define the problem function $\Xi_{we}(A)=W_e(A)$. For a given non-empty open set $U$ in $\mathbb{F}$ (with $\mathbb{F}$ being $\mathbb{C}$ or $\mathbb{R}$), let $\Xi_{poll}^{\mathbb{F}}$ be the decision problem
$$
\Xi_{poll}^{\mathbb{F}}(A,U)=\begin{cases}
1,\quad\text{ if }\overline{U}\cap(W_e(A)\backslash \mathrm{Sp}(A))\neq\emptyset\\
0,\quad\text{ otherwise.}
\end{cases}
$$
$\Xi_{poll}^{\mathbb{F}}$ decides whether spectral pollution can occur on the {closed} set $\overline{U}$. For the self-adjoint case and $\mathbb{F}=\mathbb{R}$, this is equivalent to asking whether there exists a point in the {open} set $U$ that also lies in a gap of the essential spectrum. To incorporate $U$ into $\Lambda_i$, we allow access to a countable number of open balls $\{U_m\}_{m\in\mathbb{N}}$ whose union is $U$. If $\mathbb{F}=\mathbb{R}$, then each $U_m$ is of the form $(a_m,b_m)$ with $a_m,b_m\in\mathbb{Q}\cup\{\pm\infty\}$. If $\mathbb{F}=\mathbb{C}$, then each $U_m$ is equal to $D_{r_m}(z_m)$ (the open ball of radius $r_m$ centred at $z_m$) with $r_m\in\mathbb{Q_+}\cup\{\infty\}$ and $z_m\in\mathbb{Q}+i\mathbb{Q}$. We add pointwise evaluations of the relevant sequences $\{(a_m,b_m)\}$ or $\{(r_m,z_m)\}$ to $\Lambda_i$.

\begin{theorem}[Computation of essential numerical range and whether spectral pollution can occur on a set]
\label{spec_poll_hard}
Let $\Omega=\Omega_\mathrm{N},\Omega_\mathrm{SA}$ or $\Omega_{\mathrm{B}}$ and let $i=1,2$. Then
$$
\Delta^G_2 \not\owns  \{\Xi_{we},\Omega,\Lambda_i\} \in \Pi^A_2.
$$
Furthermore, for $i=1,2$ the following classifications hold, valid also if we restrict to the case $U=U_1$ or to $U=U_1=\mathbb{F}$:
\begin{align*}
\Delta^G_3 \not\owns \{\Xi_{poll}^{\mathbb{R}},\Omega_\mathrm{SA},\Lambda_i\} \in \Sigma^A_3, \quad &&\Delta^G_3 \not\owns \{\Xi_{poll}^{\mathbb{C}},\Omega_{\mathrm{B}},\Lambda_i\} \in \Sigma^A_3.
\end{align*}
\end{theorem}

\begin{remark}[Computing spectra is easier than algorithmically determining if spectral pollution can occur on a set]
One can show that $\{\mathrm{Sp}(\cdot),\Omega_\mathrm{SA},\Lambda_1\}\in\Sigma_2^A$ and $\{\mathrm{Sp}(\cdot),\Omega_\mathrm{SA},\Lambda_2\}\in\Sigma_1^A$. Hence determining $\Xi_{poll}^\mathbb{R}$ is strictly harder than the spectral computational problem and requires two additional successive limits if $\Lambda=\Lambda_2$. Even in the general case, $\{\mathrm{Sp}(\cdot),\Omega_{\mathrm{B}},\Lambda_2\}\in\Pi_2^A$ and hence the spectral problem is strictly easier in the sense of SCI hierarchy. The proofs also make clear that we get the same classification of $\Xi_{poll}^{\mathbb{F}}$ for other classes such as $\Omega_\mathrm{N}$, $\Omega_g$ etc.\hfill$\boxtimes$
\end{remark}

\begin{remark}[Unbounded operators]
In \S \ref{append_UB_ENR}, we show that computing the essential numerical range for closed unbounded operators $T$ on $l^2(\mathbb{N})$ (under the condition that the linear span of the canonical basis forms a core of $T$) also lies in $\Pi_2^A$. The definition of the essential numerical range for such operators was recently given in \cite{bogli2020essential}. This paper showed that $W_e(T)$ consists precisely of the essential spectrum of $T$ together with all possible spectral pollution that may arise by applying projection methods to approximate the spectrum of $T$, thus generalising Theorems \ref{polish} and \ref{t:conv}. A computational example is given in \S \ref{froja}.\hfill$\boxtimes$
\end{remark}

\subsection{Lebesgue measure of spectra}
\label{Leb_sec}

A basic property of the set $\mathrm{Sp}(A)$, also connected to physical applications, is its Lebesgue measure. Well-studied operators such as the almost Mathieu operator at critical coupling \cite{avila2006} or the Fibonacci Hamiltonian \cite{sutHo1989singular} have spectra with Lebesgue measure zero. Following \cite{aubry1980analyticity}, there have been many further numerical studies \cite{thouless1983bandwidths,thouless1990scaling,thouless1991total}. For further examples of operators with numerical approximations of the Lebesgue measure, see the references in \cite{avila2017spectral,benza1991band,sire1989electronic}. Numerical studies typically look at periodic approximates \cite{puelz2015spectral}, and computing the Lebesgue measure of periodic approximates of tridiagonal operators lies in $\Delta_1^A$. The tools we develop are more general and do not assume such structure. Verification of our algorithms for the almost Mathieu operator is presented in \S \ref{jhnwagtwqrfg}.

The Lebesgue measure on $\mathbb{C}$ will be denoted by $\mathrm{Leb}$. When considering classes of self-adjoint operators, we use the Lebesgue measure on $\mathbb{R}$ denoted by $\mathrm{Leb}_{\mathbb{R}}$. We also define
$$
\widehat{\mathrm{Sp}}_{\epsilon}(A)=\{z\in\mathbb{C}:\|R(z,A)\|^{-1}<\epsilon\},
$$
whose closure is $\mathrm{Sp}_\epsilon(A)$. For a class $\Omega\subset\Omega_{\mathrm{B}}$, there are three questions we answer in this section:
\begin{enumerate}
	\item Given $A\in\Omega$, can we compute $\mathrm{Leb}(\mathrm{Sp}(A))$?
	\item Given $A\in\Omega$ and $\epsilon>0$, can we compute $\mathrm{Leb}(\widehat{\mathrm{Sp}}_{\epsilon}(A))$?\footnote{We consider the computation of $\mathrm{Leb}(\widehat{\mathrm{Sp}}_{\epsilon}(A))$ instead of $\mathrm{Leb}({\mathrm{Sp}}_{\epsilon}(A))$ since it is not clear that the level sets
\begin{equation}
\label{level_sets}
S_{\epsilon}(A):=\{z\in\mathbb{C}:\left\|R(z,A)\right\|^{-1}=\epsilon\}
\end{equation}
always have Lebesgue measure zero (this is currently an open problem for general bounded operators). This situation is analogous to the case of approximating the pseudospectra of bounded operators, where one uses the crucial property that pseudospectra cannot jump - the resolvent norm cannot be constant on open subsets of $\mathbb{C}\backslash\mathrm{Sp}(A)$ for a bounded operator $A$ acting on a separable Hilbert space \cite{shargorodsky2008level}. The question of whether the sets in \eqref{level_sets} are Lebesgue null is the measure theoretic equivalent. Note, however, that it is straightforward to show that $S_{\epsilon}(A)$ is null for $A\in\Omega_{\mathrm{N}}$ through the formula $\|R(z,A)\|^{-1}=\mathrm{dist}(z,\mathrm{Sp}(A))$.}
	\item Given $A\in\Omega$, can we determine whether $\mathrm{Leb}(\mathrm{Sp}(A))=0$?
\end{enumerate}

For the first two questions, we consider the metric space $([0,\infty),d)$ with the Euclidean metric. For question three we consider the discrete metric on $\{0,1\}$, where $1$ is interpreted as ``Yes'', and $0$ as ``No''. We denote the above problem functions by $\Xi_1^{L},\Xi_2^{L}$ and $\Xi_3^{L}$, respectively. In analogy to computing spectra and pseudospectra, $\Xi_2^L$ is the easiest to compute and can be done in one limit for a large class of operators. It also follows from the dominated convergence theorem that
\begin{equation}
\label{eps_down_leb}
\lim_{\epsilon\downarrow 0}\mathrm{Leb}(\widehat{\mathrm{Sp}}_{\epsilon}(A))=\mathrm{Leb}(\mathrm{Sp}(A)).
\end{equation}

\begin{theorem}[Lebesgue measure of spectra]
\label{Leb_1}
Given the above set-up, we have the following classifications
\begin{equation*}
\Delta^G_2 \not\owns \{\Xi_1^L,\Omega_f,\Lambda_i\} \in \Pi^A_2, \quad \Delta^G_2 \not\owns  \{\Xi_1^L,\Omega_\mathrm{D},\Lambda_i\} \in \Pi^A_2 \quad i=1,2,  
\end{equation*}
and for $\Omega=\Omega_{\mathrm{B}},\Omega_\mathrm{SA}$, $\Omega_\mathrm{N}$ or $\Omega_g$,
$$
\Delta^G_3 \not\owns \{\Xi_1^L,\Omega,\Lambda_1 \} \in \Pi^A_3,\quad \Delta^G_2 \not\owns \{\Xi_1^L,\Omega,\Lambda_2 \} \in \Pi^A_2.
$$
\end{theorem}

The algorithm constructed in the proof of Theorem \ref{Leb_1} is local, and can be adapted to find the Lebesgue measure of $\mathrm{Sp}(A)$ intersected with any compact interval or cube in one or two dimensions, respectively. Moreover, when considering $\Omega_f$, we do not need the sequence $\{c_n\}$, and the algorithm can be restricted to $\mathbb{R}$, where it converges to $\mathrm{Leb}_{\mathbb{R}}(\mathrm{Sp}(A)\cap\mathbb{R})$. Our results also hold when considering bounded diagonal operators (dropping the restriction of self-adjointness) and using $\mathrm{Leb}$ instead of $\mathrm{Leb}_{\mathbb{R}}$.

We now turn to the SCI classification of $\mathrm{Leb}(\widehat{\mathrm{Sp}}_{\epsilon}(A))$, which is useful since it provides a route to computing $\mathrm{Leb}(\mathrm{Sp}(A))$ for any $A\in\Omega_{\mathrm{B}}$ via \eqref{eps_down_leb}. This is a similar state of affairs to the computation of the spectrum itself - one can approximate the spectrum via pseudospectra. 

\begin{theorem}[Lebesgue measure of pseudospectra]
\label{Leb_2}
Given the above set-up, we have the following classifications
\begin{equation*}
\Delta^G_1 \not\owns \{\Xi_2^L,\Omega_f,\Lambda_i\} \in \Sigma^A_1, \quad \Delta^G_1 \not\owns  \{\Xi_2^L,\Omega_\mathrm{D},\Lambda_i\} \in \Sigma^A_1 \quad i=1,2,
\end{equation*}
and for $\Omega=\Omega_{\mathrm{B}},\Omega_\mathrm{SA}$, $\Omega_\mathrm{N}$ or $\Omega_g$,
$$
\Delta^G_2 \not\owns \{\Xi_2^L,\Omega,\Lambda_1 \} \in \Sigma^A_2,\quad \Delta^G_1 \not\owns \{\Xi_2^L,\Omega,\Lambda_2 \} \in \Sigma^A_1.
$$
\end{theorem}

Why is $\Xi_2^L$ easier to compute than $\Xi_1^L$? Heuristically, the pseudospectrum is less refined than the spectrum, making the measure easier to approximate. Another viewpoint is the continuity points of the maps $\Xi_1^L$ and $\Xi_2^L$. For simplicity, consider these maps restricted to $\Omega_\mathrm{D}$ and equip these diagonal operators with the operator norm topology. The following shows that $\Xi_2^L$ is more stable than $\Xi_1^L$, explaining why it is easier to approximate. Again, this is the same state of affairs as comparing $\mathrm{Sp}(A)$ and $\mathrm{Sp}_{\epsilon}(A)$ as sets.
\begin{proposition}
\label{leb_cty}
In the above set-up, the following hold:
\begin{enumerate}
	\item $\Xi_1^L$ is continuous at $A\in\Omega_\mathrm{D}$ if and only if $\mathrm{Leb}_{\mathbb{R}}(\mathrm{Sp}(A))=0$.
	\item $\Xi_2^L$ is continuous at all $A\in\Omega_\mathrm{D}$.
\end{enumerate}
\end{proposition}

Finally, when computing $\Xi_3^L$, we let $(\mathcal{M},d)$ be the set $\{0,1\}$ endowed with the discrete topology and consider the problem function
$$
\Xi_3^L(A)=
\begin{cases}
0,\quad\text{ if }\mathrm{Leb}(\mathrm{Sp}(A))>0\\
1,\quad\text{ otherwise.}
\end{cases}
$$
It is straightforward to build a family of algorithms that converge in three successive limits for this problem using the algorithm constructed in Theorem \ref{Leb_1} and its monotonicity. The next theorem shows that this is optimal, even for the set of diagonal self-adjoint bounded operators. This demonstrates how hard it is to solve decision problems about the spectrum with finite amounts of information, particularly when the problems involve an object that ignores countable sets, such as the Lebesgue measure.

\begin{theorem}[Is the spectrum Lebesgue null?]
\label{Leb_3}
Given the above set-up, we have the following classifications
$$
\Delta^G_3\not\owns\{\Xi_3^L,\Omega_f,\Lambda_i\}\in\Pi^A_3,\quad \Delta^G_3\not\owns\{\Xi_3^L,\Omega_\mathrm{D},\Lambda_i\}\in\Pi^A_3, \quad i=1,2,
$$
and for $\Omega=\Omega_{\mathrm{B}},\Omega_\mathrm{SA}$, $\Omega_\mathrm{N}$ or $\Omega_g$,
$$
\Delta^G_4\not\owns\{\Xi_3^L,\Omega,\Lambda_1\}\in\Pi^A_4, \quad \Delta^G_3\not\owns\{\Xi_3^L,\Omega,\Lambda_2\}\in\Pi^A_3.
$$
\end{theorem}

\begin{remark}
These are the first examples of computational spectral problems that require four successive limits to compute in the SCI hierarchy. To prove this, we need some tools from descriptive set theory in \S \ref{bigHth}. Note that we prove the lower bounds for \textit{general} algorithms, so regardless of the model of computation.\hfill$\boxtimes$
\end{remark}

\subsection{Fractal dimensions of spectra}
\label{frac_dims_sec}

When considering operators from physical models, such as Schr\"odinger operators in quantum mechanics, fractal dimensions of spectra are related to important physical phenomena, such as the spreading of an initially localised wavepacket \cite{killip2003dynamical}. Further applications and numerical studies have already been discussed in \S \ref{intro}. However, estimating the fractal dimension is extremely difficult. This can be explained by the SCI hierarchy - the $\mathrm{SCI}>1$, even for computing the box-counting dimension, the most basic definition of fractal dimension. The Hausdorff dimension is even worse and has $\text{SCI}\geq 3$. In this section, we exclusively treat self-adjoint operators and hence seek fractal dimensions of $\mathrm{Sp}(A)\subset\mathbb{R}$.\footnote{The proofs for general self-adjoint operators can be adapted with an additional successive limit and the use of two-dimensional covering boxes to treat the class of general bounded operators. Some care is needed to deal with the boundaries of covering boxes for the Hausdorff dimension, but we omit the details.}

\textit{\textbf{Box-Counting Dimension:}} Let $F$ be a bounded set in $\mathbb{R}$ and let $N_{\delta}(F)$ be the number of closed intervals of length $\delta>0$ required to cover $F$. We define the upper and lower box-counting dimensions as
$$
\overline{\mathrm{dim}}_B(F)=\limsup_{\delta\downarrow{}0}\frac{\log(N_{\delta}(F))}{\log(1/\delta)},\quad
\underline{\mathrm{dim}}_B(F)=\liminf_{\delta\downarrow{}0}\frac{\log(N_{\delta}(F))}{\log(1/\delta)}.
$$
When $\overline{\mathrm{dim}}_B(F)=\underline{\mathrm{dim}}_B(F)$, we can replace the $\liminf$ and $\limsup$ by $\lim$, and the common value is the box-counting dimension $\mathrm{dim}_B(F)$, an example of a fractal dimension. A possible drawback of the box-counting dimension is its lack of countable stability. For example, $\mathrm{dim}_B(\{0,1,1/2,1/3,\ldots\})=1/2$. Let $\Omega_{f}^{BD}$ be the class of self-adjoint operators in $\Omega_f$ (see \eqref{bd_disp2}) whose upper and lower box-counting dimensions of the spectrum agree. Let $\Omega_{\mathrm{SA}}^{BD}$ be the class of self-adjoint operators whose upper and lower box-counting dimensions of the spectrum agree, and denote by $\Omega_{\mathrm{D}}^{BD}$ the class of diagonal operators in $\Omega_{\mathrm{SA}}^{BD}$.

{\textit{\textbf{Hausdorff Dimension:}}} A more complicated, yet robust notion of fractal dimension is related to the Hausdorff measure \cite{falconer2004fractal,mattila1999geometry}. Let $F\subset\mathbb{R}^n$ be a bounded Borel set and let $\mathcal{C}_{\delta}(F)$ denote the class of (countable) $\delta$-covers\footnote{That is, the set of covers $\{U_i\}_{i\in I}$ with $I$ at most countable and with $\mathrm{diam}(U_i)\leq\delta$.} of $F$. One first defines the quantities (for $d\geq0$)
$$
\mathcal{H}^{d}_{\delta}(F)=\inf\left\{\sum_i\mathrm{diam}(U_i)^d:\{U_i\}\in\mathcal{C}_{\delta}(F)\right\},\quad\mathcal{H}^{d}(F)=\lim_{\delta\downarrow0}\mathcal{H}^{d}_{\delta}(F).
$$
There is a unique $d'=\mathrm{dim}_{H}(F)\geq0$, the Hausdorff dimension of $F$, such that $\mathcal{H}^{d}(F)=0$ for $d>d'$ and $\mathcal{H}^{d}(F)=\infty$ for $d<d'$. One can prove that 
$$
\mathrm{dim}_{H}(F)\leq\underline{\mathrm{dim}}_B(F)\leq\overline{\mathrm{dim}}_B(F).
$$

With these definitions in hand, we can now present the main theorem of this section.

\begin{theorem}[Fractal dimensions of spectra]
\label{fractal_theoremhhhjh}
Let $\Xi_B(A)=\mathrm{dim}_B(\mathrm{Sp}(A))$ and $\Xi_H=\mathrm{dim}_H(\mathrm{Sp}(A))$. Then for $i=1,2$,
\begin{align*}
&\Delta^G_2\not\owns\{\Xi_B,\Omega_{f}^{BD},\Lambda_i\}\in\Pi^A_2, \quad && \Delta^G_2\not\owns\{\Xi_B,\Omega_{\mathrm{D}}^{BD},\Lambda_i\}\in\Pi^A_2\\
&\Delta^G_3\not\owns\{\Xi_H,\Omega_f\cap\Omega_\mathrm{SA},\Lambda_i\}\in\Sigma^A_3, \quad && \Delta^G_3\not\owns\{\Xi_H,\Omega_\mathrm{D},\Lambda_i\}\in\Sigma^A_3,
\end{align*}
whereas
\begin{align*}
&\Delta^G_3\not\owns\{\Xi_B,\Omega_{\mathrm{SA}}^{BD},\Lambda_1\}\in\Pi^A_3,\quad &&\Delta^G_2\not\owns\{\Xi_B,\Omega_{\mathrm{SA}}^{BD},\Lambda_2\}\in\Pi^A_2\\
&\Delta^G_4\not\owns\{\Xi_H,\Omega_\mathrm{SA},\Lambda_1\}\in\Sigma^A_4,\quad &&\Delta^G_3\not\owns\{\Xi_H,\Omega_\mathrm{SA},\Lambda_2\}\in\Sigma^A_3.
\end{align*}
\end{theorem}

\begin{remark}[When $\underline{\mathrm{dim}}_B(\mathrm{Sp}(A))\neq \overline{\mathrm{dim}}_B(\mathrm{Sp}(A))$]The algorithms for $\Xi_B$ also converge without the assumption that the upper and lower box-counting dimensions of $\mathrm{Sp}(A)$ agree, to a quantity $\Gamma(A)$ with
$$
\underline{\mathrm{dim}}_B(\mathrm{Sp}(A))\leq \Gamma(A)\leq\overline{\mathrm{dim}}_B(\mathrm{Sp}(A)).
$$
One of the properties that makes the Hausdorff dimension harder to compute than the box-counting dimension is its countable stability, meaning that if $F$ is countable then $\mathrm{dim}_H(F)=0$.\hfill$\boxtimes$
\end{remark}

\begin{remark}
Some of our results have interpretations for real bounded sequences. Given such a sequence $\{a_i\}_{i\in\mathbb{N}}\subset\mathbb{R}$, we can ask the same questions about $\overline{\{a_1,a_2,\ldots\}}$ as we have asked about the spectrum. We can embed these problems as spectral problems for the class $\Omega_\mathrm{D}$ of bounded self-adjoint diagonal operators by simply considering diagonal operators with entries $\{a_1,a_2,\ldots\}$. Theorems \ref{Leb_1}, \ref{Leb_3} and \ref{fractal_theoremhhhjh} immediately then give the classifications. With regards to fractal dimensions, the key problem is to try and relate the amount of data that has been seen to the resolution obtained from the data (as highlighted in the computational example below). Once we have the framework of the SCI, we can immediately see why the problem is so difficult - the computational problem requires three successive limits for the Hausdorff dimension.\hfill$\boxtimes$
\end{remark}

Finally, the following lemma is used in the construction of the tower of algorithms for computing the Hausdorff dimension but is interesting in its own right so is listed here.

\begin{lemma}
\label{halt_test}
Let $(a,b)\subset\mathbb{R}$ be a finite open interval and let $A\in\Omega_f\cap\Omega_\mathrm{SA}$. Then determining whether
$
\mathrm{Sp}(A)\cap(a,b)\neq\emptyset
$
using $\Lambda_i$ is a problem with $\mathrm{SCI}_A=1$. Furthermore, we can design an algorithm that halts if and only the answer is ``Yes'', that is, the problem lies in $\Sigma^A_1$. Similarly the problem lies in $\Sigma_2^A$ when considering $\Omega_\mathrm{SA}$ with $\Lambda_1$ (or $\Sigma_1^A$ when we allow access to $\Lambda_2$).
\end{lemma}

\section{Computational Examples}
\label{num_test}

In this section, we demonstrate that the SCI-sharp algorithms constructed in this paper can be efficiently implemented for large-scale computations. Moreover, the algorithms have desirable convergence properties, converging monotonically or being eventually constant, as captured by the $\Sigma/\Pi$ classification. Generically, this monotonicity holds in all of the successive limits, and not just the final limit; many of the towers of algorithms undergo \textit{oscillation phenomena} where each subsequent limit is monotone but in the opposite sense/direction than the limit beforehand. We can take advantage of this when analysing the algorithms numerically. The algorithms also highlight suitable information that lowers the SCI classification to $\Sigma_1/\Pi_1$. Other advantages of the algorithms based on approximating the resolvent norm include locality, numerical stability and speed/parallelisation. In the examples that follow, we remind the reader what each parameter $n_k$ intuitively does in the relevant algorithm and simplified routines for many of the algorithms can be found in Appendix \ref{compute_rout_appendix}. Finally, we point the reader to Remark \ref{fdjojosr} - all of the algorithms can be implemented rigorously using arithmetic operations over the rationals or with methods such as interval arithmetic.

\subsection{Spectral radius}

We begin with the spectral radius and consider the upper-triangular non-normal operator on $l^2(\mathbb{Z})$ defined by its action on the canonical basis via
$$
Ae_j=e_{j-2}+i^{j}e_{j-1}.
$$
In this case, the operator norm of $A$ is $2$ and the approximation of the spectrum by finite section is $\{0\}$. Hence, to compute the spectral radius, one must resort to the techniques used in our algorithms based on rectangular truncations. Recall that the SCI classification for computing the spectral radius of such operators (where the dispersion is known) is $\Pi_2^A$ (see Theorem \ref{spec_rad_thm} for further classifications). The first parameter, $n_1$, controls the size of the rectangular truncation\footnote{For this example and other operators on $l^2(\mathbb{Z})$ below, we reorder the basis so that the operator $A$ acts on $l^2(\mathbb{N})$.} (as well as the grid resolution), whereas the second, $n_2$, controls the resolvent norm cut-off ($\epsilon=1/n_2$).

Figure \ref{SR_ex} (left) shows the output of $\Gamma_{n_2,n_1}(A)$ for computing the spectral radius. We see the expected monotonicity; $\Gamma_{n_2,n_1}(A)$ is increasing in $n_1$ but decreasing in $n_2$. It appears that $\lim_{n_1\rightarrow\infty}\Gamma_{10^2,n_1}(A)\approx\lim_{n_1\rightarrow\infty}\Gamma_{10^3,n_1}(A)\approx 1.4149$. The fact that these two values for different $n_2$ are similar suggests that we have reached convergence. Though, of course, the proof that the problem does not lie in $\Delta_2^G$ shows that we can never apply a choice of subsequences to gain convergence in one limit over the whole class $\Omega_f$. Nevertheless, the approximate value of $1.4149$ is confirmed in Figure \ref{SR_ex} (right) where we have shown pseudospectra, computed using the algorithm in \cite{colb1}.

\begin{figure}
\centering
\includegraphics[width=0.49\textwidth,trim={0mm 0mm 0mm 0mm},clip]{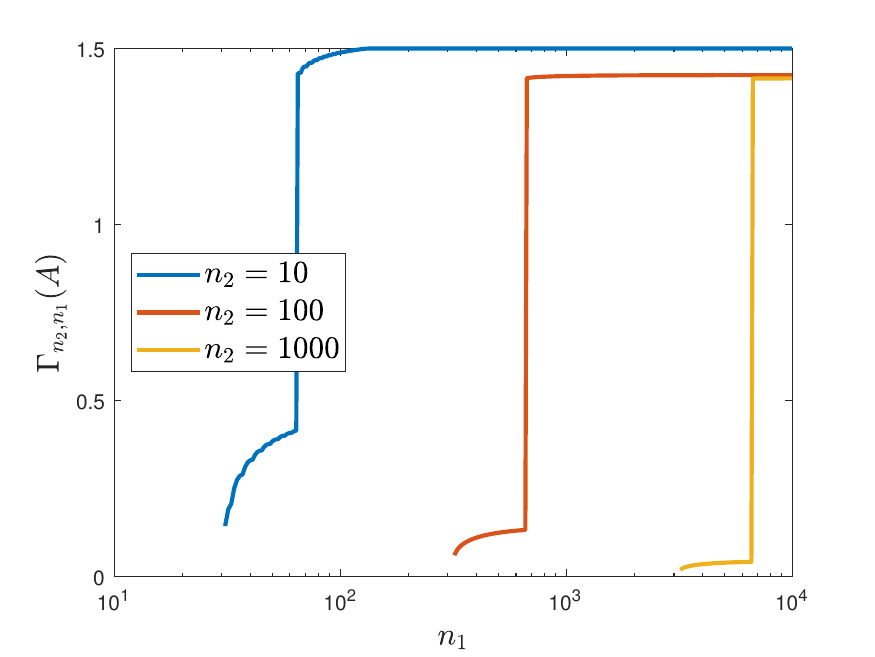}
\includegraphics[width=0.49\textwidth,trim={0mm 0mm 0mm 0mm},clip]{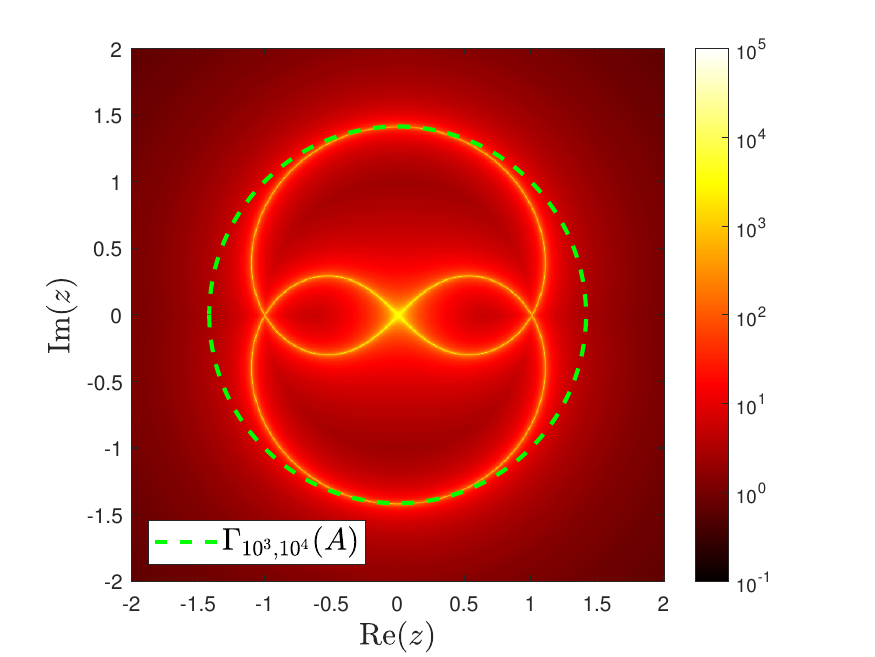}
\caption{Left: Output of the algorithm for computing the spectral radius. Right: Pseudospectrum computed using the method of \cite{colb1} (the colour scale corresponds to the resolvent norm $\|(A-zI)^{-1}\|$) which provides error control. We have show the output of $\Gamma_{10^3,10^4}(A)$ via the green dashed circle.}
\label{SR_ex}
\end{figure}

\subsection{Essential numerical range}
\label{froja}

\begin{figure}
\centering
\includegraphics[width=0.49\textwidth,trim={0mm 0mm 0mm 0mm},clip]{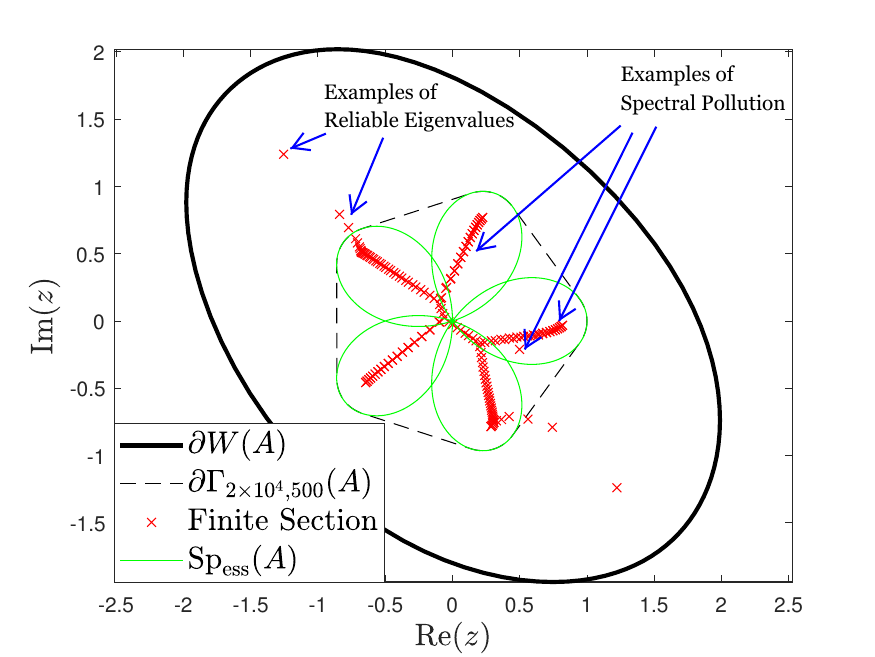}
\includegraphics[width=0.49\textwidth,trim={0mm 0mm 0mm 0mm},clip]{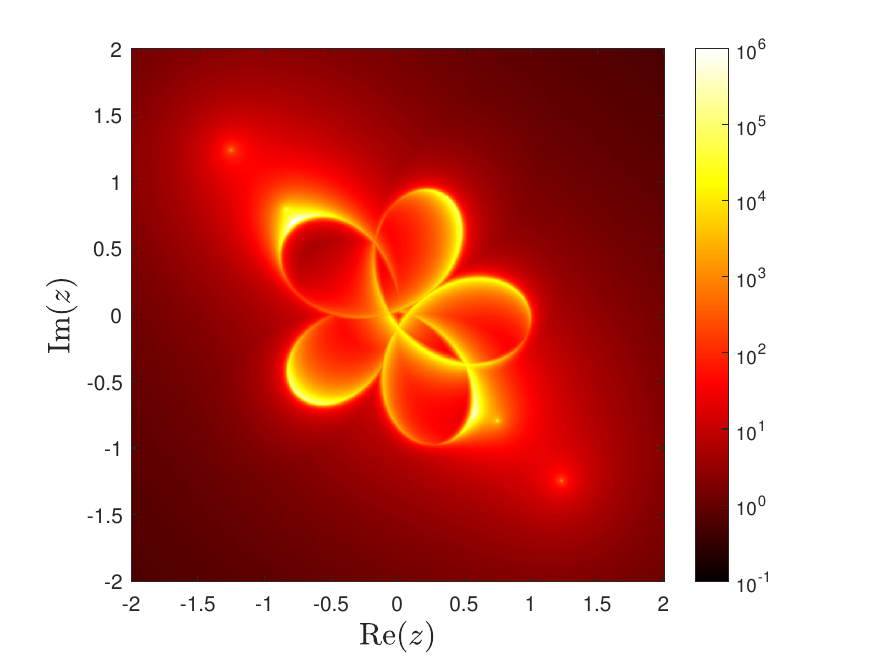}
\caption{Left: The boundaries of $\partial W(A)$ and $\partial\Gamma_{2\times 10^4,500}(A)$. We have also shown the essential spectrum of $A$ (whose convex hull, in this example, corresponds to $W_e(A)$) and the output of finite section for a $200\times 200$ truncation. Right: Pseudospectrum computed using the method of \cite{colb1} (the colour scale corresponds to the resolvent norm $\|(A-zI)^{-1}\|$) which provides error control. This confirms that eigenvalues, computed using finite section, outside $\partial\Gamma_{2\times 10^4,500}(A)$ are accurate and, in this example, indicates that the other eigenvalues correspond to spectral pollution.}
\label{NR_ex}
\end{figure}

\begin{figure}
\centering
\includegraphics[width=1\textwidth,trim={22mm 0mm 15mm 0mm},clip]{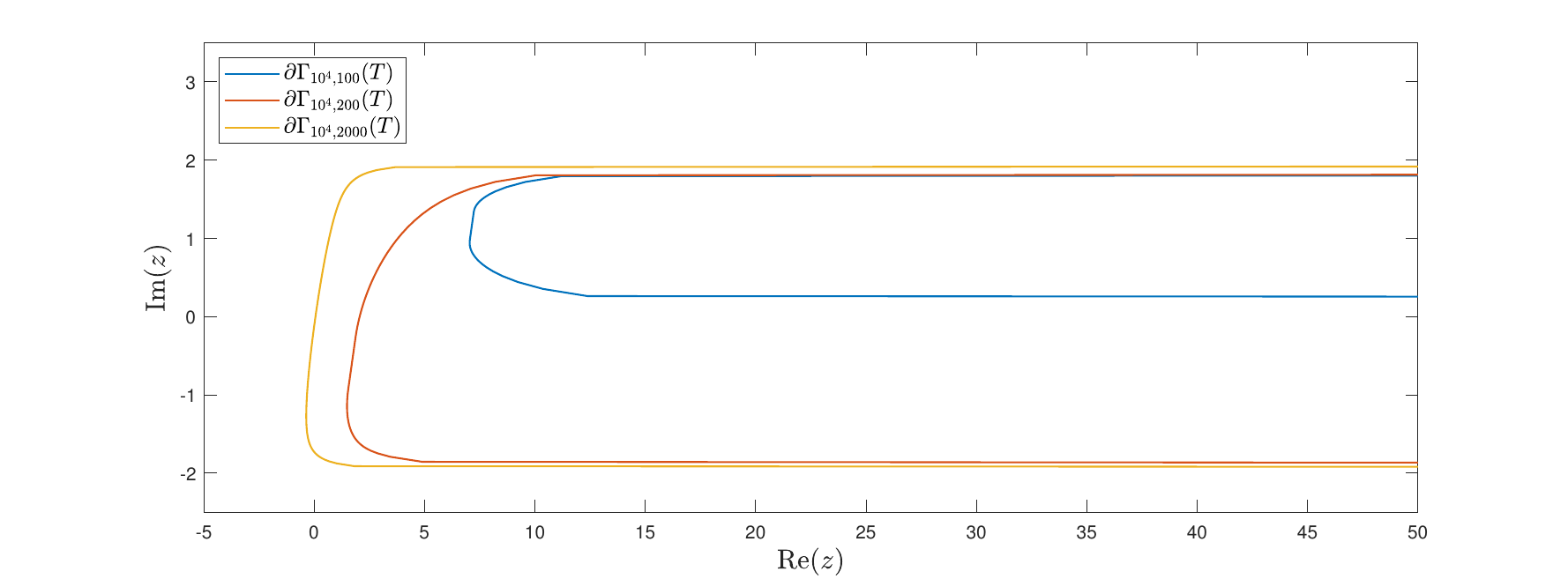}
\caption{The output of the algorithm for computing the essential numerical range of closed operators, applied to the complex Schr\"odinger operator $T$ in \eqref{complex_schd}.}
\label{NR_ex2}
\end{figure}

To demonstrate the algorithm for computing the essential numerical range, we first consider the Laurent operator $A_0$ acting on $l^2(\mathbb{Z})$ with the symbol
$$
a(t)=\frac{t^4+t^{-1}}{2}.
$$
In this case, $\mathrm{Sp}(A_0)=\mathrm{Sp}_{\mathrm{ess}}(A_0)=\{a(z):|z|=1\}$. We consider the operator $A=A_0+E$ where the compact perturbation $E$ is given by
$$
Ee_j=-\frac{3i}{1+|j|}e_{j-1}.
$$
Recall that the SCI classification for computing the essential numerical range is $\Pi_2^A$ (see Theorem \ref{spec_poll_hard}). The first parameter, $n_1$, controls the size of the truncation, whereas the second, $n_2$, controls how far along the matrix the truncations $(I-P_{n_2})P_{n_1+n_2}A|_{P_{n_1+n_2}(I-P_{n_2})\mathcal{H}}$ are taken with respect to the canonical basis.

Figure \ref{NR_ex} (left) shows the output of the algorithm $\Gamma_{n_2,n_1}(A)$ to compute the essential numerical range for $n_2=20000$ and $n_1=500$. We  show the boundary $\partial\Gamma_{n_2,n_1}(A)$ since the essential numerical range is convex. In this example, $W_\mathrm{e}(A)$ is the convex hull of $\mathrm{Sp}_{\mathrm{ess}}(A_0)$, which allows us to verify the output of the algorithm. We also show 200 eigenvalues of finite section (computed using extended precision to avoid numerical instabilities associated with non-normal truncations), the majority of which are due to truncation and provide an example of spectral pollution. This is confirmed when we compare to the pseudospectrum, also shown in Figure \ref{NR_ex} (right), computed using the algorithm in \cite{colb1}. However, eigenvalues outside $W_\mathrm{e}(A)$ correspond to true eigenvalues of $A$ (see Theorem \ref{t:conv}).

The algorithm can also be extended to unbounded operators, as outlined in \S \ref{append_UB_ENR}. For example, we consider the complex Schr\"odinger operator
\begin{equation}
\label{complex_schd}
T=-\frac{d^2}{dx^2}+(2i+1)\cos(x).
\end{equation}
By using a Gabor basis, we can represent $T$ as a closed operator on $l^2(\mathbb{N})$ such that the linear span of the canonical basis (corresponding to the Gabor basis) forms a core. This allows us to use Corollary \ref{gojoknb}, where we can compute the matrix elements (corresponding to inner products with the basis functions) with error control using quadrature. Figure \ref{NR_ex2} shows the output for $n_2=10^4$ and various $n_1$. We see the expected monotonicity as $n_1$ increases and the output for $n_1=2000$ has converged to visible accuracy in the plot.

\subsection{Capacity}

\begin{figure}
\centering
\includegraphics[width=0.49\textwidth,trim={35mm 92mm 35mm 92mm},clip]{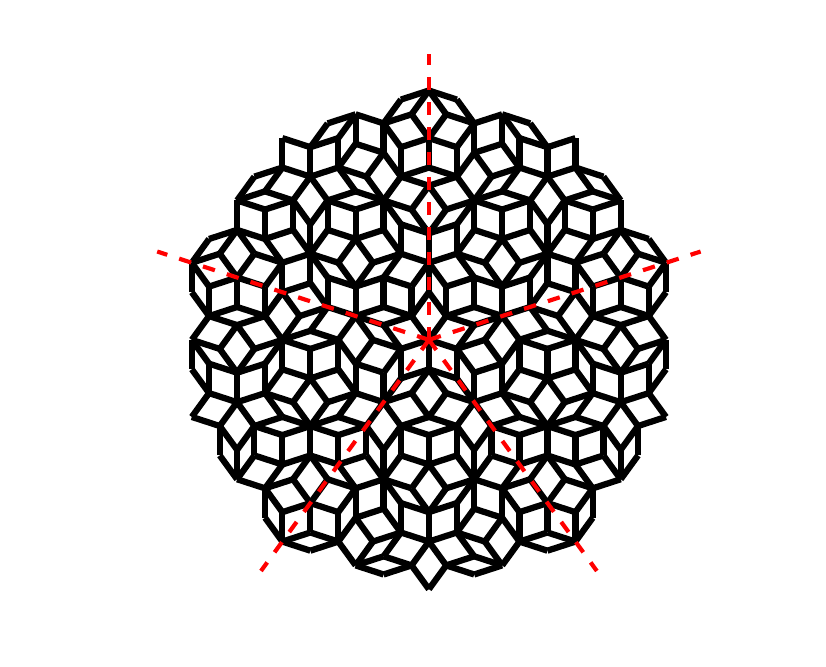}
\includegraphics[width=0.49\textwidth,trim={0mm 0mm 0mm 0mm},clip]{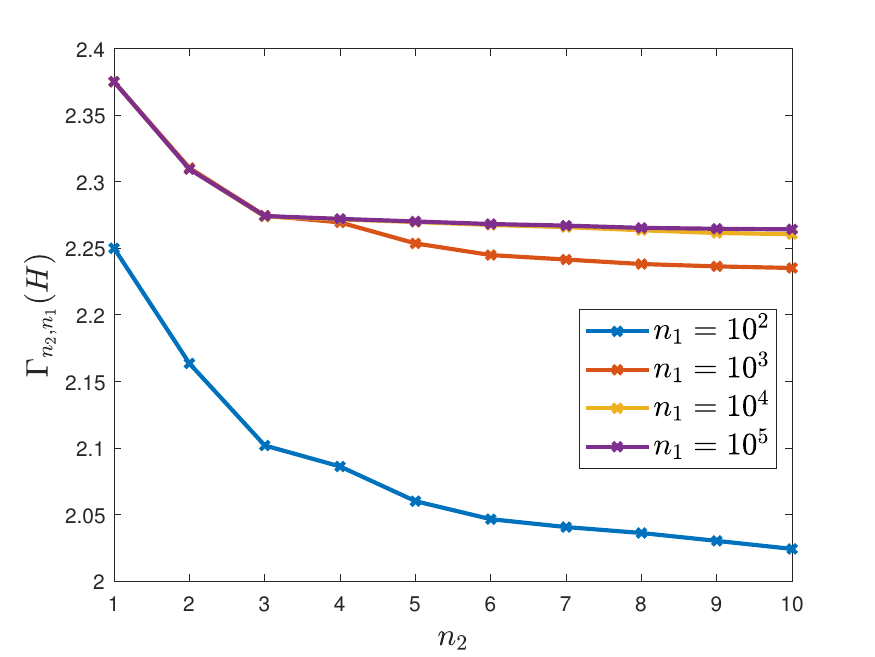}
\caption{Left: Finite portion of the Penrose tiling showing a fivefold rotational symmetry. Right: Output of the algorithm for computing the capacity of $\mathrm{Sp}(H)$, where $H$ is the operator in \eqref{H_0}.}
\label{CAP_ex}
\end{figure}

We now consider a transport Hamiltonian on a Penrose tile for which few analytical results are known. Quasicrystals were discovered in 1982 by Shechtman \cite{PhysRevLett.53.1951} who was awarded the Nobel prize in 2011 for his discovery. Over the past 30 years, there has been considerable interest in their often exotic properties \cite{stadnik2012physical,damanik2012spectral}. The Penrose tile is the standard two-dimensional model \cite{della2005band,vardeny2013optics}, and a finite portion of the tiling is shown in Figure \ref{CAP_ex} (left). However, unlike one-dimensional models, very little is known about the spectral properties of two-dimensional quasicrystals. Let $G$ be the graph consisting of the vertices, $V(G)$, of the Penrose tiling and $E(G)$ the set of edges. If there is an edge connecting two vertices $x$ and $y$, we write $x \sim y$. The (negative) Laplacian, $H$, acts on $\psi \in l^2(V(G)) \cong l^{2}(\mathbb{N})$ by 
\begin{equation}\label{H_0}
(H\psi)(x) = \sum_{y\sim x} \left(\psi(y)-\psi(x)\right).
\end{equation}
By choosing a suitable ordering of the vertices, we can represent $H$ as an operator acting on $l^2(\mathbb{N})$ of bounded dispersion with $f(n)-n\sim\mathcal{O}(\sqrt{n})$. Recall that the SCI classification for computing the capacity of the spectrum of such operators is $\Pi_2^A$ (see Theorem \ref{spec_rad_thm3} for further classifications). The first parameter, $n_1$, controls the size of the truncation used to test if intervals intersect the spectrum via Lemma \ref{halt_test}, whereas the second, $n_2$, controls the spacings of the interval coverings (which have width $2^{-n_2}$). In this example, we used the conformal mapping method of \cite{liesen2017fast} to accurately and rapidly compute the capacity of finite unions of intervals in $\mathbb{R}$ (see also Remark \ref{capacity_efficiency}).

Figure \ref{CAP_ex} (right) shows the output of $\Gamma_{n_2,n_1}(H)$ and we see the expected monotonicity; the output is increasing in $n_1$ but decreasing in $n_2$. By comparing the outputs for $n_1=10^4$ and $n_1=10^5$, it appears we have convergence up to around $n_2=8$. This suggests an upper bound (since the output is non-increasing in $n_2$) of approximately $2.26$ for the capacity of $\mathrm{Sp}(H)$ ($\mathrm{Sp}(H)$ is shown in Figure \ref{AM_leb}).

\subsection{Lebesgue measure}
\label{jhnwagtwqrfg}

\begin{figure}\centering
\begin{minipage}{.49\textwidth}
\centering
\includegraphics[width=1\textwidth,trim={0mm 0mm 00mm 0mm},clip]{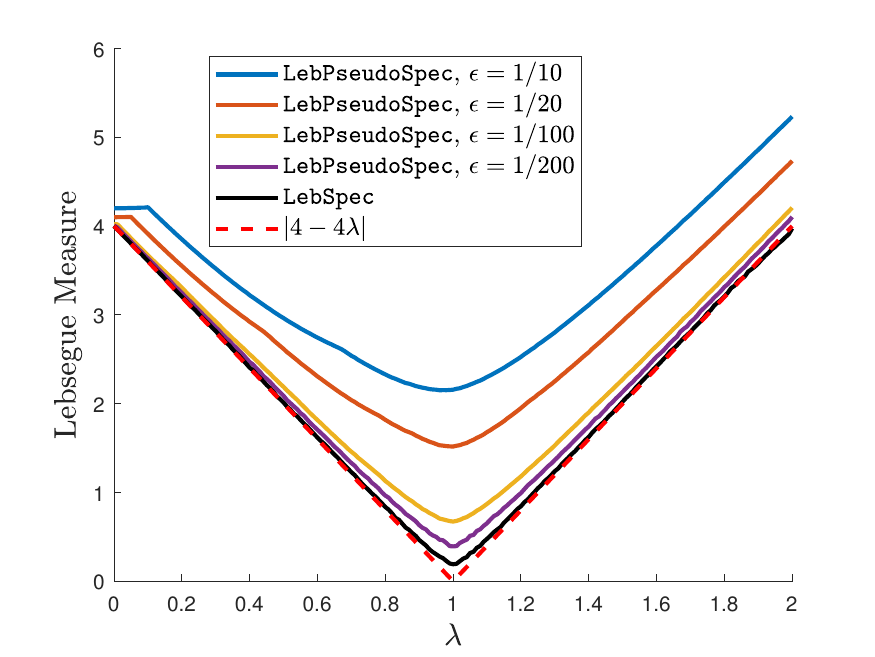}
\end{minipage}
\centering
\begin{minipage}{.49\textwidth}
\centering
\includegraphics[width=1\textwidth,trim={33mm 125mm 40mm 97mm},clip]{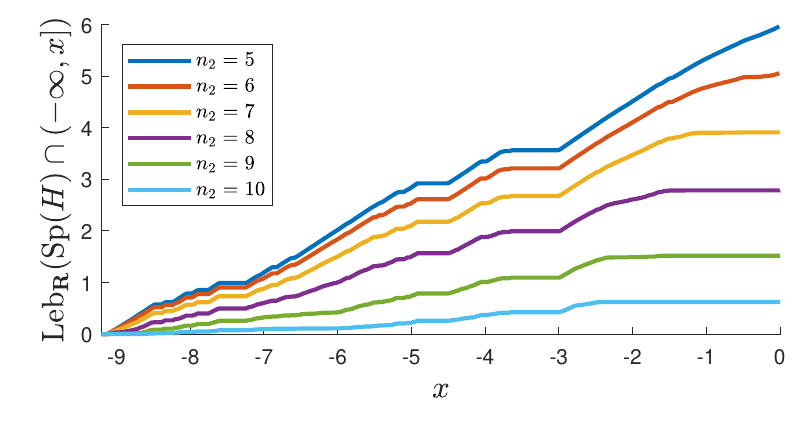}
\includegraphics[width=1\textwidth,trim={33mm 92mm 40mm 165mm},clip]{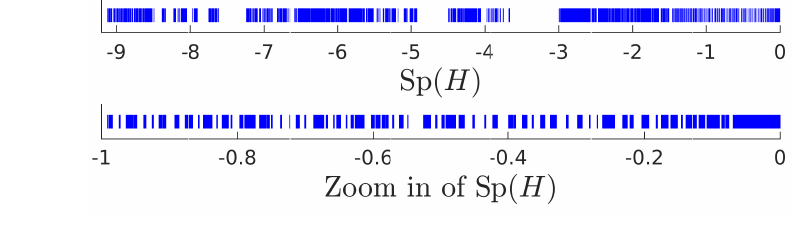}
\end{minipage}
\caption{Left: Output of algorithm \texttt{LebSpec} to compute $\mathrm{Leb}_{\mathbb{R}}(\mathrm{Sp}(H_{\alpha}))$ as well as the algorithm \texttt{LebPseudoSpec} for $\mathrm{Leb}_{\mathbb{R}}(\mathrm{Sp}_{\epsilon}(H_{\alpha}))$ (which converges to $\mathrm{Leb}_{\mathbb{R}}(\mathrm{Sp}(H_{\alpha}))$ as $\epsilon\downarrow 0$). These were computed using $n_1=10^4$ and $n_2=7$. Right: Estimates for $\mathrm{Leb}_{\mathbb{R}}(\mathrm{Sp}(H)\cap(-\infty,x])$, where $H$ is the Laplacian on a Penrose tiling in \eqref{H_0}, obtained by letting $n_1=10^5$ and selecting different $n_2$. The estimate above $-3$ appears to be well resolved, suggesting a region of Lebesgue measure $0$.}
\label{AM_leb}
\end{figure}

First, we consider the almost Mathieu operator, which is related to a wealth of mathematical and physical problems such as the Ten Martini Problem \cite{avila2009ten}. The operator acts on $l^2(\mathbb{Z})$ via
\begin{equation}
(H_{\alpha}x)_n=x_{n-1}+x_{n+1}+2\lambda\cos(2\pi n\alpha)x_n.
\label{Harper}
\end{equation}
The choice of $\lambda=1$ was studied in Hofstadter's classic paper \cite{hofstadter1976energy}, giving rise to the famous Hofstadter butterfly. In this case, the Hamiltonian represents a crystal electron in a uniform magnetic field and the spectrum can be interpreted as the allowed energies of the system. For irrational $\alpha$, we have \cite{avila2006}
\begin{equation}
\label{important_conje}
\mathrm{Leb}_{\mathbb{R}}(\mathrm{Sp}(H_{\alpha}))=4\left|1-\left|\lambda\right|\right|
\end{equation}
and we consider the case $\alpha=(\sqrt{5}-1)/2$. Recall that the SCI classification for computing the Lebesgue measure of the spectrum of such operators (where the dispersion is known) is $\Pi_2^A$, whereas the SCI classification of computing the Lebesgue measure of the pseudospectrum is $\Sigma_1^A$ (see Theorems \ref{Leb_1}, \ref{Leb_2} and \ref{Leb_3} for the further classifications). For computing the Lebesgue measure of the spectrum, the first parameter, $n_1$, controls the size of the truncation used to compute the approximation of the resolvent norm, whereas the second, $n_2$, controls the grid refinement (the spacings are $2^{-n_2}$). For the pseudospectrum, $n_1$ controls the size of the truncations and the grid spacings.

Figure \ref{AM_leb} (left) shows the output of the algorithms computing $\mathrm{Leb}_{\mathbb{R}}(\mathrm{Sp}(H_{\alpha}))$ (\texttt{LebSpec}) and also $\mathrm{Leb}_{\mathbb{R}}(\mathrm{Sp}_{\epsilon}(H_{\alpha}))$ (\texttt{LebPseudoSpec}) for a range of values of $\epsilon$. We chose values of $n_1=10^4$ and a grid spacing of $1/128$ ($n_2=7$). One can clearly see that the estimates for $\mathrm{Leb}_{\mathbb{R}}(\mathrm{Sp}_{\epsilon}(H_{\alpha}))$ are decreasing to the true value of $\mathrm{Leb}_{\mathbb{R}}(\mathrm{Sp}(H_{\alpha}))$, which is well approximated by \texttt{LebSpec}.

Next, we consider the operator $H$ in \eqref{H_0}, for which the Lebesgue measure of $\mathrm{Sp}(H)$ is unknown. We set $n_1=10^5$ and look at the average estimated error of the output via \texttt{DistSpec} (see Appendix \ref{basic_spec22}). This was of the order $10^{-3}$, so we consider grid refinements of spacing $1/32, 1/64,\ldots,1/1024$ corresponding to $n_2=5, 6,\ldots,10$. Figure \ref{AM_leb} (right) shows the output as a cumulative Lebesgue measure, that is, an estimate of $\mathrm{Leb}_{\mathbb{R}}(\mathrm{Sp}(A)\cap(-\infty,x])$ for a given $x$, along with the computed spectrum (for a grid spacing of $10^{-5}$). The figure provides strong evidence that the part of the spectrum closest to $0$ is resolved by the algorithm and has Lebesgue measure zero. We shall see more evidence for this in \S \ref{fract_omwmcf}.

\subsection{Fractal dimension}
\label{fract_omwmcf}

\begin{figure}
\centering
\includegraphics[width=0.49\textwidth,trim={0mm 0mm 0mm 0mm},clip]{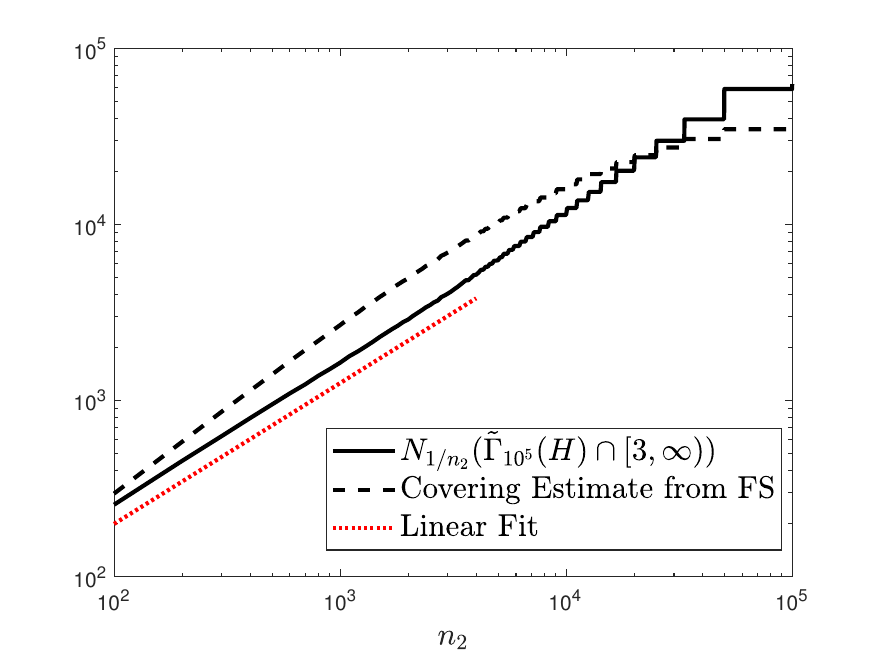}
\caption{A plot of $N_{1/n_2}(\tilde{\Gamma}_{10^5}(H)\cap[-3,\infty))$ against $n_2$. We found a scaling region with estimated box-counting dimension $\approx 0.80$. Note that for large $n_2\gtrsim 5000$, scalings are not resolved by $\tilde{\Gamma}_{10^5}$ (we can predict when this happens using the $\Sigma_1^A$ property of $\tilde{\Gamma}_{n}$). We have also shown the approximation using finite sections (square $10^5\times 10^5$ matrix truncations), as a dashed line, which overestimate the size of coverings, cannot detect the fractal structure, and break down for smaller $n_2$.}
\label{FD_ex}
\end{figure}

For this example, we again consider the operator $H$ in \eqref{H_0}, for which the fractal dimension of $\mathrm{Sp}(H)$ is unknown. In Figure \ref{FD_ex}, we plot $N_{1/n_2}(\tilde{\Gamma}_{10^5}(H)\cap[-3,\infty))$ against $n_2$ (recall that $N_\delta($F$)$ is the number of closed intervals of length $\delta> 0$ required to cover $F$). This corresponds to a rectangular truncation with $n_1=10^5$ columns. Recall that $\tilde{\Gamma}_{n}$ denotes the algorithm that converges to the spectrum with error control, in particular avoiding spectral pollution (see Appendix \ref{basic_spec22}). We also show a linear fit of slope $0.8$. The error control provided by the algorithm $\tilde{\Gamma}_{n}$ allows us to deduce the region where the fit holds, corresponding to a reliable resolution of the spectrum (this is at least as large as the region shown in the plot). In other words, we can ensure that $n_2$ is not too large so that the spacings of the coverings are not smaller than the numerically resolved spectrum. As expected, when $n_2$ is too large we see the effect of the grid spacing and the unresolved spectrum (by choosing larger $n_1$, we can take $n_2$ larger). The figure suggests that the spectrum above $-3$ is fractal with box-counting dimension $\approx 0.8$ and hence has Lebesgue measure zero, in agreement with the findings in Figure \ref{AM_leb}.

Figure \ref{FD_ex} also shows what happens when one performs the same experiment but with a finite section replacing $\tilde{\Gamma}_{n}$ (now using a square $10^5\times 10^5$ truncation). There are two noticeable features. First, for small $n_2$, using a finite section produces an overestimate of the size of the covering and the corresponding slope of the graph due to spectral pollution. In other words, finite section prevents us from detecting the fractal spectrum. Second, the covering estimate via finite section breaks down at smaller $n_2$ and it is impossible to predict suitable values of $n_2$ so that the spacings of the coverings do not go beyond the resolution of the computed spectrum. Together, these issues highlight why the finite section method is unsuitable in general\footnote{There do exist examples of operators, typically with a lot of structure, where one can use periodic versions of finite section.} for approximating fractal dimensions and why the new algorithms in this paper (which are proven to converge) are needed.

\section{Mathematical Preliminaries and Combinatorial Problems in the SCI Hierarchy}
\label{bigHth}

In this section, we begin by providing formal definitions of the SCI hierarchy. We then link the SCI hierarchy, in a certain specific case, to the Baire hierarchy on a suitable topological space. As well as being interesting in its own right, this provides a useful method of providing canonical problems high up in the SCI hierarchy. In particular, the results we prove hold for towers of \textit{general} algorithms (see Definition \ref{Gen_alg}) without the restrictions of arithmetic operations or notions of recursivity etc. This will be used extensively in the proofs of lower bounds for spectral problems that have $\mathrm{SCI}>2$, where we typically reduce the problems discussed here to the given spectral problem. It should be stressed that such links to existing hierarchies only exist in special cases when $\Omega$ and $\mathcal{M}$ are particularly well-behaved. Even when such a link does exist, the induced topology on $\Omega$ is often too complicated, unnatural or strong to be useful from a computational viewpoint. We also take the view that, for problems of scientific interest, the mappings $\Lambda$ and metric space $\mathcal{M}$ are often given to us apriori from the corresponding applications and are typically not compatible with topological viewpoints of computation.

\subsection{The SCI hierarchy}
\label{SCI_Hierarchy}

We begin by defining the Solvability Complexity Index (SCI) hierarchy, allowing us to show that our algorithms realise the boundary of what computers can achieve. We have already presented the definition of a computational problem $\{\Xi,\Omega,\mathcal{M},\Lambda\}$ in \S \ref{sec:SCI_brief_intro_FF}. Recall that the goal is to find algorithms that approximate the function $\Xi$. More generally, the main pillar of our framework is the concept of a tower of algorithms, which is needed to describe problems that need several successive limits in the computation. However, first one needs the definition of a general algorithm.

\begin{definition}[General Algorithm]\label{Gen_alg}
Given a computational problem $\{\Xi,\Omega,\mathcal{M},\Lambda\}$, a {general algorithm} is a mapping $\Gamma:\Omega\to \mathcal{M}$ such that for each $A\in\Omega$
\begin{itemize}
\item[(i)] there exists a (non-empty) finite subset of evaluations $\Lambda_\Gamma(A) \subset\Lambda$, 
\item[(ii)] the action of $\,\Gamma$ on $A$ only depends on $\{A_f\}_{f \in \Lambda_\Gamma(A)}$ where $A_f := f(A),$
\item[(iii)] for every $B\in\Omega$ such that $B_f=A_f$ for every $f\in\Lambda_\Gamma(A)$, it holds that $\Lambda_\Gamma(B)=\Lambda_\Gamma(A)$.
\end{itemize}
\end{definition}

The definition of a general algorithm is more general than the definition of a Turing machine \cite{turing1937computable} or a BSS machine \cite{BCSS}. A general algorithm has no restrictions on the operations allowed. The only restriction is that it can only take a finite amount of information, though it is allowed to \emph{adaptively} choose the finite amount of information it reads depending on the input. Condition (iii) ensures that the algorithm consistently reads the information. With a definition of a general algorithm, we can define the concept of towers of algorithms.

\begin{definition}[Tower of Algorithms]\label{tower_funct}
Given a computational problem $\{\Xi,\Omega,\mathcal{M},\Lambda\}$, a {tower of algorithms of height $k$
 for $\{\Xi,\Omega,\mathcal{M},\Lambda\}$} is a family of sequences of functions
 $$\Gamma_{n_k}:\Omega
\rightarrow \mathcal{M},\ \Gamma_{n_k, n_{k-1}}:\Omega
\rightarrow \mathcal{M},\quad \dots\quad,\ \Gamma_{n_k, \hdots, n_1}:\Omega \rightarrow \mathcal{M},
$$
where $n_k,\hdots,n_1 \in \mathbb{N}$ and the functions $\Gamma_{n_k, \hdots, n_1}$ at the lowest level of the tower are general algorithms in the sense of Definition \ref{Gen_alg}. Moreover, for every $A \in \Omega$,
$$
\Xi(A)= \lim_{n_k \rightarrow \infty} \Gamma_{n_k}(A), \quad \Gamma_{n_k, \hdots, n_{j+1}}(A)= \lim_{n_j \rightarrow \infty} \Gamma_{n_k, \hdots, n_j}(A) \quad j=k-1,\dots,1.
$$
\end{definition}

In addition to a general tower of algorithms, we focus on arithmetic towers.

\begin{definition}[Arithmetic Tower]\label{arith_tower_def}
Given a computational problem $\{\Xi,\Omega,\mathcal{M},\Lambda\}$, where $\Lambda$ is countable, we define the following: An {arithmetic tower of algorithms} of height $k$
 for $\{\Xi,\Omega,\mathcal{M},\Lambda\}$ is a tower of algorithms where the lowest functions $\Gamma = \Gamma_{n_k, \hdots, n_1} :\Omega \rightarrow \mathcal{M}$ satisfy the following:
 For each $A\in\Omega$ the mapping $(n_k, \hdots, n_1) \mapsto \Gamma_{n_k, \hdots, n_1}(A) = \Gamma_{n_k, \hdots, n_1}(\{A_f\}_{f \in \Lambda})$ is recursive, and $\Gamma_{n_k, \hdots, n_1}(A)$ is a finite string of complex numbers that can be identified with an element in $\mathcal{M}$. For arithmetic towers we let $\alpha = A$.
\end{definition} 

\begin{remark}By recursive we mean the following. If $f(A) \in \mathbb{Q}$ (or $\mathbb{Q}+i\mathbb{Q}$) for all $f \in \Lambda$, $A \in \Omega$, and $\Lambda$ is countable, then $\Gamma_{n_k, \hdots, n_1}(\{A_f\}_{f \in \Lambda})$ can be executed by a Turing machine \cite{turing1937computable}, that takes $(n_k, \hdots, n_1)$ as input, and that has an oracle tape consisting of $\{A_f\}_{f \in \Lambda}$. If $f(A) \in \mathbb{R}$ (or $\mathbb{C}$) for all $f \in \Lambda$, then $\Gamma_{n_k, \hdots, n_1}(\{A_f\}_{f \in \Lambda})$ can be executed by a BSS machine \cite{BCSS} that takes $(n_k, \hdots, n_1)$, as input, and that has an oracle that can access any $A_f$ for $f \in \Lambda$.\hfill$\boxtimes$
\end{remark}

Given the definitions above we can now define the key concept, namely, the Solvability Complexity Index: 

\begin{definition}[Solvability Complexity Index]\label{complex_ind}
A computational problem $\{\Xi,\Omega,\mathcal{M},\Lambda\}$ is said to have {Solvability Complexity Index $\mathrm{SCI}(\Xi,\Omega,\mathcal{M},\Lambda)_{\alpha} = k$}, with respect to a tower of algorithms of type $\alpha$, if $k$ is the smallest integer for which there exists a tower of algorithms of type $\alpha$ of height $k$. If no such tower exists then $\mathrm{SCI}(\Xi,\Omega,\mathcal{M},\Lambda)_{\alpha} = \infty.$ If there exists a tower $\{\Gamma_n\}_{n\in\mathbb{N}}$ of type $\alpha$ and height one such that $\Xi = \Gamma_{n_1}$ for some $n_1 < \infty$, then we define $\mathrm{SCI}(\Xi,\Omega,\mathcal{M},\Lambda)_{\alpha} = 0$. The type $\alpha$ may be General, or Arithmetic, denoted respectively G and A. We may sometimes write $\mathrm{SCI}(\Xi,\Omega)_{\alpha}$ to simplify notation when $\mathcal{M}$ and $\Lambda$ are obvious. 
\end{definition}

We will let $\mathrm{SCI}(\Xi,\Omega)_{\mathrm{A}}$ and $\mathrm{SCI}(\Xi,\Omega)_{\mathrm{G}}$ denote the SCI with respect to an arithmetic tower and a general tower, respectively. Note that a general tower means just a tower of algorithms as in Definition \ref{tower_funct}, where there are no restrictions on the mathematical operations. Thus, clearly $\mathrm{SCI}(\Xi,\Omega)_{\mathrm{A}} \geq \mathrm{SCI}(\Xi,\Omega)_{\mathrm{G}}$. The definition of the SCI immediately induces the SCI hierarchy:

\begin{definition}[The Solvability Complexity Index Hierarchy]
\label{1st_SCI}
Consider a collection $\mathcal{C}$ of computational problems and let $\mathcal{T}$ be the collection of all towers of algorithms of type $\alpha$ for the computational problems in $\mathcal{C}$.
Define 
\begin{equation*}
\begin{split}
\Delta^{\alpha}_0 &:= \{\{\Xi,\Omega\} \in \mathcal{C} \ \vert \   \mathrm{SCI}(\Xi,\Omega)_{\alpha} = 0\}\\
\Delta^{\alpha}_{m+1} &:= \{\{\Xi,\Omega\}  \in \mathcal{C} \ \vert \   \mathrm{SCI}(\Xi,\Omega)_{\alpha} \leq m\}, \qquad \quad m \in \mathbb{N},
\end{split}
\end{equation*}
as well as
\[
\Delta^{\alpha}_{1} := \{\{\Xi,\Omega\}  \in \mathcal{C}   \  \vert \ \exists \ \{\Gamma_n\}_{n\in \mathbb{N}} \in \mathcal{T}\text{ s.t. } \forall A \ d(\Gamma_n(A),\Xi(A)) \leq 2^{-n}\}. 
\]
\end{definition}

When there is additional structure on the metric space, such as in the spectral case when one considers the Attouch--Wets or the Hausdorff metric, one can extend the SCI hierarchy. For non-empty closed sets, we consider the \emph{Attouch--Wets metric} defined by
\begin{equation}\label{eq:Attouch-Wets}
d_{\mathrm{AW}}(C_1,C_2)=\sum_{n=1}^{\infty} 2^{-n}\min\left\{{1,\underset{\left|x\right|\leq n}{\sup}\left|\mathrm{dist}(x,C_1)-\mathrm{dist}(x,C_2)\right|}\right\},
\end{equation}
for $C_1,C_2\in\mathrm{Cl}(\mathbb{C}),$ where $\mathrm{Cl}(\mathbb{C})$ denotes the set of closed non-empty subsets of $\mathbb{C}$. This generalises the familiar Hausdorff metric to unbounded closed sets and corresponds to local uniform converge on compact subsets of $\mathbb{C}$.

\begin{definition}[The SCI Hierarchy (Attouch--Wets/Hausdorff metric)]
Given the set-up in Definition \ref{1st_SCI}, and suppose in addition that $(\mathcal{M},d)$ has the Attouch--Wets or the Hausdorff metric induced by another metric space $(\mathcal{M}^{\prime},d')$,
define, for $m \in \mathbb{N}$,
\begin{align*}
\Sigma^{\alpha}_0 &= \Pi^{\alpha}_0 = \Delta^{\alpha}_0,\\
\Sigma^{\alpha}_{1} &= \{\{\Xi,\Omega\} \in \Delta_{2}^{\alpha} \ \vert \  \exists \ \{\Gamma_{n}\} \in \mathcal{T}, \  \{X_{n}(A)\}\subset\mathcal{M} \text{ s.t. }  \ \Gamma_{n}(A)  \mathop{\subset}_{\mathcal{M}^{\prime}} X_n(A),\\
 & \qquad \qquad \qquad \qquad \lim_{n\rightarrow\infty}\Gamma_{n}(A)=\Xi(A),\ \ d(X_{n}(A),\Xi(A))\leq 2^{-n} \ \ \forall A \in \Omega\}, \\
\Pi^{\alpha}_{1} &= \{\{\Xi,\Omega\} \in \Delta_{2}^{\alpha} \ \vert \  \exists \ \{\Gamma_{n}\} \in \mathcal{T}, \  \{X_{n}(A)\}\subset\mathcal{M} \text{ s.t. }  \ \Xi(A)  \mathop{\subset}_{\mathcal{M}^{\prime}} X_{n}(A),\\
& \qquad \qquad \qquad \qquad \lim_{n\rightarrow\infty}\Gamma_{n}(A)=\Xi(A),\ \ d(X_{n}(A),\Gamma_n(A))\leq 2^{-n} \ \ \forall A \in \Omega\},
\end{align*}
where $\mathop{\subset}_{\mathcal{M}^{\prime}}$ means inclusion in the metric space $\mathcal{M}^{\prime}$, and $\{X_{n}(A)\}$ is a sequence where $X_n(A) \in \mathcal{M}$ depends on $A$. Moreover, 
\begin{equation*}
\begin{split}
\Sigma^{\alpha}_{m+1} = \{\{\Xi,\Omega\} \in \Delta_{m+2}^{\alpha} \ &\vert \  \exists \ \{\Gamma_{n_{m+1},\ldots,n_1}\} \in \mathcal{T}, \  \{X_{n_{m+1}}(A)\}\subset\mathcal{M} \text{ s.t. }  \ \Gamma_{n_{m+1}}(A)  \mathop{\subset}_{\mathcal{M}^{\prime}} X_{n_{m+1}}(A),\\
 & \hspace{-3mm}\lim_{n_{m+1}\rightarrow\infty}\Gamma_{n_{m+1}}(A)=\Xi(A),\ \ d(X_{n_{m+1}}(A),\Xi(A))\leq 2^{-n_{m+1}} \ \ \forall A \in \Omega\}, \\
\Pi^{\alpha}_{m+1} = \{\{\Xi,\Omega\} \in \Delta_{m+2}^{\alpha} \ &\vert \  \exists \ \{\Gamma_{n_{m+1},\ldots,n_1}\} \in \mathcal{T}, \  \{X_{n_{m+1}}(A)\}\subset\mathcal{M} \text{ s.t. }  \ \Xi(A)  \mathop{\subset}_{\mathcal{M}^{\prime}} X_{n_{m+1}}(A),\\
& \hspace{-3mm}\lim_{n_{m+1}\rightarrow\infty}\Gamma_{n_{m+1}}(A)=\Xi(A),\ \ d(X_{n_{m+1}}(A),\Gamma_{n_{m+1}}(A))\leq 2^{-n_{m+1}} \ \ \forall A \in \Omega\},
\end{split}
\end{equation*}
where $d$ can be either $d_{\mathrm{H}}$ or $d_{\mathrm{AW}}$.
\end{definition}

Note that to build a $\Sigma_1$ algorithm, it is enough (by taking subsequences of $n$) to construct $\Gamma_n(A)$ such that $\Gamma_{n}(A) \subset \mathcal{N}_{E_n(A)}(\Xi(A))$ with some computable $E_n(A)$ that converges to zero. The same idea can be applied to the real line with the usual metric, or $\{0,1\}$ with the discrete metric (we interpret $1$ as ``Yes'').

\begin{definition}[The SCI Hierarchy (totally ordered set)]
Given the set-up in Definition \ref{1st_SCI} and suppose in addition that $\mathcal{M}$ is a totally ordered set. 
Define 
\begin{equation*}
\begin{split}
\Sigma^{\alpha}_0 &= \Pi^{\alpha}_0 = \Delta^{\alpha}_0,\\
\Sigma^{\alpha}_{1} &= \{\{\Xi,\Omega\} \in \Delta_{2}^{\alpha} \ \vert \  \exists \ \{\Gamma_{n}\} \in \mathcal{T} \text{ s.t. } \Gamma_{n}(A) \nearrow \Xi(A) \ \, \forall A \in \Omega\}, 
\\
\Pi^{\alpha}_{1} &= \{\{\Xi,\Omega\} \in \Delta_{2}^{\alpha} \ \vert \  \exists \ \{\Gamma_{n}\} \in \mathcal{T} \text{ s.t. } \Gamma_{n}(A) \searrow \Xi(A) \ \, \forall A \in \Omega\},
\end{split}
\end{equation*}
where $\nearrow$ and $\searrow$ denotes convergence from below and above respectively,
as well as, for $m \in \mathbb{N}$, 
\begin{equation*}
\begin{split}
\Sigma^{\alpha}_{m+1} &= \{\{\Xi,\Omega\} \in \Delta_{m+2}^{\alpha} \ \vert \  \exists \ \{\Gamma_{n_{m+1}, \hdots, n_1}\} \in \mathcal{T} \text{ s.t. }\Gamma_{n_{m+1}}(A) \nearrow \Xi(A) \ \, \forall A \in \Omega\}, \\
\Pi^{\alpha}_{m+1} &= \{\{\Xi,\Omega\} \in \Delta_{m+2}^{\alpha} \ \vert \  \exists \ \{\Gamma_{n_{m+1}, \hdots, n_1}\} \in \mathcal{T} \text{ s.t. }\Gamma_{n_{m+1}}(A) \searrow \Xi(A) \ \, \forall A \in \Omega\}.
\end{split}
\end{equation*}
\end{definition}

\begin{remark}[$\Delta^{\alpha}_1\subsetneq \Sigma^{\alpha}_1 \subsetneq \Delta^{\alpha}_2$]
Note that the inclusions are strict. For example, if $\Omega_K$ consists of the set of compact infinite matrices acting on $l^2(\mathbb{N})$ and $\Xi(A)=\mathrm{Sp}(A)$ (the spectrum of $A$) then $\{\Xi, \Omega_K\} \in \Delta^{\alpha}_2$ but not in $ \Sigma_1^\alpha\cup\Pi_1^\alpha$ for $\alpha$ representing either towers of arithmetical or general type (see \cite{ben2015can} for a proof). Moreover, as was demonstrated in \cite{colb1}, if $\tilde \Omega$ is the set of discrete Schr\"odinger operators on $l^2(\mathbb{Z})$, then $\{\Xi, \tilde \Omega\} \in \Sigma^{\alpha}_1$ but not in $\Delta^{\alpha}_1$.\hfill$\boxtimes$
\end{remark}

Suppose we are given a computational problem $\{\Xi, \Omega, \mathcal{M}, \Lambda\}$, and that $\Lambda = \{f_j\}_{j \in \beta}$, where $\beta$ is some index set that can be finite or infinite. Obtaining $f_j$ may be a computational task on its own, which is exactly the problem in most areas of 
computational mathematics. In particular, for $A \in \Omega$, $f_j(A)$ could be the number $e^{\frac{\pi}{j} i }$ for example. Hence, we cannot access $f_j(A)$, but rather $f_{j,n}(A)$ where $f_{j,n}(A) \rightarrow f_{j}(A)$ as $n \rightarrow \infty$. 
Or, just as for problems that are high up in the SCI hierarchy, it could be that we need several successive limits, in particular one may need mappings
$f_{j,n_m,\hdots, n_1}: \Omega \rightarrow \mathbb{D} + i\mathbb{D}$, where $\mathbb{D}$ denotes the dyadic rational numbers, such that 
\begin{equation}\label{Lambda_limits}
\lim_{n_m \rightarrow \infty} \hdots \lim_{n_1 \rightarrow \infty} \|f_{j,n_m,\hdots, n_1}(A) - f_j(A)\|_{\infty} = 0 \quad  \forall j\in\beta,\forall A \in \Omega.
\end{equation}

In particular, we may view the problem of obtaining $f_j(A)$ as a problem in the SCI hierarchy, where $\Delta_1$ classification would correspond to the existence of mappings $f_{j,n}: \Omega \rightarrow \mathbb{D} + i \mathbb{D}$
such that 
 \begin{equation}\label{Lambda_limits2}
 \|f_{j,n}(A) - f_j(A)\|_{\infty} \leq 2^{-n} \quad \forall j\in\beta,\forall A \in \Omega.
 \end{equation}

This idea is formalised in the following definition.

\begin{definition}[$\Delta_{m}$-information]\label{definition:Lambda_limits}
	Let $\{\Xi, \Omega, \mathcal{M}, \Lambda\}$ be a computational problem. For $m \in \mathbb{N}$ we say that $\Lambda$ has $\Delta_{m+1}$-information if each $f_j \in \Lambda$ is not available, however, there are mappings $f_{j,n_m,\hdots, n_1}: \Omega \rightarrow \mathbb{D} + i \mathbb{D}$ such that \eqref{Lambda_limits} holds. Similarly, for $m = 0$ there are mappings $f_{j,n}: \Omega \rightarrow \mathbb{D} + i \mathbb{D}$
	such that \eqref{Lambda_limits2} holds. Finally, if $k \in \mathbb{N}$ and $\hat \Lambda$ is a collection of such functions described above such that $\Lambda$ has $\Delta_k$-information, we say that $\hat \Lambda$ provides $\Delta_k$-information for $\Lambda$. Moreover, we denote the family of all such $\hat \Lambda$ by $\mathcal{L}^k(\Lambda)$. 
\end{definition}

We want algorithms that can handle all  computational problems $\{\Xi,\Omega,\mathcal{M},\hat \Lambda\}$ when $\hat \Lambda \in  \mathcal{L}^m(\Lambda)$. To formalise this, we define a computational problem with $\Delta_m$-information.

\begin{definition}[Computational problem with $\Delta_m$-information]
\label{inexact_def_need}
	Given $m \in \mathbb{N}$ with $m>1$, a computational problem where $\Lambda$ has $\Delta_m$-information is denoted by 
	$
	\{\Xi,\Omega,\mathcal{M},\Lambda\}^{\Delta_m} := \{\tilde \Xi,\tilde \Omega,\mathcal{M},\tilde \Lambda\},
	$ 
where 
$$
\tilde \Omega =\left\{ \tilde A = \{f_{j,n_{m-1},\hdots, n_1}(A)\}_{j,n_{m-1},\hdots, n_1 \in \beta \times \mathbb{N}^{m-1}} \, \vert \, A \in \Omega, \{f_j\}_{j \in \beta} = \Lambda, f_{j,n_{m-1},\hdots, n_1} \text{ satisfy \eqref{Lambda_limits}} \right\}.
$$
Moreover, $\tilde \Xi(\tilde A) = \Xi(A)$, and we have $\tilde \Lambda = \{\tilde f_{j,n_{m-1},\hdots, n_1}\}_{j,n_{m-1},\hdots, n_1 \in \beta \times \mathbb{N}^{m-1}}$ where $\tilde f_{j,n_{m-1},\hdots, n_1}(\tilde A) = f_{j,n_{m-1},\hdots, n_1}(A)$. Note that $\tilde \Xi$ is well-defined by Definition \ref{def:comp_prob} of a computational problem. Similarly, we define $\Delta_1$-information using \eqref{Lambda_limits2}.
\end{definition}

The SCI and the SCI hierarchy, given $\Delta_m$-information, are then defined in the standard obvious way. 

\begin{remark}[Classifications in this paper]
\label{fdjojosr}
For the problems considered in this paper, the SCI classifications do not change if we consider arithmetic towers with $\Delta_1$-information. This is easy to see through Church's thesis and an analysis of the stability of our algorithms. For example, when the input is rational we have been careful to restrict all relevant operations to $\mathbb{Q}$ rather than $\mathbb{R}$, and errors incurred from $\Delta_1$-information can be removed in the first limit. Explicitly, for the algorithms based on \texttt{DistSpec} (see Appendix \ref{basic_spec22}) it is possible to carry out an error analysis. We can also bound numerical errors (e.g., using interval arithmetic \cite{tucker2011validated}) and incorporate this uncertainty for the estimation of $\left\|R(z,A)\right\|^{-1}$ to gain the same classification of our problems. Similarly, for other algorithms based on similar functions. In other words, for the results of this paper, it does not matter which model of computation one uses for a definition of `algorithm'. From a classification point of view, they are equivalent for these spectral problems. This leads to rigorous $\Sigma_k^\alpha$ or $\Pi_k^\alpha$ type error control suitable for verifiable numerics. In particular, for $\Sigma_1^{\alpha}$ or $\Pi_1^{\alpha}$ towers of algorithms, this could be useful for computer-assisted proofs.\hfill$\boxtimes$
\end{remark}

\subsection{Recalling some results from descriptive set theory}
We briefly recall the definition of the Borel hierarchy as well as some well-known theorems from descriptive set theory. It is beyond the scope of this paper to provide an extensive discussion of descriptive set theory, but we refer the reader to \cite[Chapter 2]{kechris1987descriptive} for an excellent introduction that covers the main ideas.

Let $X$ be a metric space and define
$$
\Sigma_1^0(X)=\{U\subset X:U\text{ is open}\},\quad \Pi_1^0(X)=\sim\hspace{-1mm}\Sigma_1^0(X)=\{F\subset X:F\text{ is closed}\},
$$
where for a class $\mathcal{U}$, $\sim\hspace{-1mm}\mathcal{U}$ denotes the class of complements (in $X$) of elements of $\mathcal{U}$. Inductively define
\begin{align*}
\Sigma^0_\xi(X)&=\{\cup_{n\in\mathbb{N}}A_n:A_n\in\Pi^0_{\xi_n},\xi_n<\xi\},\text{ if }\xi>1,\\
\Pi_\xi^0(X)&=\sim\hspace{-1mm}\Sigma_\xi^0(X),\quad\Delta_\xi^0(X)=\Sigma^0_\xi(X)\cap \Pi_\xi^0(X).
\end{align*}
The full Borel hierarchy extends to all $\xi<\omega_1$ ($\omega_1$ being the first uncountable ordinal) by transfinite induction but we do not need this here.

\begin{definition}
Given a class of subsets, $\mathcal{U}$, of a metric space $X$ and given another metric space $Y$, we say that the function $f:X\rightarrow Y$ is $\mathcal{U}$-measurable if $f^{-1}(U)\in\mathcal{U}$ for every open set $U\subset Y$.
\end{definition}

Given metric spaces $X$ and $Y$, the Baire hierarchy is defined as follows. A function $f:X\rightarrow Y$ is of Baire class 1, written $f\in\mathcal{B}_1$, if it is $\Sigma_2^0(X)$-measurable. For $1<\xi<\omega_1$, a function $f:X\rightarrow Y$ is of Baire class $\xi$, written $f\in\mathcal{B}_\xi$, if it is the pointwise limit of a sequence of functions $f_n$ in $\mathcal{B}_{\xi_n}$ with $\xi_n<\xi$. The following Theorem is well-known \cite[Section 24]{kechris1987descriptive} and provides a useful link between the Borel and Baire hierarchies.

\begin{theorem}[Lebesgue, Hausdorff, Banach]
\label{Baire1}
Let $X,Y$ be metric spaces with $Y$ separable and $1\leq\xi<\omega_1$. Then $f\in\mathcal{B}_\xi$ if and only if it is $\Sigma_{\xi+1}^0(X)$ measurable. Furthermore, if $X$ is zero-dimensional (Hausdorff with a basis of clopen sets) and $f\in\mathcal{B}_1$, then $f$ is the pointwise limit of a sequence of continuous functions.
\end{theorem}

The assumption that $X$ is zero-dimensional in the last statement is important. Without any assumptions, the final statement of the theorem is false, as is easily seen by considering $X=\mathbb{R}$. Examples of zero-dimensional spaces include products of the discrete space $\{0,1\}$ or the Cantor space. Any such space is necessarily totally disconnected, meaning that the connected components in the space are the one-point sets (the converse is true for locally compact Hausdorff spaces). Our primary interest will be when $Y$ is equal to $\{0,1\}$ or $[0,1]$, both with their natural topologies.

\subsection{Linking the SCI hierarchy to the Baire hierarchy in a special case}

\begin{definition}
\label{closed_under_search}
Given the triple $\{\Omega,\mathcal{M},\Lambda\}$, a class of algorithms $\mathcal{A}$ is closed under search with respect to $\{\Omega,\mathcal{M},\Lambda\}$ if whenever
\begin{enumerate}
	\item $\mathcal{I}$ is an index set,
	\item $\{n_i\}_{i\in\mathcal{I}}$ a family of natural numbers,
	\item $\{\Gamma_{i,l}:\Omega\rightarrow\mathcal{M}\}_{i\in\mathcal{I},l\leq n_i}\subset\mathcal{A}$,
	\item $\{U_{i,l}\}_{i\in\mathcal{I},l\leq n_i}$ family of basic open sets in $\mathcal{M}$ with	$\cup_{i\in\mathcal{I}}\cap_{l\leq n_i}\Gamma_{i,l}^{-1}(U_{i,l})=\Omega,$
	\item $\{c_i\}_{i\in\mathcal{I}}$ a family of points in some arbitrary dense subset of $\mathcal{M}$,
\end{enumerate}
then there is some $\Gamma\in\mathcal{A}$ such that for every $x\in\Omega$ there exists some $i\in\mathcal{I}$ with $\Gamma(x)=c_i$ and for all $l\leq n_i$ we have $\Gamma_{i,l}(x)\in U_{i,l}$.
\end{definition}

\begin{proposition}
\label{search_prop}
Suppose that $\mathcal{A}$ is closed under search with respect to $\{\Omega,\mathcal{M},\Lambda\}$, then there exists a topology $\mathcal{T}$ on $\Omega$ such that $\Delta_1^{\mathcal{A}}$ is precisely the set of continuous functions from $(\Omega,\mathcal{T})$ to $\mathcal{M}$.
\end{proposition} 

\begin{proof}
Let $\mathcal{T}$ be the topology generated by $\{\Gamma^{-1}(B):\Gamma\in\mathcal{A},B\subset\mathcal{M}\text{ basic open}\}$. Any $\Gamma\in\mathcal{A}$ is continuous with respect to this topology. Uniform limits of continuous functions into metric spaces are also continuous, and hence any function in $\Delta_1^{\mathcal{A}}$ is continuous with respect to $\mathcal{T}$.

For the other direction, suppose that $f: (\Omega,\mathcal{T})\rightarrow \mathcal{M}$ is continuous. Choose $\{c_i\}_{i\in\mathcal{I}}\subset\mathcal{M}$ such that $\mathcal{M}\subset \cup_{i\in\mathcal{I}}D(c_i,2^{-n})$. Continuity of $f$ implies that $f^{-1}(D(c_i,2^{-n}))$ are open. This implies that there is an index set $\mathcal{J}$, natural numbers $\{n_{i,j}\}_{j\in\mathcal{J}}$, a family $\{\Gamma_{i,j,l}\}_{i\in\mathcal{I},j\in\mathcal{J},l\leq n_{i,j}}$ (in $\mathcal{A}$) and a family of basic open sets $\{U_{i,j,l}\}_{i\in\mathcal{I},j\in\mathcal{J},l\leq n_{i,j}}$ with the property that
$$
f^{-1}(D(c_i,2^{-n}))=\bigcup_{j\in\mathcal{J}}\bigcap_{l\leq n_{i,j}}\Gamma^{-1}_{i,j,l}(U_{i,j,l}).
$$
It follows that
$$
\bigcup_{i\in\mathcal{I},j\in\mathcal{J}}\bigcap_{l\leq n_{i,j}}\Gamma^{-1}_{i,j,l}(U_{i,j,l})=\Omega.
$$
Since $\mathcal{A}$ is closed under search, there exists $f_n\in\mathcal{A}$ such that for every $x\in\Omega$ there exists some $i\in\mathcal{I}$ and $j\in\mathcal{J}$ with $f_n(x)=c_i$ and for all $l\leq n_{i,j}$, $x\in\Gamma^{-1}_{i,j,l}(U_{i,j,l}).$ But this implies that $d(f_n(x),f(x))< 2^{-n}$. Since $n$ was arbitrary, we have $f\in\Delta_1^{\mathcal{A}}$.
\end{proof}

The generated topology can be very perverse and not every class of algorithms is closed under search. However, we do have the following useful theorem when $\Omega$ (and $\Lambda$) is a particularly simple discrete space.

\begin{theorem}
\label{B_equiv}
Suppose that $\Omega=\{0,1\}^{\mathbb{N}}=\{\{a_i\}_{i\in\mathbb{N}}:a_i\in\{0,1\}\}$ with the set of evaluation functions $\Lambda$ equal to the set of pointwise evaluations $\{\lambda_{j}(a):=a_j:j\in\mathbb{N}\}$ and let $\mathcal{M}$ be an arbitrary separable metric space with at least two separated points. Endow $\Omega$ with the product topology, $\tilde{\mathcal{T}}$, induced by the discrete topology on $\{0,1\}$ and consider the Baire hierarchy, $\{\mathcal{B}_{\xi}((\Omega,\tilde{\mathcal{T}}),\mathcal{M})=\mathcal{B}_{\xi}\}_{\xi<\omega_1}$, of functions $f:\Omega\rightarrow\mathcal{M}$. Then for any problem function $\Xi:\Omega\rightarrow\mathcal{M}$ and $m\in\mathbb{N}$,
\begin{equation}
\label{Baire_equiv}
\{\Xi,\Omega,\Lambda\}\in\Delta^G_{m+1}\Leftrightarrow \Xi\in\mathcal{B}_{m}.
\end{equation}
In other words, the SCI corresponds to the Baire hierarchy index.
\end{theorem}

\begin{remark}
The proof shows that we can replace $\Omega$ by $\{0,1\}^{\mathbb{N}\times\mathbb{N}}$ or any other such product space (induced by a discrete topology) of the form $A^B$ with $A,B$ countable, with $\Lambda$ the corresponding component-wise evaluations, as long as $\mathcal{M}$ has at least $\left|A\right|$ jointly separated points and is separable.\hfill$\boxtimes$
\end{remark}

\begin{proof}
First we show that general algorithms are closed under search and that the topology $\mathcal{T}$ in Proposition \ref{search_prop} is equal to the product topology $\tilde{\mathcal{T}}$. Without loss of generality we can assume that $\mathcal{I}$ is well-ordered by $\prec$. Given $x\in\Omega$, let $k\in\mathbb{N}$ be minimal such that there exists $i\in\mathcal{I}$ with $x\in\cap_{l\leq n_i}\Gamma^{-1}_{i,l}(U_{i,l})$ and $\Lambda_{\Gamma_{i,l}}(x)\subset\{\lambda_j:j\leq k\}$ for $l\leq n_i$. Let $i_0$ be the $\prec$-least witness for $k$ and then define $\Gamma(x)=c_{i_0}$. The well-ordering of $\mathcal{I}$ implies that $\Gamma$ is a general algorithm and it clearly satisfies the requirements in the definition of closed under search. Note that this part of the proof only uses countability of $\Lambda$.

To equate the topologies, suppose that $\Gamma\in\Delta_0^G$ is a general algorithm. For each $a\in\Omega,$ $\Lambda_{\Gamma}(a)$ is finite and we can assume without loss of generality that it is equal to $\{\lambda_j:j\leq I(a)\}$ for some finite $I(a)$. In particular, there exists an open set $U_a$ such that any $b\in U_a$ has $\lambda_j(b)=\lambda_j(a)$ for $j\leq I(a)$ and hence $\Gamma(b)=\Gamma(a)$.
Then for any open set $B\subset\mathcal{M}$
$$
\Gamma^{-1}(B)=\bigcup_{a\in \Gamma^{-1}(B)}U_a
$$
is open. Hence each $\Gamma$ is continuous with respect to the product topology on $\Omega$. It follows that $\mathcal{T}\subset\tilde{\mathcal{T}}$. To prove the converse, we must show that each projection map $\lambda_j$ is continuous with respect to $\mathcal{T}$. Let $x_1,x_2$ be separated points in $\mathcal{M}$ and consider $f:\{0,1\}\rightarrow\mathcal{M}$ with $f(0)=x_1$ and $f(1)=x_2$. Then the composition $f\circ\lambda_j$ is a general algorithm and hence continuous with respect to $\mathcal{T}$. But this implies that $\lambda_j$ is continuous. It follows from Proposition \ref{search_prop} that $\{\Xi,\Omega,\Lambda\}\in\Delta^G_{1}$ if and only if $\Xi$ is continuous.

Now the space $(\Omega,\mathcal{T})$ is zero-dimensional and $\mathcal{M}$ is separable, hence by Theorem \ref{Baire1}, any element of $\mathcal{B}_1$ is a limit of continuous functions. The converse holds in greater generality. It follows that $\Xi\in\mathcal{B}_m$ if and only if there are $f_{n_m,\ldots,n_1}\in\Delta^G_{1}$ with
\begin{equation}
\label{Baire_lim}
\Xi(a)=\lim_{n_m\rightarrow\infty}\cdots\lim_{n_1\rightarrow\infty}f_{n_m,\ldots,n_1}(a).
\end{equation}
If this holds then there exists general algorithms $\Gamma_{n_m,\ldots,n_1}$ such that for all $a\in\Omega$,
$$
d(\Gamma_{n_m,\ldots,n_1}(a),f_{n_m,\ldots,n_1}(a))\leq 2^{-n_1}
$$
and hence
$$
\lim_{n_m\rightarrow\infty}\cdots\lim_{n_1\rightarrow\infty}\Gamma_{n_m,\ldots,n_1}(a)=\Xi(a)
$$
so that $\{\Xi,\Omega,\Lambda\}\in\Delta^G_{m+1}$. Conversely if $\{\Xi,\Omega,\Lambda\}\in\Delta^G_{m+1}$ with tower of algorithms $\{\Gamma_{n_{m},\ldots,n_1}\}$, then since each general algorithm is continuous, \eqref{Baire_lim} holds with $f_{n_m,\ldots,n_1}(a)=\Gamma_{n_{m},\ldots,n_1}$.
\end{proof}

\subsection{Combinatorial problems high up in the SCI hierarchy} We can now combine the results of the previous two subsections and obtain combinatorial problems high up in the SCI hierarchy. Let $k\in\mathbb{N}_{\geq 2}$ and let $\Omega_k$ denote the collection of all infinite arrays $\{a_{m_1,\ldots,m_k}\}_{m_1,\ldots,m_k\in\mathbb{N}}$ with entries $a_{m_1,\ldots,m_k}\in\{0,1\}$. As usual, $\Lambda_k$ is the set of component-wise evaluations/projections. Consider the formulas
\begin{align*}
P(a,m_1,\ldots,m_{k-2})&=\begin{cases}
1,\quad\text{if }\exists i \text{ }\forall j\text{ } \exists n>j \text{ s.t. }a_{m_1,\ldots,m_{k-2},n,i}=1\\
0,\quad\text{otherwise}
\end{cases},\\
Q(a,m_1,\ldots,m_{k-2})&=\begin{cases}
1,\quad\text{if } \forall^{\infty}i\forall j\text{ } \exists n>j \text{ s.t. }a_{m_1,\ldots,m_{k-2},n,i}=1\\
0,\quad\text{otherwise}
\end{cases},
\end{align*}
where $\forall^{\infty}$ means ``for all but a finite number of''. In words, $P$ decides whether the corresponding matrix has a column with infinitely many $1$'s, whereas $Q$ decides whether the matrix has only finitely many columns with only finitely many $1$'s. For $R=P$ or $Q$ consider the problem function for $a\in\Omega_k$
\begin{equation*}
\Xi_{k,R}(a)=\begin{cases}
\exists m_1\text{ } \forall m_2\text{ }\ldots\text{ }\forall m_{k-2} R(a,m_1,\ldots,m_{k-2}),\quad\text{if }k\text{ is even}\\
\forall m_1\text{ } \exists m_2\text{ }\ldots\text{ }\forall m_{k-2} R(a,m_1,\ldots,m_{k-2}),\quad\text{otherwise}
\end{cases},
\end{equation*}
that is, so that all quantifier types alternate.

\begin{theorem}
\label{DST_main}
Let $\mathcal{M}$ be either $\{0,1\}$ with the discrete metric or $[0,1]$ with the usual metric and consider the above problems $\{\Xi_k,\Omega_k,\mathcal{M},\Lambda_k\}$. For $k\in\mathbb{N}_{\geq 2}$ and $R=P$ or $Q$,
$$
\Delta_{k+1}^G\not\ni\{\Xi_{k,R},\Omega_{k},\mathcal{M},\Lambda_k\}\in\Delta_{k+2}^A.
$$
In other words, we can solve the problem via a height $k+1$ arithmetic tower, but it is impossible to do so with a height $k$ general tower. 
\end{theorem}

\begin{remark}
\label{cts_remark}
Note that we allow both \textit{discrete} and \textit{continuous} spaces $\mathcal{M}$, which will be important for our reduction arguments when proving lower bounds for classifications of spectral problems for non-discrete $\mathcal{M}$. The lower bound is a strong result in the sense that it holds regardless of the model of computation. In other words, it is the intrinsic combinatorial complexity of the problems that make the problems hard.\hfill$\boxtimes$
\end{remark}

\begin{proof}
We deal with the case of $R=P$ since the case of $R=Q$ is completely analogous. It is easy to see that $\{\Xi_{k,P},\Omega_{k},\mathcal{M},\Lambda_k\}\in\Delta_{k+2}^A$. First consider the case $k=2$ and set
$$
\Gamma_{n_3,n_2,n_1}(a)=\max_{j\leq n_3}\chi_{(n_2,\infty)}\left(\sum_{i=1}^{n_1} a_{i,j}\right),
$$
where $\chi_C$ denotes the indicator function of a set $C$. This is the decision problem that decides whether there exists a column with index at most $n_3$ such that there are at least $n_2$ $1$'s in the first $n_1$ rows. This is clearly an arithmetic tower and it is straightforward to show that this converges to $\Xi_{2,P}$ in $\mathcal{M}$ (in either of the $\{0,1\}$ and $[0,1]$ cases). For $k>2$ we simply alternate taking products (which corresponds to minima in this case) and maxima. Explicitly, we set
\begin{equation*}
\Gamma_{n_{k+1},\ldots,n_1}(a)=\begin{cases}
\displaystyle\max_{m_{1}\leq n_{k+1}}\displaystyle\prod_{m_{2}=1}^{n_{k}}\cdots\displaystyle\prod_{m_{k-2}=1}^{n_4}\left\{\displaystyle\max_{j\leq n_3}\chi_{(n_2,\infty)}\left(\sum_{i=1}^{n_1} a_{m_1,\ldots,m_{k-2},i,j}\right)\right\},\quad\text{if }k\text{ is even}\\
\displaystyle\prod_{m_{1}=1}^{n_{k+1}}\displaystyle\max_{m_{2}\leq n_{k}}\cdots\displaystyle\prod_{m_{k-2}=1}^{n_4}\left\{\displaystyle\max_{j\leq n_3}\chi_{(n_2,\infty)}\left(\sum_{i=1}^{n_1} a_{m_1,\ldots,m_{k-2},i,j}\right)\right\},\quad\text{otherwise}.
\end{cases}
\end{equation*}
Again, this is an arithmetic tower and it is straightforward to show that this converges to $\Xi_{k,P}$ in $\mathcal{M}$. It also holds that $\{\Xi_{k,P},\Omega_{k},\mathcal{M},\Lambda_k\}\in\Sigma_{k+1}^A$ if $k$ is even and $\{\Xi_{k,P},\Omega_{k},\mathcal{M},\Lambda_k\}\in\Pi_{k+1}^A$ if $k$ is odd (not to be confused with the notation for the Borel hierarchy).

Recall the topology $\mathcal{T}$ on $\Omega_k$ form Theorem \ref{B_equiv}. For the lower bound we note that $P$ is $\Sigma_3^0$ complete (in the literature it is known as the problem ``$S_3$'', see for example \cite[Section 23]{kechris1987descriptive}). Since $(\Omega_{k},\mathcal{T})$ is zero-dimensional, a theorem of Wadge implies that this means that $P$ is the indicator function of a set, also denoted by $P$, which lies in $\Sigma_3^0(\Omega_k)$ but not $\Pi_3^0(\Omega_k)$. It also follows that $\Xi_{k,P}$ is $\Sigma_{k+1}^0(\Omega_k)$ complete if $k$ is even and $\Pi_{k+1}^0(\Omega_k)$ complete otherwise. Now suppose for a contradiction that $\{\Xi_{k,P},\Omega_{k},\mathcal{M},\Lambda_k\}\in\Delta_{k+1}^G$. But then Theorem \ref{B_equiv} implies that $\Xi_{k,P}\in\mathcal{B}_k(\Omega_k,\mathcal{M})$ and hence by Theorem \ref{Baire1}, $\Xi_{k,P}$ is $\Sigma_{k+1}^0(\Omega_k)$ measurable. $\Xi_{k,P}$ is the indicator function of a set, which we denote by $\Xi_{k,P}$ with an abuse of notation, which is either $\Sigma_{k+1}^0(\Omega_k)$ or $\Pi_{k+1}^0(\Omega_k)$ complete depending on the parity of $k$. But $0$ and $1$ are separated in $\mathcal{M}$ and hence since $\Xi_{k,P}$ is $\Sigma_{k+1}^0(\Omega_k)$ measurable, $\Xi_{k,P}$ and its complement both lie in $\Sigma_{k+1}^0(\Omega_k)$. It follows that $\Xi_{k,P}\in\Sigma_{k+1}^0(\Omega_k)\cap\Pi_{k+1}^0(\Omega_k)$, contradicting the stated completeness.
\end{proof}

For our applications to spectral problems, we will use $\tilde{\Omega}$ to denote $\Omega_k$ and consider
\begin{equation}
\begin{split}
\label{canonical_problems}
\tilde{\Xi}_1=\Xi_{2,P},\quad \tilde{\Xi}_2=\Xi_{2,Q},\\
\tilde{\Xi}_3=\Xi_{3,P},\quad \tilde{\Xi}_4=\Xi_{3,Q}.
\end{split}
\end{equation}
Theorem \ref{DST_main} holds for a much wider class of decision problems, but these four are the only ones we shall use in the sequel. The decision problems $\tilde\Xi_1$ and $\tilde\Xi_2$ were shown to have $\mathrm{SCI}_G=3$ in \cite{ben2015can}, but only with regards to the discrete space $\mathcal{M}=\{0,1\}$ and the proof used a somewhat complicated Baire category argument. Theorem \ref{DST_main} is much more general, can be extended to arbitrarily large SCI, and has a much slicker proof, making clear a beautiful connection with the Baire hierarchy for well-behaved $\Omega$.

\section{Proofs Concerning Spectral Radii, Essential Spectral Radii, Capacity and Operator Norms}
\label{first_set_proofsdfhedh}

Here we prove the theorems found in \S \ref{sec:rev:spec_rad} - \ref{sec:rev:CAP}. First, we briefly recall $\Sigma_1^A$ algorithms for spectral problems presented in \cite{colb1}, that are sharp in the SCI hierarchy. The algorithms constructed in \cite{colb1} are shown as pseudocode in Appendix \ref{basic_spec22}, where we also refer the reader to a more detailed account. The following was proven in \cite{colb1} and was generalised in \cite{colbrook3} to unbounded operators:

\begin{theorem}
\label{BASIC}
For each $\Omega_f$ and $\Omega_f\cap\Omega_g$, consider the family $\Lambda$ consisting of $\Lambda_1$, together with pointwise evaluation of $f,\{c_n\}$ (and evaluation of $g$ at rational points if considering $\Omega_f\cap\Omega_g$). The algorithms presented in Appendix \ref{basic_spec22} achieve $\Sigma_1^A$ error control. In particular the following classification holds:
$$
\Delta_1^G\not\owns\{\Xi_1,\Omega_f\cap\Omega_g,\Lambda_1\}\in\Sigma_1^A,\quad
\Delta_1^G\not\owns\{\Xi_2,\Omega_f,\Lambda_1\}\in\Sigma_1^A.
$$
\end{theorem}

We now turn to the proof of Theorem \ref{spec_rad_thm}, dealing first with the evaluation set $\Lambda_1$. Suppose that $\{\tilde{\Gamma}_{n_k,\ldots,n_1}\}$ is a $\Pi^A_k$ tower of algorithms to compute the spectrum of a class of operators, where the output is a finite set for each $n_1,\ldots,n_k$. It is then clear that 
$$
\Gamma_{n_k,\ldots,n_1}(A)=\sup_{z\in\tilde{\Gamma}_{n_k,\ldots,n_1}(A)}\left|z\right|+\frac{1}{2^{n_k}}
$$
provides a $\Pi^A_k$ tower of algorithms for the spectral radius. Strictly speaking, the above may not be an arithmetic tower owing to the absolute value. But it can be approximated to arbitrary precision (from above say), the error of which can be absorbed in the first limit. In what follows, we always assume this is done without further comment. Similarly if $\{\tilde{\Gamma}_{n_k,\ldots,n_1}\}$ provides a $\Sigma^A_k$ tower of algorithms for the spectrum (and outputs a finite set for each $n_1,\ldots,n_k$),
$$
\Gamma_{n_k,\ldots,n_1}(A)=\sup_{z\in\tilde{\Gamma}_{n_k,\ldots,n_1}(A)}\left|z\right|-\frac{1}{2^{n_k}}
$$
provides a $\Sigma^A_k$ tower of algorithms for the spectral radius. If we only have a height $k$ tower with no $\Sigma_k$ or $\Pi_k$ type error control for the spectrum, then taking the supremum of absolute values shows that we get a height $k$ tower for the spectral radius.

The fact that $\{\Xi_r,\Omega_\mathrm{D}\} \in \Sigma^A_1$, $\{\Xi_r,\Omega_f\cap\Omega_g\} \in \Sigma^A_1$, $\{\Xi_r,\Omega_g\} \in \Sigma^A_2$, $\{\Xi_r,\Omega_f\} \in \Pi^A_2$ and $\{\Xi_r,\Omega_{\mathrm{B}}\} \in \Pi^A_3$ hence follow from Theorems \ref{BASIC} and the results of \cite{ben2015can}. It is clear that $\{\Xi_r,\Omega_\mathrm{D}\}\notin\Delta^G_1$ and this also shows that $\{\Xi_r,\Omega_\mathrm{N}\}\notin\Delta^G_1$ and $\{\Xi_r,\Omega_f\cap\Omega_g\}\notin\Delta^G_1$. Hence, we must show the positive result that $\{\Xi_r,\Omega_\mathrm{N}\} \in \Sigma^A_1$ and prove the lower bounds $\{\Xi_r,\Omega_g\}\notin\Delta^G_2$, $\{\Xi_r,\Omega_f\}\notin\Delta^G_2$ and $\{\Xi_r,\Omega_{\mathrm{B}}\}\notin\Delta^G_3$.

\begin{proof}[Proof of Theorem \ref{spec_rad_thm} for $\Lambda_1$]
Throughout this proof we use the evaluation set $\Lambda_1$, which we drop from the notation for convenience.

\textbf{Step 1:} $\{\Xi_r,\Omega_\mathrm{N}\} \in \Sigma^A_1$. Recall that the spectral radius of a normal operator $A\in\Omega_{\mathrm{B}}$ is equal to its operator norm. Consider the finite section matrices $P_nAP_n\in\mathbb{C}^{n\times n}$. It is straightforward to show that
$$
\left\|P_nAP_n\right\|\uparrow \left\|A\right\|\quad\text{ as }n\rightarrow\infty.
$$
The norm $\left\|P_nAP_n\right\|$ is the square root of the largest eigenvalue of the semi-positive definite self-adjoint matrix $(P_nAP_n)^*(P_nAP_n)$. This can be estimated from below to an accuracy of $1/n$ using Corollary 6.9 of \cite{colbrook3}, which then yields a $\Sigma_1^A$ algorithm for $\{\Xi_r,\Omega_\mathrm{N}\}$.

\textbf{Step 2:} $\{\Xi_r,\Omega_g\}\notin\Delta^G_2$. Recall that we assumed the existence of a $\delta\in(0,1)$ such that $g(x)\leq (1-\delta)x$. Let $\epsilon>0$, then it is easy to see that the matrices
$$
S_{\pm}(\epsilon)=\begin{pmatrix}
1 & 0 \\
\pm \epsilon & 1
\end{pmatrix}
$$
have norm bounded by $1+\epsilon+\epsilon^2$ and are clearly inverse of each other. Choose $\epsilon$ small such that $(1+\epsilon+\epsilon^2)^2\leq1/(1-\delta)$. If $B\in\mathbb{C}^{2\times 2}$ is normal, it follows that $\hat{B}:=S_{+}(\epsilon)BS_{-}(\epsilon)$ lies in $\Omega_g$ and has the same spectrum as $B$. We choose
$$
\hat{B}=S_{+}(\epsilon)\begin{pmatrix}
1 & -\epsilon \\
-\epsilon & 0
\end{pmatrix}S_{-}(\epsilon)=\begin{pmatrix}
1+\epsilon^2 & -\epsilon \\
\epsilon^3 & -\epsilon^2
\end{pmatrix}.
$$
The crucial property of $\hat{B}$ is that the first entry $1+\epsilon^2$ is strictly greater in magnitude than the two eigenvalues $(1\pm\sqrt{1+4\epsilon^2})/2$.

Now suppose for a contradiction that a height one tower, $\{\Gamma_n\}$, solves the problem. We will gain a contradiction by showing that $\Gamma_n(A)$ does not converge for an operator of the form,
$$
A=\bigoplus_{r=1}^\infty A_{l_r},\quad A_{m}:=\begin{pmatrix}
1+\epsilon^2& & & &-\epsilon\\
 &0& & & \\
 & &\ddots& & \\
 & & &0& \\
\epsilon^3& & & &-\epsilon^2\\
\end{pmatrix}
\in\mathbb{C}^{m\times m},
$$
where we only consider $l_k\geq 3$. Each $A_m$ is unitarily equivalent to the matrix $\hat{B}\oplus 0\in\mathbb{C}^{m\times m}$ and has spectrum equal to $\{0,(1\pm\sqrt{1+4\epsilon^2})/2\}$. Any $A$ of the above form is unitarily equivalent to a direct sum of an infinite number of $\hat{B}$'s and the zero operator and hence lies in $\Omega_g$. Now suppose that $l_1,\ldots,l_k$ have been chosen and consider the operator
$$
B_k=A_{l_1}\oplus\dots\oplus A_{l_k}\oplus C,\quad C=\mathrm{diag}\{1+\epsilon^2,0,\ldots\}.
$$
The spectrum of $B_k$ is $\{0,(1\pm\sqrt{1+4\epsilon^2})/2,1+\epsilon^2\}$ and hence there exists $\eta>0$ and $n(k)\geq k$ such that $\Gamma_{n(k)}(B_k)>(1+\sqrt{1+4\epsilon^2})/2+\eta$. But $\Gamma_{n(k)}(B_k)$ can only depend on the evaluations of the matrix entries $\{B_k\}_{ij}=\langle B_ke_j,e_i \rangle$ with $i,j\leq N(B_k,n(k))$ (as well as evaluations of the function $g$) into account. If we choose $l_{k+1}>N(B_k,n(k))$ then by the assumptions in Definition \ref{Gen_alg}, $\Gamma_{n(k)}(A)=\Gamma_{n(k)}(B_k)>(1+\sqrt{1+4\epsilon^2})/2+\eta$. But $\Gamma_n(A)$ must converge to $(1+\sqrt{1+4\epsilon^2})/2$, a contradiction.

\textbf{Step 3:} $\{\Xi_r,\Omega_f\}\notin\Delta^G_2$. Suppose for a contradiction that a height one tower, $\{\Gamma_n\}$, solves the problem. We will gain a contradiction by showing that $\Gamma_n(A)$ does not converge for an operator of the form
$$
A=\bigoplus_{r=1}^\infty C_{l_r}\oplus A_{l_r},\quad A_{m}:=\begin{pmatrix}
0& 1& & &\\
 &0& 1& & \\
 & &\ddots& \ddots& \\
 & & && 1\\
& & & &0\\
\end{pmatrix}
\in\mathbb{C}^{m\times m},\quad C_m=\mathrm{diag}\{0,0,\ldots,0\}\in\mathbb{C}^{m\times m},
$$
where we assume that $l_r\geq r$ to ensure that the spectrum of $A$ is equal to the unit disc $B_1(0)$. Note that the function $f(n)=n+1$ will do for the bounded dispersion with $c_n=0$. Now suppose that $l_1,\ldots,l_k$ have been chosen and consider the operator
$$
B_k=\big(C_{l_1}\oplus A_{l_1}\big)\oplus\cdots\oplus \big(C_{l_k}\oplus A_{l_k}\big)\oplus C,\quad C=\mathrm{diag}\{0,0,\ldots\}.
$$
The spectrum of $B_k$ is $\{0\}$ and hence there exist $n(k)\geq k$ such that $\Gamma_{n(k)}(B_k)<1/4$. But $\Gamma_{n(k)}(B_k)$ can only depend on the evaluations of the matrix entries $\{B_k\}_{ij}=\langle B_ke_j,e_i \rangle$ with $i,j\leq N(B_k,n(k))$ (as well as evaluations of the function $f$) into account. If we choose $l_{k+1}>N(B_k,n(k))$ then by the assumptions in Definition \ref{Gen_alg}, $\Gamma_{n(k)}(A)=\Gamma_{n(k)}(B_k)<1/4$. But $\Gamma_n(A)$ must converge to $1$, a contradiction.

\textbf{Step 4:} $\{\Xi_r,\Omega_{\mathrm{B}}\}\notin\Delta^G_3$. Suppose for a contradiction that $\{\Gamma_{n_2,n_1}\}$ is a height two (general) tower and without loss of generality, assume it to be non-negative. We use the results of \S\ref{bigHth}. Let $(\mathcal{M},d)$ be the space $[0,1]$ with the usual metric (note in particular this is not discrete so we use Remark \ref{cts_remark}), let $\tilde\Omega$ denote the collection of all infinite matrices $\{a_{i,j}\}_{i,j\in\mathbb{N}}$ with entries $a_{i,j}\in\{0,1\}$ and recall the problem function
\begin{equation*}
\tilde\Xi_1(\{a_{i,j}\}):\text{ Does }\{a_{i,j}\}\text{ have a column containing infinitely many non-zero entries?}
\end{equation*}
Theorem \ref{DST_main} in \S\ref{bigHth} shows that $\mathrm{SCI}(\tilde\Xi_1,\tilde\Omega)_{G} = 3$. We will gain a contradiction by using the supposed height two tower to solve $\{\tilde\Xi_1,\tilde\Omega\}$.

Without loss of generality, identify $\Omega_{\mathrm{B}}$ with $\mathcal{B}(X)$ where $X=\bigoplus_{j=1}^{\infty}X_j$ in the $l^2$-sense with $X_j=l^2(\mathbb{N})$. Now let $\{a_{i,j}\}\in\tilde\Omega$ and define $B_j\in\mathcal{B}(X_j)$ with the matrix representation
$$
(B_{j})_{k,i}= \begin{cases}
        1, & \text{if } k=i\text{ and }a_{k,j}=0\\
        1, & \text{if } k<i\text{ and }a_{l,j}=0\text{ for }k<l<i\\
        0, & \text{otherwise } 0\leq n\leq 1.
        \end{cases}
$$
Let $\mathcal{I}_j$ be the index set of all $i$ where $a_{i,j}=1$. $B_j$ acts as a unilateral shift on $\overline{\mathrm{span}}\{e_k:k\in \mathcal{I}_j\}$ and the identity on its orthogonal complement. It follows that
$$
\mathrm{Sp}(B_j)= \begin{cases}
        1, & \text{if } \mathcal{I}_j=\emptyset\\
        \{0,1\}, & \text{if } \mathcal{I}_j\text{ is finite and non-empty}\\
        \mathbb{D}\quad(\text{the unit disc}),& \text{if } \mathcal{I}_j\text{ is infinite}.
        \end{cases}
$$
For the matrix $\{a_{i,j}\}$ define $A\in\Omega_{\mathrm{B}}$ by
$$
A=\bigoplus_{j=1}^{\infty}(B_j-\frac{1}{2}I_j),
$$
where $I_j$ denotes the identity operator on $\mathbb{C}^{j\times j}$, then $\mathrm{Sp(A)}=\overline{\cup_{j=1}^\infty\mathrm{Sp}(B_j)}-\frac{1}{2}$.

Hence we see that
$$
\Xi_r(A)= \begin{cases}
        \frac{1}{2}, & \text{if }\tilde\Xi_1(\{a_{i,j}\})=0\\
        \frac{3}{2}, & \text{if }\tilde\Xi_1(\{a_{i,j}\})=1.
        \end{cases}
$$
We then set $\tilde{\Gamma}_{n_2,n_1}(\{a_{i,j}\})=\min\{\max\{\Gamma_{n_2,n_1}(A)-1/2,0\},1\}$. It is clear that this defines a generalised algorithm mapping into $[0,1]$. In particular, given $N$ we can evaluate $\{A_{k,l}:k,l\leq N\}$ using only finitely many evaluations of $\{a_{i,j}\}$, where we can use a bijection between canonical bases of $l^2(\mathbb{N})$ and $\bigoplus_{j=1}^{\infty}X_j$ to view $A$ as acting on $l^2(\mathbb{N})$. But then $\{\tilde\Gamma_{n_2,n_1}\}$ provides a height two tower for $\{\tilde\Xi_1,\tilde\Omega\}$, a contradiction.
\end{proof}

\begin{remark}
The algorithm in step 1 of the above proof works for any operator whose operator norm is equal to the spectral radius. If, instead, the operator is spectraloid, meaning the spectral radius is equal to the numerical radius
$$
w(A):=\sup\{\left|\left\langle Ax,x\right\rangle\right|:\left\|x\right\|=1\},
$$
then a similar argument will hold by estimating $w(P_nAP_n)$. To do this, we need a way of computing $w(A)$ to a given accuracy using finitely many arithmetic operations and comparisons (e.g., Lemma \ref{num_ran_fs} below).\hfill$\boxtimes$
\end{remark}

\begin{proof}[Proof of Theorem \ref{spec_rad_thm} for $\Lambda_2$]
Here we prove the changes for $\Xi_r$ when we consider the evaluation set $\Lambda_2$. It is clear that the classifications in $\Sigma_1^A$ do not change. It is also easy to use the algorithm in Theorem \ref{BASIC} (now using $\Lambda_2$ to collapse the first limit and approximate $\gamma_n$ - see Appendix \ref{basic_spec22}) to prove $\{\Xi_r,\Omega_g,\Lambda_2\}\in\Sigma_1^A$. Similarly we can use the algorithm for the spectrum of operators in  $\Omega_f$ for $\Omega_{\mathrm{B}}$ using $\Lambda_2$ to collapse the first limit and hence $\{\Xi_r,\Omega_{\mathrm{B}},\Lambda_2\}\in\Pi_2^A$. Since $\Omega_f\subset \Omega_{\mathrm{B}}$, it follows that we only need to prove $\{\Xi_r,\Omega_f,\Lambda_2\}\not\in\Delta_2^G$. This can be proven using exactly the same example and a similar argument to step 3 of the proof of Theorem \ref{spec_rad_thm} (hence omitted).
\end{proof}

\begin{proof}[Proof of Theorem \ref{spec_rad_thm2}]
We begin by proving the results for $\Lambda_1$. For the lower bounds, it is enough to show that $\{\Xi_{er},\Omega_\mathrm{D},\Lambda_1\}\not\in\Delta_2^G$ and $\{\Xi_{er},\Omega_{\mathrm{B}},\Lambda_1\}\not\in\Delta_3^G$. For the upper bounds, we must show that $\{\Xi_{er},\Omega_f,\Lambda_1\}\in\Pi_2^A$, $\{\Xi_{er},\Omega_{\mathrm{B}},\Lambda_1\}\in\Pi_3^A$ and $\{\Xi_{er},\Omega_\mathrm{N},\Lambda_1\}\in\Pi_2^A$. The lower bounds for $\Lambda_2$ follow from $\{\Xi_{er},\Omega_\mathrm{D},\Lambda_1\}\not\in\Delta_2^G$ and for the upper bounds it is enough to prove $\{\Xi_{er},\Omega_{\mathrm{B}},\Lambda_2\}\in\Pi_2^A$.

\textbf{Step 1:} $\{\Xi_{er},\Omega_\mathrm{D},\Lambda_1\}\not\in\Delta_2^G$. This is the same argument as in step 3 of the proof of Theorem \ref{spec_rad_thm}, however now we replace $A_m$ by $A_m=\mathrm{diag}\{1,1,\ldots,1\}\in\mathbb{C}^{m\times m}$ and use the fact that $\Xi_{er}(B_k)=0$. It follows that given the proposed height one tower $\{\Gamma_n\}$ and the constructed $A$, $\Xi_{er}(A)=1$ but $\Gamma_{n(k)}(A)<1/4$, the required contradiction.

\textbf{Step 2:} $\{\Xi_{er},\Omega_{\mathrm{B}},\Lambda_1\}\not\in\Delta_3^G$. This is the same argument as step 4 of the proof of Theorem \ref{spec_rad_thm}.

\textbf{Step 3:} $\{\Xi_{er},\Omega_f,\Lambda_1\}\in\Pi_2^A$, $\{\Xi_{er},\Omega_{\mathrm{B}},\Lambda_1\}\in\Pi_3^A$ and $\{\Xi_{er},\Omega_{\mathrm{B}},\Lambda_2\}\in\Pi_2^A$. $\{\Xi_{er},\Omega_f,\Lambda_1\}\in\Pi_2^A$ follows immediately from the existence of a $\Pi_2^A$ tower of algorithms for the essential spectrum of operators in $\Omega_f$ proven in \cite{ben2015can}. The output of this tower is a finite collection of rectangles with complex rational vertices, hence we can gain an approximation of the maximum absolute value over this output to any given precision. This can be used to construct a $\Pi_2^A$ tower for $\{\Xi_{er},\Omega_f,\Lambda_1\}$. Similarly, $\{\Xi_{er},\Omega_{\mathrm{B}},\Lambda_1\}\in\Pi_3^A$ follows from the $\Pi_3^A$ tower of algorithms for $\{\mathrm{Sp}_{\mathrm{ess}},\Omega_{\mathrm{B}},\Lambda_1\}$ constructed in \cite{ben2015can}. Finally, we can use $\Lambda_2$ to collapse the first limit of the algorithm for the essential spectrum in \cite{ben2015can}, giving a $\Pi_2^A$ algorithm and this can be used to show $\{\Xi_{er},\Omega_{\mathrm{B}},\Lambda_2\}\in\Pi_2^A$.

\textbf{Step 4:} $\{\Xi_{er},\Omega_\mathrm{N},\Lambda_1\}\in\Pi_2^A$.
A $\Pi_2^A$ tower is constructed in the proof of Theorem \ref{spec_poll_hard} for the essential numerical range, $W_e(A)$, of normal operators (using $\Lambda_1$) and this outputs a finite collection of points. For normal operators $A$, $W_e(A)$ is the convex hull of the essential spectrum and hence $\sup_{z\in W_e(A)}\left|z\right|$ is equal to $\Xi_{er}(A)$. Hence a $\Pi_2^A$ tower for $\{\Xi_{er},\Omega_\mathrm{N},\Lambda_1\}$ follows by taking the maximum absolute value over the tower for $W_e(A)$.
\end{proof}

\begin{proof}[Proof of Theorem \ref{spec_rad_thm3}]
Note that given a height $k$ arithmetical tower $\{\widehat\Gamma_{n_k,\ldots,n_1}(\cdot,p)\}$ for $\Xi_{r,p}$ and a class $\Omega'$, we can build a $\Pi_{k+1}^A$ tower for $\{\Xi_{cap},\Omega'\}$ as follows. Let $p_1,p_2,\ldots$ be an enumeration of the monic polynomials with rational coefficients and $\tilde\Gamma_{n_k,\ldots,n_1}(\cdot,p)$ be an approximation to $|\widehat\Gamma_{n_k,\ldots,n_1}(\cdot,p)|^{1/\mathrm{deg}(p)}$ to accuracy $1/n_1$ using finitely many arithmetic operations and comparisons. Define
$$
\Gamma_{n_{k+1},\ldots,n_1}(A)=\min_{1\leq m\leq n_{k+1}}\tilde\Gamma_{n_k,\ldots,n_1}(A,p_m).
$$
The fact that this is a convergent $\Pi_{k+1}^A$ tower is clear. This, together with inclusions of the considered classes of operators, means that to prove the positive results we only need to prove $\{\Xi_{r,p},\Omega_f,\Lambda_1\}\in\Sigma_1^A$, $\{\Xi_{r,p},\Omega_{\mathrm{B}},\Lambda_1\}\in\Sigma_2^A$ and $\{\Xi_{r,p},\Omega_{\mathrm{B}},\Lambda_2\}\in\Sigma_1^A$. Likewise, for the negative results we only need to prove $\{\Xi_{cap},\Omega_{\mathrm{D}},\Lambda_2\}\not\in\Delta^G_2$ (the fact that $\{\Xi_{r,p},\Omega_{\mathrm{D}},\Lambda_2\}\not\in\Delta^G_1$ is obvious), $\{\Xi_{cap},\Omega_{\mathrm{N}},\Lambda_1\}\not\in\Delta^G_3$ and $\{\Xi_{r,p},\Omega_{\mathrm{N}},\Lambda_2\}\not\in\Delta^G_2$. We shall prove these results with $\Omega_{\mathrm{N}}$ replaced by the class of self-adjoint bounded operators denoted by $\Omega_{\mathrm{SA}}$.

\textbf{Step 1:} $\{\Xi_{r,p},\Omega_f,\Lambda_1\}\in\Sigma_1^A$. The function $f$ and sequence $\{c_n\}$ allows us to compute the matrix elements of $p(A)$ for any $A\in\Omega_f$ and polynomial $p$ to arbitrary accuracy. We can then use the same argument as step 1 of the proof of Theorem \ref{spec_rad_thm}, approximating $\|P_np(A)P_n\|$ instead of $\left\|P_nAP_n\right\|$.

\textbf{Step 2:} $\{\Xi_{r,p},\Omega_{\mathrm{B}},\Lambda_1\}\in\Sigma_2^A$ and $\{\Xi_{r,p},\Omega_{\mathrm{B}},\Lambda_2\}\in\Sigma_1^A$. For the first result, we note that
$$
\lim_{m\rightarrow\infty}\|P_np(P_{m}AP_m)P_n\|=\|P_np(A)P_n\|
$$
and let $\Gamma_{n,m}(A,p)$ be an approximation of $\|P_np(P_{m}AP_m)P_n\|$ to accuracy $1/m$, which can be computed in finitely many arithmetic operations and comparisons. To prove $\{\Xi_{r,p},\Omega_{\mathrm{B}},\Lambda_2\}\in\Sigma_1^A$, for any given $A\in\Omega_{\mathrm{B}}$ we can use $\Lambda_2$ to compute a function $f_A$ and sequence $\{c_n(A)\}$ bounding the dispersion such that $A\in\Omega^{f_A}$ and use step 1.

\textbf{Step 3}: $\{\Xi_{cap},\Omega_\mathrm{SA},\Lambda_1\}\notin\Delta^G_3$. Suppose for a contradiction that $\{\Gamma_{n_2,n_1}\}$ is a height two (general) tower for the problem and without loss of generality, assume it to be non-negative. Our strategy will be as in the proof of Theorem \ref{spec_rad_thm} (recall also the results of \S\ref{bigHth}). Let $(\mathcal{M},d)$ be the space $[0,1]$ with the usual metric (note in particular this is not discrete so we use remark \ref{cts_remark}), let $\tilde{\Omega}$ denote the collection of all infinite matrices $\{a_{i,j}\}_{i,j\in\mathbb{N}}$ with entries $a_{i,j}\in\{0,1\}$ and consider the problem function
\begin{equation*}
\tilde{\Xi}_2(\{a_{i,j}\}):\text{ Does }\{a_{i,j}\}\text{ have (only) finitely many columns with (only) finitely many 1's?}
\end{equation*}
Recall that it was shown in \S \ref{bigHth} that $\mathrm{SCI}(\tilde{\Xi}_2,\tilde{\Omega})_{G} = 3$. We will gain a contradiction by using the supposed height two tower to solve $\{\tilde{\Xi}_2,\tilde{\Omega}\}$. Without loss of generality, identify $\Omega_\mathrm{SA}$ with self adjoint operators in $\mathcal{B}(X)$ where $X=\bigoplus_{j=1}^{\infty}X_j$ in the $l^2$-sense with $X_j=l^2(\mathbb{N})$. To proceed, we need the following elementary lemma, which will be useful in constructing examples of spectral pollution.

\begin{lemma}
\label{useful_spec_poll}
Let $z_1,z_2,\ldots,z_k\in[-1,1]$ and let $a_{j}=\sqrt{1-z_j^2}$ (say positive square root). Then the symmetric matrix
$$
B(z_1,\ldots,z_k)=\left(\begin{array}{ccccc|ccccc}
z_1& 0& \cdots& &  &        a_1& 0& \cdots& &                                 \\
0 &z_2& 0 & \cdots& &          0 &a_2& 0 & \cdots& \\
\vdots& 0&\ddots& & &          \vdots& 0&\ddots& & \\
&\vdots& & &   &                   &\vdots& & &\\
& & & &z_k                &  & & & &a_k\\
\hline
a_1& 0& \cdots& &   &         -z_1& 0& \cdots& &                                 \\
0 &a_2& 0 & \cdots&  &        0 &-z_2& 0 & \cdots& \\
\vdots& 0&\ddots& &    &      \vdots& 0&\ddots& & \\
&\vdots& & &   &                   &\vdots& & &\\
& & & &a_k                   & & & & &-z_k
\end{array}\right)
\in\mathbb{C}^{2k\times 2k}
$$
has eigenvalues $\pm 1$ (repeated $k$ times).
\end{lemma}
\begin{proof}
By a change of basis, the above matrix is equivalent to a block diagonal matrix with blocks
$$
\begin{pmatrix}
z_j & a_j\\
a_j& -z_j
\end{pmatrix}.
$$
These blocks have eigenvalues $\{-1,1\}$.
\end{proof}

Now choose a sequence of rational numbers $\{z_j\}_{j\in\mathbb{N}}\in[-1,1]$ that is also dense in $[-1,1]$ and let $B_j=B(z_1,\ldots,z_j)$. For each column of a given $\{a_{i,j}\}\in\tilde\Omega$, let the infinite matrix $C^{(j)}$ be defined as follows. If $k,l<j+1$ then $C^{(j)}_{kl}=z_{k}\delta_{k,l}$. Let $r(i)$ denote the row of the $i$th one of the column $\{a_{i,j}\}_{i\in\mathbb{N}}$ (with $r(i)=\infty$ if $\sum_{m}a_{m,j}<i$ and $r(0)=0$). If $r(i)<\infty$ then for $k\leq l$ define
$$
C^{(j)}_{kl}=\begin{cases}
        a_p\delta_{k,l-(r(i)-r(i-1)-1)}, & p=1,\ldots,j,l=r(i)+j\cdot (2i-1)+p-1\\
        -z_p\delta_{k,l}, & p=1,\ldots,j,l=r(i)+j\cdot (2i-1)+p-1\\
				z_p\delta_{k,l}, & p=1,\ldots,j,l=r(i)+2j\cdot i+p-1\\
				0, & \text{otherwise},
        \end{cases}
$$
and extend $C^{(j)}_{kl}$ below the diagonal to a symmetric matrix. The key property of this matrix is that if the column $\{a_{i,j}\}_{i\in\mathbb{N}}$ has infinitely many 1s, then its is unitarily equivalent to an infinite direct sum of infinitely many $B_j$ together with the zero operator acting on some subspace (whose dimension is equal to the number of zeros in the column). In this case $\mathrm{Sp}(C^{(j)})=\{-1,1,0\}$ or $\{-1,1\}$. On the other hand, if $\{a_{i,j}\}_{i\in\mathbb{N}}$ has finitely many 1s, then $C^{(j)}$ is unitarily equivalent the direct sum of a finite number of $B_j$, the diagonal operator $\mathrm{diag}\{z_1,\ldots,z_k\}$ and the zero operator acting on some subspace. In this case $\{z_1,\ldots,z_j\}\subset\mathrm{Sp}(C^{(j)})$. Let $A=\bigoplus_{j=1}^{\infty}C^{(j)}$, then it is clear that if $\tilde\Xi_2(\{a_{i,j}\})=1$, then $\mathrm{Sp}(A)$ is a finite set, otherwise it is the entire interval $[-1,1]$.

Now we use the following facts for bounded self-adjoint operators $A$. If $\mathrm{Sp}(A)$ is a finite set then $\Xi_{cap}(A)=0$ whereas if $\mathrm{Sp}(A)=[-1,1]$ then $\Xi_{cap}(A)=1/2$ (this can be proven easily using the minimal $l^\infty$ norm property of monic Chebyshev polynomials). We then define $\tilde{\Gamma}_{n_2,n_1}(\{a_{i,j}\})=\min\{\max\{1-2\Gamma_{n_2,n_1}(A),0\},1\}$. It is clear that this defines a generalised algorithm. In particular, given $N$ we can evaluate $\{A_{k,l}:k,l\leq N\}$ using only finitely many evaluations of $\{a_{i,j}\}$, where we can use a bijection between canonical bases of $l^2(\mathbb{N})$ and $\bigoplus_{j=1}^{\infty}X_j$ to view $A$ as acting on $l^2(\mathbb{N})$. We also have the convergence $\lim_{n_2\rightarrow\infty}\lim_{n_1\rightarrow\infty}\tilde{\Gamma}_{n_2,n_1}(\{a_{i,j}\})=\tilde\Xi_2(\{a_{i,j}\})$, a contradiction.

\textbf{Step 4:} $\{\Xi_{cap},\Omega_{\mathrm{D}},\Lambda_2\}\not\in\Delta^G_2$. This is the same argument as in step 3 of the proof of Theorem \ref{spec_rad_thm}. However, we now replace $A_m$ by $A_m=\mathrm{diag}\{d_1,d_2,\ldots,d_m\}\in\mathbb{C}^{m\times m}$, where $\{d_m\}$ is a dense subsequence of $[-1,1]$, and use the fact that $\Xi_{cap}(B_k)=0$. It follows that given the proposed height one tower $\{\Gamma_n\}$ and the constructed $A$, $\Xi_{cap}(A)=1/2$ but $\Gamma_{n(k)}(A)<1/4$, the required contradiction.

\textbf{Step 5:} $\{\Xi_{r,p},\Omega_{\mathrm{SA}},\Lambda_2\}\not\in\Delta^G_2$. Recall that we are given some polynomial $p$ of degree at least two. We assume without loss of generality that the zeros of $p$ are $\pm1$ and $\left|p(0)\right|>1$ (the more general case is similar). The argument is similar to step 3 of the proof of Theorem \ref{spec_rad_thm}, but we spell it out since it uses Lemma \ref{useful_spec_poll}. Suppose for a contradiction that a height one tower, $\{\Gamma_n\}$, solves the problem. We will gain a contradiction by showing that $\Gamma_n(A)$ does not converge for an operator of the form,
$$
A=\bigoplus_{r=1}^\infty B(z_1,\ldots,z_{l_r}),
$$
and define
$$
C=\mathrm{diag}\{z_1,z_2,\ldots\}\in\Omega_{\mathrm{B}}.
$$
We assume that $l_r\geq r$ to ensure that the spectrum of $A$ is equal to $\{-1,1\}$ and hence $\Xi_{r,p}(A)=0$. Now suppose that $l_1,\ldots,l_k$ have been chosen and consider the operator
$$
B_k=B(z_1)\oplus\cdots\oplus B(z_1,\ldots,z_{l_k})\oplus C.
$$
The spectrum of $B_k$ is $[-1,1]$ so that $\Xi_{r,p}(B_k)>1$ and hence there exists $n(k)\geq k$ such that $\Gamma_{n(k)}(B_k)>1/4$. But $\Gamma_{n(k)}(B_k)$ can only depend on the evaluations of the matrix entries $\{B_k\}_{ij}=\langle B_ke_j,e_i \rangle$ with $i,j\leq N(B_k,n(k))$ (as well as evaluations of the function $f$) into account. If we choose $l_{k+1}>N(B_k,n(k))$ then by the assumptions in Definition \ref{Gen_alg}, $\Gamma_{n(k)}(A)=\Gamma_{n(k)}(B_k)>1/4$. But $\Gamma_n(A)$ must converge to $0$, a contradiction.
\end{proof}

\begin{remark}[Efficiently computing the capacity]
\label{capacity_efficiency}
Listing the monic polynomials with rational coefficients in the above proof is very inefficient. In practice, it is much better to split the domain of interest into intervals (or squares if in the complex plane, but we stick to the self-adjoint case in the following discussion). Suppose that each interval has dyadic endpoints and a diameter of $2^{-n_2}$ and that our operator is self-adjoint with known bounded dispersion. One can then apply Lemma \ref{halt_test} (denoting the index of that tower by $n_1$) to obtain an interval covering of the spectrum which will converge as $n_1\rightarrow\infty$, modulo the possibility of isolated points of the spectrum located at the endpoints of the intervals. Since the capacity of a compact set is unaltered by adding finitely many points, we do not have to worry about the endpoints - the limit of the capacity of this covering as $n_1\rightarrow\infty$ will be the capacity of a covering of the spectrum. As $n_2\rightarrow\infty$, we can use the fact that capacity is right-continuous as a set function (for compact sets $E_n,E$ with $E_n\downarrow E$, one has $\mathrm{cap}(E_n)\downarrow\mathrm{cap}(E)$) to obtain a $\Pi_2^A$ algorithm. The point of this is that it reduces the computation of the resulting tower $\{\Gamma_{n_2,n_1}\}$ to computing the capacity of finite unions of disjoint closed intervals in $\mathbb{R}$. In our computational examples, we made use of the method in \cite{liesen2017fast}, which uses conformal mappings and can deal with thousands of intervals.\hfill$\boxtimes$
\end{remark}

\section{Proofs Concerning Essential Numerical Ranges, Essential Spectra and Spectral Pollution}
\label{gowgknk}

\begin{proof}[Proof of Theorem \ref{spec_poll_hard} for $\Xi_{we}$] For the lower bounds, it is enough to note that $\{\Xi_{we},\Omega_\mathrm{D},\Lambda_2\}\not\in\Delta_2^G$ by the same argument as step 1 of the proof of Theorem \ref{spec_rad_thm2}. The construction is exactly the same but yields $d_{\mathrm{H}}(\Gamma_{n(k)}(A),\{0\})\leq1/2$, whereas $\Xi_{we}(A)=[0,1]$. Hence the proposed height one tower cannot converge. To construct a $\Pi_2^A$ tower for general operators, we need the following Lemma:

\begin{lemma}
\label{num_ran_fs}
Let $B\in\mathbb{C}^{n\times n}$ and $\epsilon>0$. Then using finitely many arithmetic operations and comparisons, we can compute points $z_1,\ldots,z_k\in\mathbb{Q}+i\mathbb{Q}$ such that
$$
d_{\mathrm{H}}(\{z_1,\ldots,z_k\},W(B))\leq \epsilon.
$$
\end{lemma}
\begin{proof}
Recall from step 1 of the proof of Theorem \ref{spec_rad_thm} that we can compute an upper bound $M\in\mathbb{Q}_+$ for $\|B\|$ in finitely many arithmetic operations and comparisons. Now choose points $x_1,\ldots,x_k\in\mathbb{Q}^n$, each of norm at most $1$, such that $d_{\mathrm{H}}(\{x_1,\ldots,x_k\},\{x\in\mathbb{C}^n:\|x\|=1\})<\epsilon/(3M)$. These can be computed in finitely many arithmetic operations and comparisons using generalised polar coordinates and approximations of trigonometric identities. It follows that
$$
d_{\mathrm{H}}(\{\langle Bx_1,x_1\rangle,\ldots,\langle Bx_k,x_k\rangle\},W(B))\leq2\epsilon/3.
$$
We then let each $z_j\in \mathbb{Q}+i\mathbb{Q}$ be a $\epsilon/4$ approximation of $\langle Bx_j,x_j\rangle$, which can be computed in finitely many arithmetic operations and comparisons.
\end{proof}

\begin{remark}[Efficient computation]
In practice, there are much more efficient methods of computation. For example, the method of Johnson \cite{johnson1978numerical}, reduces the computation of $W(A)$ for $A\in\mathbb{C}^{n\times n}$ to a series of $n\times n$ Hermitian eigenvalue problems.\hfill$\boxtimes$
\end{remark}

It is well-known that for $A\in\Omega_{\mathrm{B}}$,
\begin{align*}
\overline{W(P_nA|_{P_n\mathcal{H}})}&\uparrow \overline{W(A)},\\
\overline{W((I-P_n)A|_{(I-P_n)\mathcal{H}})}&\downarrow W_{e}(A).
\end{align*}
Given $A$, let $\Gamma_{n_2,n_1}(A)$ be a finite collection of points produced by the algorithm in Lemma \ref{num_ran_fs} applied to $B=(I-P_{n_2})P_{n_1+n_2+1}A|_{P_{n_1+n_2+1}(I-P_{n_2})\mathcal{H}}$ and $\epsilon=1/n_1$. The above limits show that $\{\Gamma_{n_2,n_1}\}$ provides a $\Pi_2^A$ tower for $\{\Xi_{er},\Omega_{\mathrm{B}},\Lambda_1\}$.
\end{proof}

\begin{proof}[Proof of Theorem \ref{spec_poll_hard} for $\Xi_{poll}^{\mathbb{F}}$] We will prove that $\{\Xi_{poll}^{\mathbb{R}},\Omega_\mathrm{D},\Lambda_i\}\not\in\Delta^G_3$ and $\{\Xi_{poll}^{\mathbb{C}},\Omega_{\mathrm{B}},\Lambda_1\}\in\Sigma^A_3$. The construction of towers for $\Xi_{poll}^{\mathbb{R}}$ are similar, as are the arguments for lower bounds.

\textbf{Step 1:} $\{\Xi_{poll}^{\mathbb{C}},\Omega_{\mathrm{B}},\Lambda_1\}\in\Sigma^A_3$. Let $\{\tilde \Gamma_{n_2,n_1}\}$ be the $\Pi_2^A$ tower for $\{\Xi_{er},\Omega_{\mathrm{B}},\Lambda_1\}$ constructed above. Recall the definition
$$
\gamma_{n_2,n_1}(z;A)=\min\{\sigma_{\mathrm{inf}}(P_{n_1}(A-zI){|_{P_{n_2}\mathcal{H}}}),\sigma_{\mathrm{inf}}(P_{n_1}(A^*-\bar{z}I){|_{P_{n_2}\mathcal{H}}})\}
$$ 
and that this can be approximated to any given accuracy in finitely many arithmetic operations and comparisons (see also Appendix \ref{basic_spec22}). We assume that we approximate from below to an accuracy of $1/n_1$ and call this approximation $\tilde \gamma_{n_2,n_1}$. The function $\gamma_{n_2,n_1}(z;A)$ is Lipschitz continuous with Lipschitz constant bounded by $1$. Define the set
$$
V_{n_1}=\bigcup_{m=1}^{n_1}U_{m},
$$
where $U_{m}$ are the approximations to the open set $U$. By taking squares of distances to ball centres, we can decide whether a point $z\in\mathbb{Q}+i\mathbb{Q}$ has $\mathrm{dist}(z,V_{n_1})<\eta$ for any given $\eta\in\mathbb{Q}_+$. Let $\Upsilon_{n_2,n_1}(A,U)$ be the finite collection of all $z\in\tilde \Gamma_{n_2,n_1}(A)$ with $\mathrm{dist}(z,V_{n_1})<1/n_2-1/n_1$. If $\Upsilon_{n_2,n_1}(A,U)$ is empty then set $Q_{n_2,n_1}(A,U)=0$, otherwise set
$$
Q_{n_2,n_1}(A,U):=\sup_{z\in\Upsilon_{n_2,n_1}(A,U)} \tilde\gamma_{n_2,n_1}(z;A)-\frac{1}{n_1}.
$$
The above remarks show that this can be computed using finitely many arithmetic operations and comparisons.

Let $W_{n_2}=\overline{W((I-P_{n_2})A|_{(I-P_{n_2})\mathcal{H}})}$ and $W_{n_2,n_1}=W((I-P_{n_2})P_{n_1+n_2+1}A|_{P_{n_1+n_2+1}(I-P_{n_2})\mathcal{H}})$. We claim that the set $\Upsilon_{n_2,n_1}(A,U)$ converges to
$$
\Upsilon_{n_2}(A,U):=\overline{\left\{z\in W_{n_2}:\mathrm{dist}(z,\overline{U})<\frac{1}{n_2}\right\}},
$$
as $n_1\rightarrow\infty$, meaning also if $\Upsilon_{n_2}(A,U)$ is empty then $\Upsilon_{n_2,n_1}(A,U)$ is empty for large $n_1$. If $z\in \Upsilon_{n_2,n_1}(A,U)$, then there exists $\hat{z}\in W_{n_2,n_1}\subset W_{n_2}$ with $\left|z-\hat{z}\right|\leq1/{n_1}$. Since
$$
\mathrm{dist}(z,\overline{U})\leq\mathrm{dist}(z,V_{n_1})<1/n_2-1/n_1,
$$
it follows that $\mathrm{dist}(\hat{z},\overline{U})<{1}/{n_2}$ and hence $\Upsilon_{n_2}(A,U)$ is non-empty. So to prove convergence, we only need to deal with the case $\Upsilon_{n_2}(A,U)\neq\emptyset$. The above argument also shows that any limit point of a subsequence $z_{m(j)}\in\Upsilon_{n_2,m(j)}(A,U)$ must lie in $\Upsilon_{n_2}(A,U)$. Hence to prove the claim, we need to only prove that for any $z\in\Upsilon_{n_2}(A,U)$, there exists $z_{n_1}$ that are contained in $\Upsilon_{n_2,n_1}(A,U)$ for large $n_1$ and converge to $z$.

Let $z\in W_{n_2} $ with $\mathrm{dist}(z,\overline{U})<{1}/{n_2}$, then there exists $\epsilon>0$ and $j>0$ such that $\mathrm{dist}(z,U_j)<{1}/{n_2}-\epsilon$. There also exists $z_{n_1}\in\tilde\Gamma_{n_2,n_1}(A)$ with $z_{n_1}\rightarrow z$. It must hold for $n_1>j$ that
\begin{align*}
\mathrm{dist}({z}_{n_1},V_{n_1})\leq \mathrm{dist}({z}_{n_1},V_j)&\leq \left|z_{n_1}-z\right|+\mathrm{dist}(z,U_j)\\
&<\left|z_{n_1}-z\right|+\frac{1}{n_2}-\epsilon.
\end{align*}
This last quantity is smaller than $1/n_2-1/n_1$ for large $n_1$ and hence $z_{n_1}\in\Upsilon_{n_2,n_1}(A,U)$ for large $n_1$. It follows for any $z\in\Upsilon_{n_2}(A,U)$, there exists $z_{n_1}$ that are contained in $\Upsilon_{n_2,n_1}(A,U)$ for large $n_1$ and converge to $z$.

Define
$$
Q_{n_2}(A,U):=\sup_{z\in\Upsilon_{n_2}(A,U)} \gamma_{n_2}(z;A),
$$
where we recall that $\gamma_{n_2}(z;A)=\min\{\sigma_{\mathrm{inf}}((A-zI){|_{P_{n_2}\mathcal{H}}}),\sigma_{\mathrm{inf}}((A^*-\bar{z}I){|_{P_{n_2}\mathcal{H}}})\}.$ If $z\in\Upsilon_{n_2,n_1}(A,U)$, then the above shows that there exists $\hat{z}\in\Upsilon_{n_2}(A,U)$ with $\left|z-\hat{z}\right|\leq1/{n_1}$. It follows that
\begin{align*}
\tilde\gamma_{n_2,n_1}(z;A)-\frac{1}{n_1}&\leq \gamma_{n_2,n_1}(z;A)-\frac{1}{n_1}\\
&\leq \gamma_{n_2,n_1}(\hat{z};A)\leq \gamma_{n_2}(z;A),
\end{align*}
where we have used the bound on the Lipschitz constant and the fact that $\gamma_{n_2,n_1}$ converge up to $\gamma_{n_2}$ (and uniformly on compact subsets of $\mathbb{C}$). It follows that $Q_{n_2,n_1}(A,U)\leq Q_{n_2}(A,U)$ and this also covers the case that $\Upsilon_{n_2}(A,U)=\emptyset$ if we define the supremum over the empty set to be $0$. The set convergence proven above and uniform convergence of $\tilde\gamma_{n_2,n_1}$ implies that $Q_{n_2,n_1}(A,U)$ converges to $Q_{n_2}(A,U)$. It is also clear that the $\Upsilon_{n_2}(A,U)$ are nested and converge down to $W_e(A)\cap\overline{U}$ since $W_{n_2}$ converges down to $W_{e}(A)$. The functions $\gamma_{n_2}$ also converge down to
$$
\gamma(z;A)=\left\|R(z,A)\right\|^{-1}
$$
uniformly on compact subsets of $\mathbb{C}$ and hence $Q_{n_2}(A,U)$ converges down to
$$
Q(A,U)=\sup_{z\in W_e(A)\cap\overline{U}} \left\|R(z,A)\right\|^{-1}.
$$

Define
$$
\Gamma_{n_3,n_2,n_1}(A,U)=1-\chi_{[0,1/{n_3}]}(Q_{n_2,n_1}(A,U))\in\{0,1\}.
$$
The above show that
$$
\lim_{n_1\rightarrow\infty} \Gamma_{n_3,n_2,n_1}(A,U)=1-\chi_{[0,1/{n_3}]}(Q_{n_2}(A,U))=:\Gamma_{n_3,n_2}(A,U).
$$
Since $\chi_{[0,1/n_3]}$ has right limits and $Q_{n_2}(A,U)$ are non-increasing,
$$
\lim_{n_2\rightarrow\infty} \Gamma_{n_3,n_2}(A,U)=1-\chi_{[0,1/{n_3}]}(Q(A,U)\pm):=\Gamma_{n_3}(A,U),
$$
where $\pm$ denotes one of the right or left limits (it is possible to have either). Now if $\Xi_{poll}^{\mathbb{C}}(A,U)=0$, then $\Gamma_{n_3}(A,U)=0$ for all $n_3$. But if $\Xi_{poll}^{\mathbb{C}}(A,U)=1$, then for large $n_3$, $\Gamma_{n_3}(A,U)=1$. Moreover, in this latter case, $\Gamma_{n_3}(A,U)=1$ signifies the existence of $z\in W_e(A)\cap\overline{U}$ with $\gamma(z;A)>0$ and hence $z\not\in\mathrm{Sp}(A)$. Hence $\{\Gamma_{n_3,n_2,n_1}\}$ provides a $\Sigma_3^A$ tower.

\textbf{Step 2:} $\{\Xi_{poll}^{\mathbb{R}},\Omega_{\mathrm{D}},\Lambda_2\}\not\in\Delta^G_3$. We will argue for the case that $U=U_1=\mathbb{R}$ and the restricted case is similar. Assume for a contradiction that this is false and that $\{\widehat{\Gamma}_{n_2,n_1}\}$ is a general height two tower for $\{\Xi_{poll}^{\mathbb{R}},\Omega_{\mathrm{D}},\Lambda_2\}$. We follow the same strategy as the proof of Theorem \ref{spec_rad_thm} step 4 (recall also the results of \S\ref{bigHth}). Let $(\mathcal{M},d)$ be discrete space $\{0,1\}$ and $\tilde\Omega$ denote the collection of all infinite matrices $\{a_{i,j}\}_{i,j\in\mathbb{N}}$ with entries $a_{i,j}\in\{0,1\}$ and consider the problem function
\begin{equation*}
\tilde\Xi_1(\{a_{i,j}\}):\text{ Does }\{a_{i,j}\}\text{ have a column containing infinitely many non-zero entries?}
\end{equation*}

For $j\in\mathbb{N}$, let $\{b_{i,j}\}_{i\in\mathbb{N}}$ be a dense subset of $I_j:=[1-1/2^{2j-1},1-1/{2^{2j}}]$. Given a matrix $\{a_{i,j}\}_{i,j\in\mathbb{N}}\in\tilde\Omega$, construct a matrix $\{c_{i,j}\}_{i,j\in\mathbb{N}}$ by letting $c_{i,j}=a_{i,j}b_{r(i,j),j}$ where
$$
r(i,j)=\max\left\{1,\sum_{k=1}^{i}a_{k,j}\right\}.
$$
Now consider any bijection $\phi:\mathbb{N}\rightarrow\mathbb{N}^2$ and define the diagonal operator
\begin{equation*}
A=\mathrm{diag}(c_{\phi(1)},c_{\phi(2)},c_{\phi(3)},\ldots).
\end{equation*}
The algorithm $\widehat{\Gamma}_{n_2,n_1}$ thus translates to an algorithm $\Gamma'_{n_2,n_1}$ for $\{\tilde\Xi_1,\tilde\Omega\}$. Namely, set $\Gamma'_{n_2,n_1}(\{a_{i,j}\}_{i\in\mathbb{N}})=\widehat{\Gamma}_{n_2,n_1}(A)$. The fact that $\phi$ is a bijection shows that the lowest level $\Gamma'_{n_2,n_1}$ are generalised algorithms (and are consistent). In particular, given $N$, we can find $\{A_{i,j}:i,j\leq N\}$ using finitely many evaluations of the matrix values $\{c_{k,l}\}$ (the same is true for $A^*A$ and $AA^*$ since the operator is diagonal). But for any given $c_{k,l}$ we can evaluate this entry using only finitely many evaluations of the matrix values $\{a_{m,n}\}$ by the construction of $r$. Finally, note that
\begin{equation*}
\mathrm{Sp}(A)=\{1\}\cup\left(\bigcup_{j:\{a_{i,j}\}_{i\in\mathbb{N}}\text{ has infinitely many 1s}}I_j\right)\cup Q,
\end{equation*}
where $Q$ lies in the discrete spectrum. The intervals $I_j$ are also separated. It follows that there is a gap in the essential spectrum if and only if there exists a column $\{a_{i,j}\}_{i\in\mathbb{N}}$ with infinitely many 1s. Otherwise the essential spectrum is $\{1\}$. It follows that $\tilde \Xi(\{a_{i,j}\})=\Xi_{poll}^{\mathbb{R}}(A,\mathbb{R})$, and hence we get a contradiction.
\end{proof}

\subsection{Essential numerical range for unbounded operators}
\label{append_UB_ENR}

The essential numerical range (see \eqref{ess_num_RR_deff}) was first introduced for a bounded operator $A$ in \cite{stampfli1968growth}, as the closure of the numerical range of the image of $A$ in the Calkin algebra:
$$
W_e(A)=\bigcap_{B\in\Omega_K}\overline{W(A+B)}.
$$
Other equivalent characterisations were then given in \cite{fillmore1972essential}. The unbounded case is significantly different from the bounded case, and definitions that are equivalent in the bounded case may yield very different sets in the unbounded case. The definition for unbounded operators appeared in \cite{bogli2020essential}, and required the development of several new ideas and tools. In this section, we let $\Omega_{\mathcal{C}}$ denote the set of closed operators $T$ with domain $\mathcal{D}(T)\subset l^2(\mathbb{N})$ such that the linear span of the canonical basis forms a core of $T$. This latter condition ensures that we can use the usual matrix representation of the operator $T$ and hence the evaluation functions $\Lambda_1$. We follow \cite{bogli2020essential} and define
\begin{equation}
\label{unbounded_num_R_def}
W_e(T)=\left\{\lambda\in\mathbb{C}:\exists\{x_n\}_{n\in\mathbb{N}}\subset\mathcal{D}(T),\|x_n\|=1,x_n\xrightarrow{w}0,\lim_{n\rightarrow\infty}\langle Tx_n,x_n\rangle=\lambda\right\}.
\end{equation}
In \cite{bogli2020essential}, it was shown that for any $T\in\Omega_{\mathcal{C}}$, $W_e(T)$ consists precisely of the essential spectrum of $T$ together with all possible spectral pollution that may arise by applying projection methods to find the spectrum of $T$ numerically. This result therefore generalises Theorems \ref{polish} and \ref{t:conv}. The set $W_e(T)$ is closed and convex, but, unlike the case when $T$ is bounded, $W_e(T)$ may be empty. We first need two simple lemmas.

\begin{lemma}
\label{fhwif1}
Let $T\in\Omega_{\mathcal{C}}$, then $W(P_nT|_{P_n\mathcal{H}})\uparrow \overline{W(T)}$ in the Attouch--Wets topology as $n\rightarrow\infty$.
\end{lemma}

\begin{proof}
It is clear that
$$
W(P_nT|_{P_n\mathcal{H}})\subset W(T):=\{\langle Tx,x\rangle:x\in\mathcal{D}(T),\|x\|=1\},
$$
and that the sets $W(P_nT|_{P_n\mathcal{H}})$ are increasing with $n$. Now let $\lambda\in\overline{W(T)}$ be arbitrary. It is enough to show that there exists $\lambda_n\in W(P_nT|_{P_n\mathcal{H}})$ such that $\lambda_n\rightarrow\lambda$ as $n\rightarrow\infty$. By assumption, there exists $x_n\in\mathcal{D}(T)$ such that $\|x_n\|=1$ and $\lim_{n\rightarrow\infty}\langle Tx_n,x_n\rangle=\lambda$. Since the linear span of the canonical basis forms a core of $T$, we can assume without loss of generality that each $x_n$ has finite support with respect to the canonical basis. By taking subsequences if necessary, we may assume that $P_nx_n=x_n$ and hence $\langle Tx_n,x_n\rangle\in W(P_nT|_{P_n\mathcal{H}})$. The result now follows.
\end{proof}

\begin{lemma}
\label{fhwif2}
Let $T\in\Omega_{\mathcal{C}}$. If $W_e(T)\neq\emptyset$, then $\overline{W((I-P_n)T|_{(I-P_n)\mathcal{H}})}\downarrow W_{e}(T)$ in the Attouch--Wets topology as $n\rightarrow\infty$. If $W_e(T)=\emptyset$, then for any compact set $K$, $K\cap\overline{W((I-P_n)T|_{(I-P_n)\mathcal{H}})}=\emptyset$ for large $n$.
\end{lemma}

\begin{proof}
We clearly have that $\overline{W((I-P_n)T|_{(I-P_n)\mathcal{H}})}$ are non-empty and decreasing in $n$. It is enough to show the following two results:
\begin{enumerate}
	\item If $\lambda \in W_e(T)$, then $\lambda\in\overline{W((I-P_n)T|_{(I-P_n)\mathcal{H}})}$ for all $n$.
	\item If $\lambda \notin W_e(T)$, then $\liminf_{n\rightarrow\infty} \mathrm{dist}(\lambda,\overline{W((I-P_n)T|_{(I-P_n)\mathcal{H}})})>0$.
\end{enumerate}

We first prove (1), so assume that $\lambda \in W_e(T)$. Then, since the linear span of the canonical basis functions form a core of $T$, we can assume that there exists $x_n$ with $\|x_n\|=1$ such that each $x_n$ has finite support with respect to the canonical basis, $x_n\xrightarrow{w}0$ and $\lim_{n\rightarrow\infty}\langle Tx_n,x_n\rangle=\lambda$. It follows that for any fixed $m$, $\lim_{n\rightarrow\infty}P_mx_n=0$ and hence $\lambda\in \overline{W((I-P_m)T|_{(I-P_m)\mathcal{H}})}$.

Finally, to see (2), suppose that this were false for some $\lambda \notin W_e(T)$. We may then choose $\lambda_n\in  \overline{W((I-P_n)T|_{(I-P_n)\mathcal{H}})}$ such that $\liminf_{n\rightarrow\infty} |\lambda-\lambda_n|=0$. By taking subsequences if necessary, we may assume that $\lambda_n\rightarrow \lambda$ and that there exists $x_n$ with $\|x_n\|=1$, $P_{n}x_n=0$ and $|\langle Tx_n,x_n\rangle-\lambda_n|\rightarrow 0$. But this implies that $x_n\xrightarrow{w}0$ and $\lim_{n\rightarrow\infty}\langle Tx_n,x_n\rangle=\lambda$. Therefore $\lambda \in W_e(T)$, the required contradiction.
\end{proof}

We have the following corollary, which shows that the SCI classification of computing $W_e(T)$ for $T\in\Omega_{\mathcal{C}}$ remains $\Pi_2^A$ (one can make this precise by adding the empty set to the Attouch--Wets topology, but we omit the details).

\begin{corollary}
\label{gojoknb}
There exists a height two tower of arithmetic algorithms $\{\Gamma_{n_2,n_1}\}$, using $\Lambda_1$ (the matrix values with respect to the canonical basis) and $\Delta_1-$information (see Definition \ref{inexact_def_need}), such that for any $T\in\Omega_{\mathcal{C}}$, the following hold with respect to the Attouch--Wets topology:
\begin{itemize}
	\item $\Gamma_{n_2,n_1}(T)\uparrow \Gamma_{n_2}(T)\subset\overline{W(T)}$ as $n_1\rightarrow\infty$.
	\item If $W_e(T)\neq\emptyset$, then $\Gamma_{n_2}(T)\downarrow W_{e}(T)$ as $n_2\rightarrow\infty$. If $W_e(T)=\emptyset$, then for any compact set $K$, $K\cap\Gamma_{n_2}(T)=\emptyset$ for large $n_2$.
\end{itemize}
\end{corollary}

\begin{proof}
We simply let $\Gamma_{n_2,n_1}(T)$ be an approximation of
$$
W\left((I-P_{n_2})P_{n_1+n_2+1}T|_{P_{n_1+n_2+1}(I-P_{n_2})\mathcal{H}}\right)
$$
that can be computed in finitely many arithmetic operations and comparisons, even when using inexact input (see Definition \ref{inexact_def_need} and Remark \ref{fdjojosr}), using the arguments in \S \ref{gowgknk}. The results now follow from Lemmas \ref{fhwif1} and \ref{fhwif2}.
\end{proof}

\section{Proofs Concerning Lebesgue Measure}
\label{pf_leb}

We use the function $\texttt{DistSpec}$ in Appendix \ref{basic_spec22}. For ease of notation, we suppress the dispersion function $f$ in calling $\texttt{DistSpec}$, but assume that we know $\{c_n\}$ with $D_{f,n}(A)\leq c_n$ and $c_n\rightarrow 0$ as $n\rightarrow\infty$. However, the proof of convergence also works when using $c_n=0$ (which does not necessarily bound $D_{f,n}(A)$). The key observation is the following:

\vspace{2mm}
\textbf{\textit{Observation:}} If $A\in\Omega_f$, then the function $F_n(z):=\texttt{DistSpec}(A,n,z,f(n))+c_n$ converges uniformly to $\left\|R(z,A)\right\|^{-1}$ from above on compact subsets of $\mathbb{C}$. By taking successive minima, we can assume without loss of generality that $F_n$ is non-increasing in $n$.
\vspace{2mm}

The other ingredient needed is the following proposition
\begin{proposition}
\label{comp_prop}
Given a finite union of disks in the complex plane, the Lebesgue measure of their intersection with the interior of a rectangle can be computed within arbitrary precision, using finitely many arithmetical operations and comparisons on the centres and radii of the discs, as well as the position of the rectangle.
\end{proposition}

\begin{proof}
Without loss of generality, we assume that the rectangle is $\{x+iy:x,y\in[0,1]\}$. Consider dividing the rectangle into $n^2$ subrectangles using the division of $[0,1]$ into $n$ equal intervals. Given such a subrectangle, we can easily test via a finite number of arithmetic operations and comparisons whether the centre is in the union of the circles. Let $r(n)$ denote the number of subrectangles whose centre lies in the union. Then, since the boundary of the union of the circles has measure zero, it is easy to see that $r(n)/n^2$ converges to the desired Lebesgue measure. Moreover, we can bound the number of subrectangles that intersect the boundary of any of the circles, and this can be used to obtain any desired precision. 
\end{proof}

\begin{proof}[Proof of Theorem \ref{Leb_1}]

\textbf{Step 1:} $\{\Xi_1^L,\Omega_f,\Lambda_i\},\{\Xi_1^L,\Omega_\mathrm{D},\Lambda_i\}\in\Pi^A_2$. It is enough to consider $\Lambda_1$. We will estimate $\mathrm{Leb}(\mathrm{Sp}(A))$ by estimating the Lebesgue measure of the resolvent set on the closed square $[-C,C]^2$, where $\left\|A\right\|\leq C$. We do not assume $C$ is known. For $n_1,n_2\in\mathbb{N}$, let 
\begin{equation*}
\texttt{Grid}(n_1,n_2)=\left(\frac{1}{2^{n_2}}\mathbb{Z}+\frac{1}{2^{n_2}}i\mathbb{Z}\right)\cap[-n_1,n_1]^2.
\end{equation*}
Letting $B(x,r)$ and $D(x,r)$ denote the closed and open balls of radius $r$ around $x$, respectively\footnote{We set $D(x,0)=\emptyset$.}, in $\mathbb{C}$ (or $\mathbb{R}$ where appropriate), we define
\begin{equation*}
U(n_1,n_2,A)=[-n_1,n_1]\times[-n_1,n_1]\cap(\cup_{z\in\texttt{Grid}(n_1,n_2)}B(z,F_{n_1}(z))).
\end{equation*}
Note that $\mathrm{Leb}(U(n_1,n_2,A))$ can be computed up to arbitrary predetermined precision using only arithmetic operations and comparisons by Proposition \ref{comp_prop}. Using this we can define
$$
\Gamma_{n_2,n_1}(A)=4n_1^2-\mathrm{Leb}(U(n_1,n_2,A))
$$
where, without loss of generality, we assume that we have computed the exact value of the Lebesgue measure (since we can absorb this error in the first limit). $\Gamma_{n_2,n_1}$ are arithmetical algorithms using the fact that $\texttt{DistSpec}$ is and the above proposition. The only non-trivial part is convergence. The algorithm is summarised in the routine \texttt{LebSpec} in \S \ref{leb_num}.

We now show that the algorithm \texttt{LebSpec} converges and realises the $\Pi_2^A$ classification. There exists a compact set $K$ such that $\left\|R(z,A)\right\|^{-1}> 1$ on $K^c$ and without loss of generality we can make $C$ larger, $C\in\mathbb{N}$ and take $K=[-C,C]^2$. For $n_1\geq C$
$$
U(n_1,n_2,A)=([-C,C]^2\cap(\cup_{z\in\texttt{Grid}(n_1,n_2)}B(z,F_{n_1}(z))))\cup([-n_1,n_1]^2\backslash[-C,C]^2),
$$
since $F_{n}(z)\geq \left\|R(z,A)\right\|^{-1}$. It follows that for large $n_1$
$$
\Gamma_{n_2,n_1}(A)=4C^2-\mathrm{Leb}([-C,C]^2\cap(\cup_{z\in\texttt{Grid}(n_1,n_2)}B(z,F_{n_1}(z)))).
$$
As $n_1\rightarrow\infty$, $[-C,C]^2\cap(\cup_{z\in\texttt{Grid}(n_1,n_2)}B(z,F_{n_1}(z)))$ converges to the closed set
\begin{equation*}
K(n_2,A)=[-C,C]^2\cap(\cup_{z\in\texttt{Grid}(C,n_2)}B(z,\left\|R(z,A)\right\|^{-1}))
\end{equation*}
from above and hence
\begin{equation*}
\lim_{n_1\rightarrow\infty}\Gamma_{n_2,n_1}(A)=4C^2-\mathrm{Leb}(K(n_2,A)),
\end{equation*}
from below. Consider the relatively open set
\begin{equation*}
V(n_2,A)=[-C,C]^2\cap(\cup_{z\in\texttt{Grid}(C,n_2)}D(z,\left\|R(z,A)\right\|^{-1})).
\end{equation*}
Clearly, $\mathrm{Leb}(K(n_2,A))=\mathrm{Leb}(V(n_2,A))$ since the sets differ by a finite collection of circular arcs or points (recall we defined the open ball of radius zero to be the empty set). Hence we must show that 
\begin{equation*}
\lim_{n_2\rightarrow\infty}\mathrm{Leb}(V(n_2,A))=\mathrm{Leb}(\rho_{C}(A)),
\end{equation*}
where $\rho_{C}(A)=[-C,C]^2\backslash\mathrm{Sp}(A)$. For $z\in\rho_C(A)$,
$$
\mathrm{dist}(z,\mathrm{Sp}(A))\geq\left\|R(z,A)\right\|^{-1}
$$
and hence we get $V(n_2,A)\subset\rho_{C}(A)$. Since $\rho_C(A)$ is relatively open, a simple density argument using the continuity of $\left\|R(z,A)\right\|^{-1}$ yields $V(n_2,A)\uparrow\rho_{C}(A)$ as $n_2\rightarrow\infty$ since the grid refines itself. So we get
$$
\mathrm{Leb}(V(n_2,A))\uparrow\mathrm{Leb}(\rho_{C}(A)).
$$
This proves the convergence and also shows that $\Gamma_{n_2}(A)\downarrow \Xi_1^L(A)$, thus yielding the $\Pi_2^A$ classification. The same argument works in the one-dimensional case when considering self-adjoint operators $\Omega_{\mathrm{D}}$ and $\mathrm{Leb}_{\mathbb{R}}$. We simply restrict everything to the real line and consider the interval $[-C,C]$ rather than a square.

\textbf{Step 2:} $\{\Xi_1^L,\Omega_f,\Lambda_i\},\{\Xi_1^L,\Omega_\mathrm{D},\Lambda_i\}\notin\Delta^G_2$. It is enough to consider $\Lambda_2$. We will only show that $\mathrm{SCI}(\Xi_1^L,\Omega_\mathrm{D},\Lambda_2)_{G} \geq 2$ for which we use $\mathrm{Leb}_{\mathbb{R}}$ and the two-dimensional case is similar. Suppose for a contradiction that there exists a height one tower $\{\Gamma_n\}$, then $\Lambda_{\Gamma_n}(A)$ is finite for each $A\in\Omega_\mathrm{D}$. Hence, for every $A$ and $n$ there exists a finite number $N(A,n)\in\mathbb{N}$ such that the evaluations from $\Lambda_{\Gamma_n}(A)$ only take the matrix entries $A_{ij} = \left\langle Ae_j , e_i\right\rangle$ with $i,j\leq N(A,n)$ into account.

Pick any sequence $a_1,a_2,\ldots$ that is dense in the unit interval $[0,1]$. Consider the matrix operators $A_m=\mathrm{diag}\{a_1,a_2,\ldots,a_m\}\in\mathbb{C}^{m\times m}$, $B_m=\mathrm{diag}\{0,0,\ldots,0\}\in\mathbb{C}^{m\times m}$ and $C=\mathrm{diag}\{0,0,\ldots\}$. Set $A=\bigoplus_{m=1}^{\infty}(B_{k_m}\oplus A_{k_m})$, where we choose an increasing sequence $k_m$ inductively as follows. Set $k_1=1$ and suppose that $k_1,\ldots,k_m$ have been chosen. $\mathrm{Sp}(B_{k_1}\oplus A_{k_1}\oplus\cdots\oplus B_{k_m}\oplus A_{k_m}\oplus C)=\{0,a_1,a_2,\ldots,a_{k_m}\}$ and hence $\mathrm{Leb}(\mathrm{Sp}(B_{k_1}\oplus A_{k_1}\oplus\cdots\oplus B_{k_m}\oplus A_{k_m}\oplus C))=0$ so there exists some $n_m\geq m$ such that if $n\geq n_m$ then
\begin{equation*}
\Gamma_n(B_{k_1}\oplus A_{k_1}\oplus\cdots\oplus B_{k_m}\oplus A_{k_m}\oplus C)\leq \frac{1}{2}.
\end{equation*}
Now let $k_{m+1}\geq \max\{N(B_{k_1}\oplus A_{k_1}\oplus\cdots\oplus B_{k_m}\oplus A_{k_m}\oplus C,n_m),k_m+1\}$. Any evaluation function $f_{i,j}\in\Lambda$ is simply the $(i,j)^{\rm th}$ matrix entry and hence by construction
\begin{equation*}
f_{i,j}(B_{k_1}\oplus A_{k_1}\oplus\cdots\oplus B_{k_m}\oplus A_{k_m}\oplus C)=f_{i,j}(A),
\end{equation*}
for all $f_{i,j}\in\Lambda_{\Gamma_{n_m}}(B_{k_1}\oplus A_{k_1}\oplus\cdots\oplus B_{k_m}\oplus A_{k_m}\oplus C)$. By assumption (iii) in Definition \ref{Gen_alg}, it follows that $\Lambda_{\Gamma_{n_m}}(B_{k_1}\oplus A_{k_1}\oplus\cdots\oplus B_{k_m}\oplus A_{k_m}\oplus C)=\Lambda_{\Gamma_{n_m}}(A)$ and hence by assumption (ii) in the same definition that $\Gamma_{n_m}(A)=\Gamma_{n_m}(B_{k_1}\oplus A_{k_1}\oplus\cdots\oplus B_{k_m}\oplus A_{k_m}\oplus C)\leq 1/2$. But $\lim_{n\rightarrow\infty}(\Gamma_n(A))=\mathrm{Leb}(\overline{\{0,a_1,a_2,\ldots\}})=1$, a contradiction.

\textbf{Step 3}: $\{\Xi_1^L,\Omega,\Lambda_1\}\in\Pi^A_3$ for $\Omega=\Omega_{\mathrm{B}},\Omega_\mathrm{SA}$, $\Omega_\mathrm{N}$ or $\Omega_g$. We will deal with the case of $\Omega_{\mathrm{B}}$. The cases of $\Omega_\mathrm{N}$ and $\Omega_g$ then follow via $\Omega_{\mathrm{N}}\subset\Omega_{g}\subset \Omega_{\mathrm{B}}$ and the one-dimensional Lebesgue measure case for $\Omega_\mathrm{SA}$ is similar. A careful analysis of the proof in step 1 yields that
\begin{itemize}
	\item $\Gamma_{n_2,n_1}(A)$ converges to $\Gamma_{n_2}(A)$ from below as $n_1\rightarrow\infty$.
	\item $\Gamma_{n_2}(A)$ converges to $\mathrm{Leb}(\mathrm{Sp}(A))$ monotonically from above as $n_2\rightarrow\infty$.
\end{itemize}
We can ensure that the first limit converges from below by always slightly overestimating the Lebesgue measure of $U(n_1,n_2)$ (with error converging to zero) and using Proposition \ref{comp_prop}. These observations will be used later to answer question 3. We do not need to know $c_n$ for the above proof to work, but we will need it for the first of the above facts. A slight alteration of the proof/algorithm by inserting an additional successive limit deals with the general case.

Define the function
$$
\gamma_{n,m}(z;A)=\min\{\sigma_{\mathrm{inf}}(P_m(A-zI){|_{P_n\mathcal{H}}}),\sigma_{\mathrm{inf}}(P_m(A^*-\bar{z}I){|_{P_n\mathcal{H}}})\},
$$
where $\sigma_{\mathrm{inf}}$ denotes the injection modulus/smallest singular value (see also Appendix \ref{basic_spec22}). One can show that $\gamma_{n,m}$ converges uniformly on compact subsets to
$$
\gamma_{n}(z;A)=\min\{\sigma_{\mathrm{inf}}((A-zI){|_{P_n\mathcal{H}}}),\sigma_{\mathrm{inf}}((A^*-\bar{z}I){|_{P_n\mathcal{H}}})\},
$$
as $m\rightarrow\infty$ and that this converges uniformly down to $\left\|R(z,A)\right\|^{-1}$ on compact subsets as $n\rightarrow\infty$ \cite{hansen2011solvability}. With a slight abuse of notation, we can approximate $\gamma_{n,m}(z;A)$ to within $1/m$ by $\texttt{DistSpec}(A,n,z,m)$ (where the spacing of the search routine is $1/m$, see also Appendix \ref{basic_spec22}) so that this converges uniformly on compact subsets to $\gamma_{n}(z;A)$. In exactly the same manner as before, define
\begin{align*}
U(n_1,n_2,n_3,A)&=[-n_2,n_2]^2\cap(\cup_{z\in\texttt{Grid}(n_2,n_3)}B(z,\gamma_{n_2,n_1}(z;A))),\\
\Gamma_{n_3,n_2,n_1}(A)&=(2n_2)^2-\mathrm{Leb}(U(n_1,n_2,n_3,A)).
\end{align*}
The stated uniform convergence means that the argument in step 1 carries through and we have a height three tower, realising the $\Pi^A_3$ classification.

\textbf{Step 4}: $\{\Xi_1^L,\Omega_\mathrm{SA},\Lambda_1\}\notin\Delta^G_3$. The proof is exactly the same argument as the proof of step 3 of Theorem \ref{spec_rad_thm3}. However, in this case to gain the contradiction, we then define $\tilde{\Gamma}_{n_2,n_1}(\{a_{i,j}\})=\min\{\max\{1-\Gamma_{n_2,n_1}(A)/2,0\},1\}$ where $\{\Gamma_{n_2,n_1}\}$ is the supposed height two tower for $\{\Xi_1^L,\Omega_\mathrm{SA},\Lambda_1\}$.

\textbf{Step 5}: $\{\Xi_1^L,\Omega,\Lambda_1\}\notin\Delta^G_3$ for $\Omega=\Omega_{\mathrm{B}},\Omega_{\mathrm{N}}$, or $\Omega_g$. Since $\Omega_\mathrm{N}\subset\Omega_g\subset\Omega_{\mathrm{B}}$, we only need to deal with $\Omega_\mathrm{N}$. We can use a similar argument as in step 4, but now replacing each $C^{(j)}$ by
$$
D^{(j)}=\bigoplus_{k=1}^j ih_k C^{(j)},
$$
where $h_1,h_2,\ldots$ is a dense sequence in $[0,1]$ and this operators acts on $X_j=\bigoplus_{k=1}^j l^2(\mathbb{N})$. This ensures that the spectrum of the operator yields a positive two-dimensional Lebesgue measure if and only if $\tilde\Xi_2(\{a_{i,j}\})=0$. The rest of the argument is entirely analogous.

\textbf{Step 6}: $\Delta^G_2 \not\owns\{\Xi_1^L,\Omega,\Lambda_2\}\in\Pi^A_2$ for $\Omega=\Omega_{\mathrm{B}},\Omega_\mathrm{SA}$, $\Omega_\mathrm{N}$ or $\Omega_g$. The impossibility result follows by considering diagonal operators. For the existence of $\Pi^A_2$ algorithms, we can use the construction in step 3, but the knowledge of matrix values of $A^*A$ allows us to skip the first limit and approximate $\gamma_n$ directly.
\end{proof}

\begin{proof}[Proof of Theorem \ref{Leb_2}]
Using the convergence
$$
\lim_{\epsilon\downarrow 0}\mathrm{Leb}(\widehat{\mathrm{Sp}}_{\epsilon}(A))=\mathrm{Leb}(\mathrm{Sp}(A)),
$$
the lower bounds in Theorem \ref{Leb_1} immediately imply the lower bounds in Theorem \ref{Leb_2}. Hence we only need to construct the appropriate algorithms.

\textbf{Step 1}: $\{\Xi_2^L,\Omega_f,\Lambda_1\},\{\Xi_2^L,\Omega_{\mathrm{D}},\Lambda_1\}\in\Sigma_1^A$. Let $A\in\Omega_f$ and
$$
E_n=\frac{1}{n}\left(\mathbb{Z}+i\mathbb{Z}\right)\cap\{z\in\mathbb{C}:F_n(z)\leq\epsilon\}\cap[-n,n]^2.
$$
Clearly, we can compute $E_n$ with finitely many arithmetic operations and comparisons and we set
$$
\Gamma_n(A)=\mathrm{Leb}\left(\cup_{z\in E_n}D(z,\max\{0,\epsilon-F_n(z)\})\right).
$$
Proposition \ref{comp_prop} shows that, without loss of generality, we can assume $\Gamma_n(A)$ can be computed exactly using finitely many arithmetic operations and comparisons. The algorithm is presented in the \texttt{LebPseudoSpec} routine in \S \ref{leb_num} and the following shows that this algorithm is sharp in the SCI hierarchy.

Suppose that $F_n(z)<\epsilon$ and that $\left|w\right|<\epsilon-F_n(z)$. If $z\in\mathrm{Sp}(A)$ then clearly 
$$
\left\|R(z+w,A)\right\|^{-1}\leq\left|w\right|<\epsilon-F_n(z)\leq \epsilon,
$$
and this holds trivially if $z+w\in\mathrm{Sp}(A)$. So assume that neither of $z,z+w$ are in the spectrum. The resolvent identity yields
$$
\left\|R(z+w,A)\right\|\geq\left\|R(z,A)\right\|-\left|w\right|\left\|R(z+w,A)\right\|\left\|R(z,A)\right\|,
$$
which rearranges to
$$
\left\|R(z+w,A)\right\|^{-1}\leq\left\|R(z,A)\right\|^{-1}+\left|w\right|<\epsilon.
$$
It follows that $\cup_{z\in E_n}D(z,\max\{0,\epsilon-F_n(z)\})$ is in $\widehat{\mathrm{Sp}}_{\epsilon}(A)$ and hence that $\Gamma_{n}(A)\leq \Xi_2^L(A)$. Without loss of generality by taking successive maxima we can assume that $\Gamma_{n}(A)$ is increasing. Together, these will yield the $\Sigma^A_1$ classification once convergence is shown. Using the uniform convergence of $F_n$ and density of $1/n(\mathbb{Z}+i\mathbb{Z})\cap[-n,n]^2$, we see that pointwise convergence holds:
$$
\chi_{\cup_{z\in E_n}D(z,\max\{0,\epsilon-F_n(z)\}}\rightarrow\chi_{\widehat{\mathrm{Sp}}_{\epsilon}(A)},
$$
where $\chi_E$ denotes the indicator function of a set $E$. It follows by the dominated convergence theorem that $\Gamma_n(A)\rightarrow\mathrm{Leb}(\widehat{\mathrm{Sp}}_{\epsilon}(A))$. The proof for $\Omega_{\mathrm{D}}$ is similar by restricting everything to the real line.

\textbf{Step 2}: $\{\Xi_2^L,\Omega,\Lambda_1\}\in\Sigma_2^A$ for $\Omega=\Omega_{\mathrm{B}},\Omega_\mathrm{SA}$, $\Omega_\mathrm{N}$ or $\Omega_g$. To prove this, we simply replace $F_{n_1}$ by the functions $\gamma_{n_2,n_1}$ and set
$$
\Gamma_{n_2,n_1}(A)=\mathrm{Leb}\left(\cup_{z\in E_{n_2}}D(z,\max\{0,\epsilon-\gamma_{n_2,n_1}(z;A)\})\right).
$$

\textbf{Step 3}: $\{\Xi_2^L,\Omega,\Lambda_2\}\in\Sigma_1^A$ for $\Omega=\Omega_{\mathrm{B}},\Omega_\mathrm{SA}$, $\Omega_\mathrm{N}$ or $\Omega_g$. The knowledge of matrix values of $A^*A$ allows us to skip the first limit in the construction of step 2 and approximate $\gamma_n$ directly.
\end{proof}

\begin{proof}[Proof of Proposition \ref{leb_cty}]
We begin with the proof of 1. Suppose $A\in\Omega_\mathrm{D}$ has $\mathrm{Leb}_{\mathbb{R}}(\mathrm{Sp}(A))=0$ and let $A_n\in\Omega_\mathrm{D}$ be such that $\left\|A-A_n\right\|\rightarrow 0$ as $n\rightarrow\infty$. This implies that $\mathrm{Sp}(A_n)\rightarrow\mathrm{Sp}(A)$ since all our operators are normal. To prove that $\mathrm{Leb}_{\mathbb{R}}(\mathrm{Sp}(A_n))\rightarrow0$, it is enough to prove that
\begin{equation} 
\label{fn}
\mathrm{Leb}(F_n)\downarrow0,
\end{equation}
where $F_n=\mathrm{Sp}(A)\cup(\cup_{m\geq n}\mathrm{Sp}(A_m))$. But $F_n$ decreases to $\mathrm{Sp}(A)$ and is bounded in measure, so \eqref{fn} holds. For the converse, let $\mathrm{Leb}_{\mathbb{R}}(\mathrm{Sp}(A))>0$. Without loss of generality, assume that all of $A$'s entries lie in $[0,1]$. Let $\mathbb{D}_n$ denote the set $\{j/2^n\}_{j=1}^n$ and let us consider the map
$
\phi_n:x\hookrightarrow2^{-n}\left\lceil x2^n\right\rceil
$
on $[0,1]$. Let $A_n$ be the diagonal operator obtained by applying $\phi_n$ to each of $A$'s entries. We clearly have that $\left\|A-A_n\right\|\rightarrow 0$ as $n\rightarrow\infty$ but note that $\mathrm{Sp}(A_n)$ is finite so has Lebesgue measure $0$. Hence $\Xi_1^L$ is discontinuous at $A$.

To prove 2, note that for $A\in\Omega_\mathrm{D}$, $\mathrm{Leb}_{\mathbb{R}}(S_{\epsilon}(A))=0$. Let $A_n\in\Omega_\mathrm{D}$ have $\left\|A-A_n\right\|\rightarrow 0$. Then given some $0<\delta<\epsilon$ it holds for large $n$ that
$
\mathrm{Sp}_{\epsilon-\delta}(A)\subset\mathrm{Sp}_{\epsilon}(A_n)\subset\mathrm{Sp}_{\epsilon+\delta}(A)
$
and hence that
\begin{align*}
\limsup_{n\rightarrow\infty}\mathrm{Leb}_{\mathbb{R}}(\mathrm{Sp}_{\epsilon}(A_n))&\leq\mathrm{Leb}_{\mathbb{R}}(\mathrm{Sp}_{\epsilon+\delta}(A))\\
\liminf_{n\rightarrow\infty}\mathrm{Leb}_{\mathbb{R}}(\mathrm{Sp}_{\epsilon}(A_n))&\geq\mathrm{Leb}_{\mathbb{R}}(\mathrm{Sp}_{\epsilon-\delta}(A)).
\end{align*}
Now let $\delta\downarrow0$ and use the fact that $\Xi_2^L$ is continuous in $\epsilon$. 
\end{proof}

Finally, we deal with the question of determining if the Lebesgue measure is zero. Recall that for this problem, $(\mathcal{M},d)$ denotes the set $\{0,1\}$ endowed with the discrete topology and we consider the problem function
$$
\Xi_3^L(A)=
\begin{cases}
0,\quad\text{ if }\mathrm{Leb}(\mathrm{Sp}(A))>0\\
1,\quad\text{ otherwise.}
\end{cases}
$$

\begin{proof}[Proof of Theorem \ref{Leb_3}] We will show that $\{\Xi_3^L,\Omega_f,\Lambda_1\}\in \Pi_3^A$ and $\{\Xi_3^L,\Omega_\mathrm{D},\Lambda_2\}\notin\Delta_3^G$. The analogous statements $\{\Xi_3^L,\Omega_\mathrm{D},\Lambda_1\}\in \Pi_3^A$ and $\{\Xi_3^L,\Omega_f,\Lambda_2\}\notin\Delta_3^G $ follow from similar arguments.

The lower bound argument can also be used when considering $\Lambda_2$ and $\Omega=\Omega_{\mathrm{B}},\Omega_\mathrm{SA}$, $\Omega_\mathrm{N}$ or $\Omega_g$. We will also prove the lower bound $\{\Xi_3^L,\Omega_\mathrm{SA},\Lambda_1\}\notin \Delta_4^G$. The remaining lower bounds for $\Lambda_1$ follow from a similar argument and construction as in step 5 of the proof of Theorem \ref{Leb_1} to ensure we are dealing with two-dimensional Lebesgue measure. Finally, we prove that $\{\Xi_3^L,\Omega_{\mathrm{B}},\Lambda_1\}\in\Pi_4^A$. The upper bounds for $\Omega=\Omega_\mathrm{SA}$, $\Omega_\mathrm{N}$ or $\Omega_g$ and $\Lambda_1$ follow an almost identical argument. When considering $\Lambda_2$, we can collapse the first limit in the same manner as we did for solving $\Xi_1^L$.

\textbf{Step 1}: $\{\Xi_3^L,\Omega_f,\Lambda_1\}\in \Pi_3^A$. First we use the algorithm used to compute $\Xi_1^L$ in Theorem \ref{Leb_1}, which we shall denote by $\widetilde{\Gamma}$, to build a height 3 tower for $\{\Xi_3^L,\Omega_f\}$. As above, $\Omega_f$ denotes the set of bounded operators with the usual assumption of bounded dispersion (now with known bounds $c_n$). Recall that we observed
\begin{itemize}
	\item $\widetilde{\Gamma}_{n_2,n_1}(A)$ converges to $\widetilde{\Gamma}_{n_2}(A)$ from below as $n_1\rightarrow\infty$.
	\item $\widetilde{\Gamma}_{n_2}(A)$ converges to $\mathrm{Leb}(\mathrm{Sp}(A))$ monotonically from above as $n_2\rightarrow\infty$.
\end{itemize}
We can alter our algorithms, by taking maxima, so that we can assume without loss of generality that $\widetilde{\Gamma}_{n_2,n_1}(A)$ converges to $\widetilde{\Gamma}_{n_2}(A)$ \textit{monotonically} from below as $n_1\rightarrow\infty$. Now let
$$
\Gamma_{n_3,n_2,n_1}(A)=\chi_{[0,1/n_3]}(\widetilde{\Gamma}_{n_2,n_1}(A)).
$$
Note that $\chi_{[0,1/n_3]}$ is left continuous on $[0,\infty)$ with right limits. Hence by the assumed monotonicity
\begin{equation*}
\lim_{n_1\rightarrow\infty}\Gamma_{n_3,n_2,n_1}(A)=\chi_{[0,1/n_3]}(\widetilde{\Gamma}_{n_2}(A)).
\end{equation*}
It follows that
\begin{equation*}
\lim_{n_2\rightarrow\infty}\lim_{n_1\rightarrow\infty}\Gamma_{n_3,n_2,n_1}(A)=\chi_{[0,1/n_3]}(\mathrm{Leb}(\mathrm{Sp}(A))\pm),
\end{equation*}
where $\pm$ denotes one of the right or left limits (it is possible to have either). It is then easy to see that
\begin{equation*}
\lim_{n_3\rightarrow\infty}\lim_{n_2\rightarrow\infty}\lim_{n_1\rightarrow\infty}\Gamma_{n_3,n_2,n_1}(A)=\Xi_3^L(A).
\end{equation*}
It is also clear that the answer to the question is ``No'' if $\Gamma_{n_3}(A)=0$, which yields the $\Pi_3^A$ classification.

\textbf{Step 2}: $\{\Xi_3^L,\Omega_\mathrm{D},\Lambda_1\}\notin\Delta_3^G$.
Assume for a contradiction that this is false and $\{\widehat{\Gamma}_{n_2,n_1}\}$ is a general height two tower for $\{\Xi_3^L,\Omega_\mathrm{D}\}$. Let $(\mathcal{M},d)$ be discrete space $\{0,1\}$ and $\tilde\Omega$ denote the collection of all infinite matrices $\{a_{i,j}\}_{i,j\in\mathbb{N}}$ with entries $a_{i,j}\in\{0,1\}$ and consider the problem function
\begin{equation*}
\tilde\Xi_1(\{a_{i,j}\}):\text{ Does }\{a_{i,j}\}\text{ have a column containing infinitely many non-zero entries?}
\end{equation*}

For $j\in\mathbb{N}$, let $\{b_{i,j}\}_{i\in\mathbb{N}}$ be a dense subset of $I_j:=[1-1/2^{j-1},1-1/{2^j}]$. Given a matrix $\{a_{i,j}\}_{i,j\in\mathbb{N}}\in\tilde\Omega$, construct a matrix $\{c_{i,j}\}_{i,j\in\mathbb{N}}$ by letting $c_{i,j}=a_{i,j}b_{r(i,j),j}$ where
$$
r(i,j)=\max\left\{1,\sum_{k=1}^{i}a_{k,j}\right\}.
$$
Now consider any bijection $\phi:\mathbb{N}\rightarrow\mathbb{N}^2$ and define the diagonal operator
\begin{equation*}
A=\mathrm{diag}(c_{\phi(1)},c_{\phi(2)},c_{\phi(3)},\ldots).
\end{equation*}
The algorithm $\widehat{\Gamma}_{n_2,n_1}$ thus translates to an algorithm $\Gamma'_{n_2,n_1}$ for $\{\tilde\Xi_1,\tilde\Omega\}$. Namely, set $\Gamma'_{n_2,n_1}(\{a_{i,j}\}_{i\in\mathbb{N}})=\widehat{\Gamma}_{n_2,n_1}(A)$. The fact that $\phi$ is a bijection shows that the lowest level $\Gamma'_{n_2,n_1}$ are generalised algorithms (and are consistent). In particular, given $N$, we can find $\{A_{i,j}:i,j\leq N\}$ using finitely many evaluations of the matrix values $\{c_{k,l}\}$. But for any given $c_{k,l}$, we can evaluate this entry using only finitely many evaluations of the matrix values $\{a_{m,n}\}$ by the construction of $r$. Finally note that
\begin{equation*}
\mathrm{Sp}(A)=\left(\bigcup_{j:\sum_ia_{i,j}=\infty}I_j\right)\cup Q,
\end{equation*}
where $Q$ is at most countable. Hence
\begin{equation*}
\mathrm{Leb}_{\mathbb{R}}(\mathrm{Sp}(A))=\sum_{j:\sum_ia_{i,j}=\infty}\frac{1}{2^j}.
\end{equation*}
It follows that $\tilde \Xi_1(\{a_{i,j}\})=\Xi_3^L(A)$ and hence we get a contradiction.

\textbf{Step 3}: $\{\Xi_3^L,\Omega_\mathrm{SA},\Lambda_1\}\notin\Delta_4^G$. Suppose for a contradiction that $\{\Gamma_{n_3,n_2,n_1}\}$ is a height three tower of general algorithms for the problem $\{\Xi_3^L,\Omega_\mathrm{SA},\Lambda_1\}$. Let $(\mathcal{M},d)$ be the space $\{0,1\}$ with the discrete metric, let $\tilde{\Omega}$ denote the collection of all infinite arrays $\{a_{m,i,j}\}_{m,i,j\in\mathbb{N}}$ with entries $a_{m,i,j}\in\{0,1\}$ and consider the problem function
\begin{equation*}
\begin{split}
\tilde{\Xi}_4(\{a_{m,i,j}\}):\text{ For every $m$, does }\{a_{m,i,j}\}_{i,j}&\text{ have (only) finitely many columns}\\ &\quad\quad\quad\quad\quad\quad\quad\text{with (only) finitely many 1's?}
\end{split}
\end{equation*}
Recall that it was shown in \S \ref{bigHth} that $\mathrm{SCI}(\tilde{\Xi}_4,\tilde{\Omega})_{G} = 4$. We will gain a contradiction by using the supposed height three tower to solve $\{\tilde{\Xi}_4,\tilde{\Omega}\}$.

The construction follows step 3 of the proof of Theorem \ref{spec_rad_thm3} closely. For fixed $m$, recall the construction of the operator $A_m:=A(\{a_{m,i,j}\}_{i,j})$ from that proof, the key property being that if $\{a_{m,i,j}\}_{i,j}$ has (only) finitely many columns with (only) finitely many 1's then $\mathrm{Sp}(A_m)$ is a finite subset of $[-1,1]$, otherwise it is the whole interval $[-1,1]$. Now consider the intervals $I_m=[1-2^{m-1},1-2^{m}]$ and affine maps, $\alpha_m$, that act as a bijection from $[-1,1]$ to $I_m$. Without loss of generality, identify $\Omega_\mathrm{SA}$ with self adjoint operators in $\mathcal{B}(X)$ where $X=\bigoplus_{i=1}^{\infty}\bigoplus_{j=1}^{\infty}X_{i,j}$ in the $l^2$-sense with $X_{i,j}=l^2(\mathbb{N})$. We then consider the operator
$$
T(\{a_{m,i,j}\}_{m,i,j})=\bigoplus_{m=1}^\infty \alpha_{m}(A_m).
$$
The same arguments in the proof of Theorem \ref{spec_rad_thm3} show that the map
$$
\tilde\Gamma_{n_3,n_2,n_1}(\{a_{m,i,j}\}_{m,i,j})=\Gamma_{n_3,n_2,n_1}(T(\{a_{m,i,j}\}_{m,i,j}))
$$
defines a general tower using the relevant pointwise evaluation functions of the array $\{a_{m,i,j}\}_{m,i,j}$. If it holds that $\tilde{\Xi}_4(\{a_{m,i,j}\})=1$, then $\mathrm{Sp}(T(\{a_{m,i,j}\}_{m,i,j}))$ is countable and hence $\Xi_3^L(T(\{a_{m,i,j}\}_{m,i,j}))=1$. On the other hand, if $\tilde{\Xi}_4(\{a_{m,i,j}\})=0$, then there exists $m$ with $\mathrm{Sp}(A_m)=[-1,1]$ and hence $I_m\subset \mathrm{Sp}(T(\{a_{m,i,j}\}_{m,i,j}))$ so that $\Xi_3^L(T(\{a_{m,i,j}\}_{m,i,j}))=0$. It follows that $\{\tilde\Gamma_{n_3,n_2,n_1}\}$ provides a height three tower for $\{\tilde{\Xi}_4,\tilde{\Omega}\}$, a contradiction.

\textbf{Step 4}: $\{\Xi_3^L,\Omega_{\mathrm{B}},\Lambda_1\}\in\Pi_4^A$. Recall the tower of algorithms to solve $\{\Xi_1^L,\Omega_{\mathrm{B}},\Lambda_1\}$, and denote it by $\widetilde{\Gamma}$. Our strategy will be the same as in step 1 but with an additional successive limit. It is easy to show that
\begin{itemize}
	\item $\widetilde{\Gamma}_{n_3,n_2,n_1}(A)$ converges to $\widetilde{\Gamma}_{n_3,n_2}(A)$ from above as $n_1\rightarrow\infty$.
	\item $\widetilde{\Gamma}_{n_3,n_2}(A)$ converges to $\widetilde{\Gamma}_{n_3}(A)$ from below as $n_2\rightarrow\infty$.
	\item $\widetilde{\Gamma}_{n_3}(A)$ converges to $\mathrm{Leb}(\mathrm{Sp}(A))$ from above as $n_3\rightarrow\infty$.
\end{itemize}
Again, by taking successive maxima or minima where appropriate, we can assume that all of these are monotonic. Now let
$$
\Gamma_{n_4,n_3,n_2,n_1}(A)=\chi_{[0,1/n_4]}(\widetilde{\Gamma}_{n_3,n_2,n_1}(A)).
$$
Note that $\chi_{[0,1/n_4]}$ is left continuous on $[0,\infty)$ with right limits. Hence by the assumed monotonicity and arguments as in step 1, it is easy to see that
\begin{equation*}
\lim_{n_4\rightarrow\infty}\lim_{n_3\rightarrow\infty}\lim_{n_2\rightarrow\infty}\lim_{n_1\rightarrow\infty}\Gamma_{n_4,n_3,n_2,n_1}(A)=\Xi_3^L(A).
\end{equation*}
It is also clear that the answer to the question is ``No'' if $\Gamma_{n_4}(A)=0$, which yields the $\Pi_4^A$ classification.
\end{proof}

\section{Proofs Concerning Fractal Dimensions}
\label{pf_frac}

We begin with the box-counting dimension. For the construction of towers of algorithms, it is useful to use a slightly different but equivalent \cite{falconer2004fractal} definition of the upper and lower box-counting dimensions. Let $F\subset\mathbb{R}$ be bounded and $N'_\delta(F)$ denote the number of $\delta$-mesh intervals that intersect $F$. A $\delta$-mesh interval is an interval of the form $[m\delta,(m+1)\delta]$ for $m\in\mathbb{Z}$. Then
$$
\overline{\mathrm{dim}}_B(F)=\limsup_{\delta\downarrow{}0}\frac{\log(N'_{\delta}(F))}{\log(1/\delta)},\quad
\underline{\mathrm{dim}}_B(F)=\liminf_{\delta\downarrow{}0}\frac{\log(N'_{\delta}(F))}{\log(1/\delta)}.
$$

\begin{proof}[Proof of Theorem \ref{fractal_theoremhhhjh} for box-counting dimension] Since $\Omega_{BD}^\mathrm{D}\subset\Omega^{BD}_f\subset\Omega^{BD}_{\mathrm{SA}}$, it is enough to prove that $\{\Xi_B,\Omega^{BD}_f,\Lambda_1\}\in\Pi_2^A$, $\{\Xi_B,\Omega_{\mathrm{SA}}^{BD},\Lambda_2\}\in\Pi_2^A$, $\{\Xi_B,\Omega_{\mathrm{SA}}^{BD},\Lambda_1\}\in\Pi_3^A$, $\{\Xi_B,\Omega_{\mathrm{SA}}^{BD},\Lambda_1\}\not\in\Delta_3^A$ and $\{\Xi_B,\Omega^{BD}_\mathrm{D},\Lambda_2\}\not\in\Delta_2^A$.

\textbf{Step 1}: $\{\Xi_B,\Omega^{BD}_f,\Lambda_1\}\in\Pi_2^A$. Recall the existence of a height one tower, $\{\tilde\Gamma_n\}$, using $\Lambda_1$ for $\mathrm{Sp}(A)$, $A\in\Omega^{BD}_f$ from Appendix \ref{basic_spec22}. Furthermore, $\tilde\Gamma_n(A)$ outputs a finite collection $\{z_{1,n},\ldots,z_{k_n,n}\}\subset \mathbb{Q}$ such that $\mathrm{dist}(z_{j,n},\mathrm{Sp}(A))\leq2^{-n}$. Define the intervals
$$
I_{j,n}=[z_{j,n}-2^{-n},z_{j,n}+2^{-n}]
$$
and let $\mathcal{I}_{m}$ denote the collection of all $2^{-m}$-mesh intervals. Let $\Upsilon_{m,n}(A)$ be any union of finitely many such mesh intervals with minimal length $\left|\Upsilon_{m,n}(A)\right|$ (``length'' being the number of intervals $\in\mathcal{I}_{m}$ that make up $\Upsilon_{m,n}(A)$) such that
$$
\Upsilon_{m,n}(A)\cap I_{j,l}\neq\emptyset,\quad\text{for }1\leq l\leq n,1\leq j\leq k_l.
$$
There may be more than one such collection, so we can gain a deterministic algorithm by enumerating each $\mathcal{I}_{m}$ and choosing the first such collection in this enumeration. It is then clear that $\left|\Upsilon_{m,n}(A)\right|$ is increasing in $n$. Furthermore, to determine $\Upsilon_{m,n}(A)$, there are only finitely many intervals in $\mathcal{I}_{m}$ to consider, namely those that have non-empty intersection with at least one $I_{j,l}$ with $1\leq l\leq n,1\leq j\leq k_l$. It follows that $\Upsilon_{m,n}(A)$ and hence $\left|\Upsilon_{m,n}(A)\right|$ can be computed in finitely may arithmetic operations and comparisons using $\Lambda_1$. 

Suppose that $I=[a,b]\in\mathcal{I}_{m}$ has $(a,b)\cap\mathrm{Sp}(A)\neq\emptyset$. Then for large $n$ there exists $z_{j,n}\in I$ such that $I_{j,n}\subset I$ and hence $I\subset\Upsilon_{m,n}(A)$ for large $n$. If $z\in\mathrm{Sp}(A)\cap 2^{-m}\mathbb{Z}$, then a similar argument shows that $z\subset\Upsilon_{m,n}(A)$ for large $n$. Since $\mathrm{Sp}(A)$ is bounded and $\mathrm{Sp}(A)\cap 2^{-m}\mathbb{Z}$ finite, it follows that $\mathrm{Sp}(A)\subset\Upsilon_{m,n}(A)$ for large $n$ and hence
$$
N_{2^{-m}}(\mathrm{Sp}(A))\leq \liminf_{n\rightarrow\infty}\left|\Upsilon_{m,n}(A)\right|.
$$
Let $W_m(A)$ be the union of all intervals in $\mathcal{I}_m$ that intersect $\mathrm{Sp}(A)$. It is clear that $W_m(A)\cap I_{j,l}\neq\emptyset$ for $1\leq l\leq n,1\leq j\leq k_l$ and hence $\left|\Upsilon_{m,n}(A)\right|\leq N'_{2^{-m}}(\mathrm{Sp}(A))$. It follows that $\lim_{n\rightarrow\infty}\left|\Upsilon_{m,n}(A)\right|=\delta_{m}(A)$ exists with
\begin{equation}
\label{delta_bd}
N_{2^{-m}}(\mathrm{Sp}(A))\leq\delta_{m}(A)\leq N'_{2^{-m}}(\mathrm{Sp}(A)).
\end{equation}

For $n_2>n_1$ set $\Gamma_{n_2,n_1}(A)=0$, otherwise set
$$
\Gamma_{n_2,n_1}(A)=\max_{n_2\leq k \leq n_1} \max_{1\leq j\leq n_1}\frac{\log(\left|\Upsilon_{k,j}(A)\right|)}{k\log(2)}.
$$
The above monotone convergence and \eqref{delta_bd} shows that
\begin{align*}
\lim_{n_1\rightarrow\infty}\Gamma_{n_2,n_1}(A)&=\Gamma_{n_2}(A)=\sup_{k\geq n_2}\frac{\log(\delta_k(A))}{k\log(2)}\geq \limsup_{k\rightarrow\infty}\frac{\log(\delta_k(A))}{k\log(2)},\\
\lim_{n_2\rightarrow\infty}\Gamma_{n_2}(A)&=\limsup_{k\rightarrow\infty}\frac{\log(\delta_k(A))}{k\log(2)}.
\end{align*}
Hence, by the assumption that the box-counting dimension exists, we have constructed a $\Pi_2^A$ tower.

\textbf{Step 2}: $\{\Xi_B,\Omega_{\mathrm{SA}}^{BD},\Lambda_2\}\in\Pi_2^A$ and $\{\Xi_B,\Omega_{\mathrm{SA}}^{BD},\Lambda_1\}\in\Pi_3^A$. The first of these is exactly as in step 1, using $\Lambda_2$ to construct the relevant $\Sigma_1^A$ tower for the spectrum. The proof that $\{\Xi_B,\Omega_{\mathrm{SA}}^{BD},\Lambda_1\}\in\Pi_3^A$ uses a height two tower, $\{\tilde\Gamma_{n_2,n_1}\}$, using $\Lambda_1$ for $\mathrm{Sp}(A)$, $A\in\Omega_{\mathrm{SA}}^{BD}$ (or any self-adjoint $A$) constructed in \cite{ben2015can}. This tower has the property that each $\tilde\Gamma_{n_2,n_1}(A)$ is a finite subset of $\mathbb{Q}$ and, for fixed $n_2$, is constant for large $n_1$. Moreover, if $z\in\lim_{n_1\rightarrow\infty}\tilde\Gamma_{n_2,n_1}(A)$, then $\mathrm{dist}(z,\mathrm{Sp}(A))\leq2^{-n_2}$. It follows that we can use the same construction as step 1 with an additional limit at the start to reach the finite set $\lim_{n_1\rightarrow\infty}\tilde\Gamma_{n_2,n_1}(A)$.

\textbf{Step 3}: $\{\Xi_B,\Omega^{BD}_\mathrm{D},\Lambda_2\}\not\in\Delta_2^A$. This is exactly the same argument as step 2 of the proof of Theorem \ref{Leb_1} with Lebesgue measure replaced by box-counting dimension.

\textbf{Step 4}: $\{\Xi_B,\Omega_{\mathrm{SA}}^{BD},\Lambda_1\}\not\in\Delta_3^A$. This is exactly the same argument as step 4 of the proof of Theorem \ref{Leb_1} with Lebesgue measure replaced by box-counting dimension.
\end{proof}

We now turn to the Hausdorff dimension. Recall Lemma \ref{halt_test} on the problem of determining whether $\mathrm{Sp}(A)\cap(a,b)\neq\emptyset$.

\begin{proof}[Proof of Lemma \ref{halt_test}]
We start with the class $\Omega_f\cap\Omega_\mathrm{SA}$. We can interpret this problem as a decision problem and the following algorithm as one that halts on output ``Yes''. Let $c=(a+b)/2$ and $\delta=(b-a)/2$, then the idea is to simply test whether $\texttt{DistSpec}(A,n,c,f(n))+c_n<\delta$. If the answer is yes, then we output ``Yes'', otherwise we output ``No'' and increase $n$ by one. Note that $\mathrm{Sp}(A)\cap(a,b)\neq\emptyset$ if and only if $\left\|R(c,A)\right\|^{-1}<\delta$ and hence as $\texttt{DistSpec}(A,n,c,f(n))+c_n$ converges down to $\left\|R(c,A)\right\|^{-1}$ we see that this provides a convergent algorithm. For $\Omega_\mathrm{SA}$ we require an additional successive limit by replacing $\texttt{DistSpec}(A,n,c,f(n))+c_n$ with the function $\gamma_{n_2,n_1}(z;A)$. If we have access to $\Lambda_2$, then this can be avoided in the usual way.
\end{proof}

To build our algorithm for the Hausdorff dimension, we use an alternative, equivalent definition for compact sets. We consider the case of subsets of $\mathbb{R}$. Let $\rho_k$ denote the set of all closed binary intervals of the form
$
[2^{-k}m,2^{-k}(m+1)], m\in\mathbb{Z}.
$
Set
$$
\mathcal{A}_k(F)=\left\{\{U_i\}_{i\in I}:I\text{ is finite },F\subset\cup_{i\in I}U_i,U_i\in\cup_{l\geq k}\rho_l\right\}
$$
and define
$$
\tilde{\mathcal{H}}^{d}_{k}(F)=\inf\left\{\sum_i\mathrm{diam}(U_i)^d:\{U_i\}_{i\in I}\in\mathcal{A}_{k}(F)\right\},\quad\tilde{\mathcal{H}}^{d}(F)=\lim_{k\rightarrow\infty}\tilde{\mathcal{H}}^{d}_{k}(F).
$$
The following can be found in \cite{fernandez2014fractal} (Theorem 3.13):

\begin{theorem}
\label{new_haus}
Let $F$ be a bounded subset of $\mathbb{R}$. Then there exists a unique $d'=\mathrm{dim}_{H'}(F)$ such that $\tilde{\mathcal{H}}^{d}(F)=0$ for $d>d'$ and $\tilde{\mathcal{H}}^{d}(F)=\infty$ for $d<d'$. Furthermore, $d'=\mathrm{dim}_H(\overline{F})$.
\end{theorem}

Denoting the dyadic rationals by $\mathbb{D}$, we shall compute $\mathrm{dim}_H(\mathrm{Sp}(A))$ via approximating the above applied to $F=\mathrm{Sp}(A)\cap\mathbb{D}^c$ and using the lemma \ref{halt_test}.

\begin{proof}[Proof of Theorem \ref{fractal_theoremhhhjh} for Hausdorff dimension]
It is enough to prove the lower bounds $\{\Xi_H,\Omega_\mathrm{D},\Lambda_2\}\notin\Delta^G_3$, $\{\Xi_H,\Omega_\mathrm{SA},\Lambda_1\}\notin\Delta^G_4$ and construct the towers of algorithms for the inclusions $\{\Xi_H,\Omega_f\cap\Omega_\mathrm{SA},\Lambda_1\}\in\Sigma^A_3$, $\{\Xi_H,\Omega_\mathrm{SA},\Lambda_1\}\in\Sigma_4^A$ and $\{\Xi_H,\Omega_\mathrm{SA},\Lambda_2\}\in\Sigma_3^A$.

\textbf{Step 1}: $\{\Xi_H,\Omega_\mathrm{D},\Lambda_2\}\notin\Delta^G_3$. Suppose for a contradiction that a height two tower, $\{\Gamma_{n_2,n_1}\}$, exists for $\{\Xi_H,\Omega_\mathrm{D}\}$ (taking values in $[0,1]$ without loss of generality). We repeat the argument in the proof of Theorem \ref{Leb_3}. Consider the same problem 
\begin{equation*}
\tilde\Xi_1(\{a_{i,j}\}):\text{ Does }\{a_{i,j}\}\text{ have a column containing infinitely many non-zero entries?}
\end{equation*}
However, now we consider the above mapping to $[0,1]$ with the usual metric. We consider the same operator
$
A=\mathrm{diag}(c_{\phi(1)},c_{\phi(2)},c_{\phi(3)},\ldots)
$
with
\begin{equation*}
\mathrm{Sp}(A)=\left(\bigcup_{j:\sum_{i}a_{i,j}=\infty}I_j\right)\cup Q,
\end{equation*}
where $Q$ is at most countable. We use the fact that the Hausdorff dimension satisfies
$$
\mathrm{dim}_H(\cup_{j=1}^\infty X_j)=\sup_{j\in\mathbb{N}}\mathrm{dim}_H(X_j)
$$
and that $\mathrm{dim}_H(Q)=0$ for any countable $Q$ to note that $\Xi_H(A)=\tilde\Xi_1(\{a_{i,j}\})$. We set $\tilde{\Gamma}_{n_2,n_1}(\{a_{i,j}\}_{i,j})=\Gamma_{n_2,n_1}(A)$ to provide a height two tower for $\tilde\Xi_1$. But this contradicts Theorem \ref{DST_main}.

\textbf{Step 2}: $\{\Xi_H,\Omega_\mathrm{SA},\Lambda_1\}\notin\Delta^G_4$. Suppose for a contradiction that $\{\Gamma_{n_3,n_2,n_1}\}$ is a height three tower of general algorithms for the problem $\{\Xi_H,\Omega_\mathrm{SA},\Lambda_1\}$ (taking values in $[0,1]$ without loss of generality). Let $(\mathcal{M},d)$ be the space $[0,1]$ with the usual metric, let $\tilde{\Omega}$ denote the collection of all infinite arrays $\{a_{m,i,j}\}_{m,i,j\in\mathbb{N}}$ with entries $a_{m,i,j}\in\{0,1\}$ and consider the problem function

\begin{equation*}
\begin{split}
\tilde{\Xi}_4(\{a_{m,i,j}\}):\text{ For every $m$, does }\{a_{m,i,j}\}_{i,j}&\text{ have (only) finitely many columns}\\ &\quad\quad\quad\quad\quad\quad\quad\text{with (only) finitely many 1's?}
\end{split}
\end{equation*}
Recall that it was shown in \S \ref{bigHth} that $\mathrm{SCI}(\tilde{\Xi}_4,\tilde{\Omega})_{G} = 4$. We will gain a contradiction by using the supposed height three tower to solve $\{\tilde{\Xi}_4,\tilde{\Omega}\}$. We use the same construction as in step 3 of the proof of Theorem \ref{Leb_3}. If $\tilde{\Xi}_4(\{a_{m,i,j}\})=1$, then $\mathrm{Sp}(T(\{a_{m,i,j}\}_{m,i,j}))$ is countable and hence $\Xi_H(T(\{a_{m,i,j}\}_{m,i,j}))=0$. On the other hand, if $\tilde{\Xi}_4(\{a_{m,i,j}\})=0$, then there exists $m$ with $\mathrm{Sp}(A_m)=[-1,1]$ and hence $I_m\subset \mathrm{Sp}(T(\{a_{m,i,j}\}_{m,i,j}))$ so that $\Xi_H(T(\{a_{m,i,j}\}_{m,i,j}))=1$. It follows that $\tilde\Gamma_{n_3,n_2,n_1}(\{a_{m,i,j}\}_{m,i,j})=1-\Gamma_{n_3,n_2,n_1}(T(\{a_{m,i,j}\}_{m,i,j}))$ provides a height three tower for $\{\tilde{\Xi}_4,\tilde{\Omega}\}$, a contradiction.

\textbf{Step 3}: $\{\Xi_H,\Omega_f\cap\Omega_\mathrm{SA},\Lambda_1\}\in\Sigma_3^A$. To construct a height three tower for $A\in\Omega_f\cap\Omega_\mathrm{SA}$, if $n_2<n_3$ set $\Gamma_{n_3,n_2,n_1}(A)=0$. Otherwise, consider the set
$$
\mathcal{A}_{n_3,n_2,n_1}(A)=\left\{\{U_i\}_{i\in I}:I\text{ is finite },S_{n_1,n_2}(A)\subset\cup_{i\in I}U_i,U_i\in\cup_{n_3\leq l\leq n_2}\rho_l\right\}
$$
where $S_{n_1,n_2}(A)$ is the union of all $S\in\rho_{n_2}$ with $S\subset[-n_1,n_1]$ and such that the algorithm discussed in Lemma \ref{halt_test} outputs ``Yes'' for the interior of $S$ and input parameter $n_1$. We then define
$$
h_{n_3,n_2,n_1}(A,d)=\inf\left\{\sum_i\mathrm{diam}(U_i)^d:\{U_i\}\in\mathcal{A}_{n_3,n_2,n_1}(A)\right\}.
$$
If $S_{n_1,n_2}(A)$ is empty then we interpret the infinum as $0$. There are only finitely many sets to check and hence the infinum is a minimisation problem over finitely many coverings (see \S \ref{need_append_frac} for a discussion of efficient implementation). It follows that $h_{n_3,n_2,n_1}(A,d)$ defines a general algorithm computable in finitely many arithmetic operations and comparisons. Furthermore, it is easy to see that
$$
\lim_{n_1\rightarrow\infty}h_{n_3,n_2,n_1}(A,d)=\inf\left\{\sum_i\mathrm{diam}(U_i)^d:\{U_i\}\in\mathcal{C}_{n_3,n_2}(A)\right\}=:h_{n_3,n_2}(A,d)
$$
from below (since we are covering larger sets as $n_1$ increases). Here
$$
\mathcal{C}_{n_3,n_2}(A)=\left\{\{U_i\}_{i\in I}:I\text{ is finite },\mathrm{Sp}(A)\cap\mathbb{D}_{n_2}^c\subset\cup_{i\in I}U_i,U_i\in\cup_{n_3\leq l\leq n_2}\rho_l\right\}
$$
and $\mathbb{D}_k:=1/{2^k}\cdot\mathbb{Z}$ denotes the dyadic rationals of resolution $k$.
We now use the property that $\mathcal{A}_k(F)$ consists of collections of finite coverings. As $n_2\rightarrow\infty$, $h_{n_3,n_2}(A,d)$ is non-increasing (since we take infinum over a larger class of coverings and the sets $\mathrm{Sp}(A)\cap\mathbb{D}_{n_2}^c$ decrease) and hence converges to some number. Clearly
\begin{equation*}
\label{sandwich1}
\lim_{n_2\rightarrow\infty}h_{n_3,n_2}(A,d)=:h_{n_3}(A,d)\geq \tilde{\mathcal{H}}^d_{n_3}(\mathrm{Sp}(A)\cap\mathbb{D}^c).
\end{equation*}
For $\epsilon>0$, let $l\in\mathbb{N}$ and $\{U_i\}\in\mathcal{A}_{n_3}(\mathrm{Sp}(A)\cap\mathbb{D}_l^c)\}$ with
$$
\sum_{i}\mathrm{diam}(U_i)^d\leq\epsilon+\tilde{\mathcal{H}}^d_{n_3}(\mathrm{Sp}(A)\cap\mathbb{D}_l^c).
$$
For large enough $n_2$, $\{U_i\}\in\mathcal{C}_{n_3,n_2}(A)$ and hence since $\epsilon>0$ was arbitrary,
\begin{equation*}
\label{sandwich2}
h_{n_3}(A,d)\leq \tilde{\mathcal{H}}^d_{n_3}(\mathrm{Sp}(A)\cap\mathbb{D}_l^c)
\end{equation*}
for all $l$. For a fixed $A$ and $d$, $h_{n_3}(A,d)$ is non-decreasing in $n_3$ and hence converges to a function of $d$, $h(A,d)$ (possibly taking infinite values). Furthermore,
$$
\tilde{\mathcal{H}}^{d}(\mathrm{Sp}(A)\cap\mathbb{D}^c)\leq h(A,d) \leq \tilde{\mathcal{H}}^{d}(\mathrm{Sp}(A)\cap\mathbb{D}_l^c).
$$
Since the set $\mathrm{Sp}(A)\cap\mathbb{D}$ is countable, its Hausdorff dimension is zero. Using sub-additivity of Hausdorff dimension and Theorem \ref{new_haus},
\begin{align*}
\mathrm{dim}_H(\mathrm{Sp}(A))&\leq \mathrm{dim}_H(\mathrm{Sp}(A)\cap\mathbb{D}^c)\\
&\leq \mathrm{dim}_H(\overline{\mathrm{Sp}(A)\cap\mathbb{D}^c})=\mathrm{dim}_{H'}(\mathrm{Sp}(A)\cap\mathbb{D}^c)\\
&\leq \mathrm{dim}_H(\overline{\mathrm{Sp}(A)\cap\mathbb{D}^c_l})=\mathrm{dim}_{H'}(\mathrm{Sp}(A)\cap\mathbb{D}^c_l)\\
&\leq \mathrm{dim}_H(\mathrm{Sp}(A)).
\end{align*}
It follows that $h(A,d)=0$ if $d>\mathrm{dim}_H(\mathrm{Sp}(A))$ and that $h(A,d)=\infty$ if $d<\mathrm{dim}_H(\mathrm{Sp}(A))$. Define
$$
\Gamma_{n_3,n_2,n_1}(A)=\sup_{j=1,\ldots,2^{n_3}}\left\{\frac{j}{2^{n_3}}:h_{n_3,n_2,n_1}(A,k/{2^{n_3}})+\frac{1}{n_2}>\frac{1}{2}\text{ for }k=1,\ldots,j\right\},
$$
where in this case we define the maximum over the empty set to be $0$.

Consider $n_2\geq n_3$. Since $h_{n_3,n_2,n_1}(A,d)\uparrow h_{n_3,n_2}(A,d)$, it is clear that
$$
\lim_{n_1\rightarrow\infty}\Gamma_{n_3,n_2,n_1}(A)=\sup_{j=1,\ldots,2^{n_3}}\left\{\frac{j}{2^{n_3}}:h_{n_3,n_2}(A,k/2^{n_3})+\frac{1}{n_2}>\frac{1}{2}\text{ for }k=1,\ldots,j\right\}=:\Gamma_{n_3,n_2}(A).
$$
If $h_{n_3}(A,d)\geq1/2$ then $h_{n_3,n_2}(A,d)+1/{n_2}>1/2$ for all $n_2$ otherwise $h_{n_3,n_2}(A,d)+1/{n_2}<1/2$ eventually. Hence
$$
\lim_{n_2\rightarrow\infty}\Gamma_{n_3,n_2}(A)=\sup_{j=1,\ldots,2^{n_3}}\left\{\frac{j}{2^{n_3}}:h_{n_3}(A,k/{2^{n_3}})\geq\frac{1}{2}\text{ for }k=1,\ldots,j\right\}=:\Gamma_{n_3}(A).
$$
Using the monotonicity of $h_{n_3}(A,d)$ in $d$ and the proven properties of the limit function $h$, it follows that
$$
\lim_{n_3\rightarrow\infty}\Gamma_{n_3}(A)=\mathrm{dim}_H(\mathrm{Sp}(A)).
$$
The fact that $h_{n_3}$ is non-decreasing in $n_3$, the set $\{1/2^{n_3},2/2^{n_3},\ldots,1\}$ refines itself, and the stated monotonicity collectively show that convergence is monotonic from below, and hence we get the $\Sigma_3^A$ classification.

\textbf{Step 4}: $\{\Xi_H,\Omega_\mathrm{SA},\Lambda_1\}\in\Sigma_4^A$ and $\{\Xi_H,\Omega_\mathrm{SA},\Lambda_2\}\in\Sigma_3^A$. The first of these can be proven as in step 3 by replacing $(n_1,n_2,n_3)$ by $(n_2,n_3,n_4)$ and the set $S_{n_2,n_1}(A)$ by the set $S_{n_3,n_2,n_1}(A)$ given by the union of all $S\in\rho_{n_3}$ with $S\subset[-n_2,n_2]$ and such that the $\Sigma_2^A$ tower of algorithms discussed in Lemma \ref{halt_test} outputs ``Yes'' for the interior of $S$ and input parameters $(n_2,n_1)$. To prove $\{\Xi_H,\Omega_\mathrm{SA},\Lambda_2\}\in\Sigma_3^A$, we use exactly the same construction as in step 3 now using the $\Sigma_1^A$ algorithm (which uses $\Lambda_2$) given by Lemma \ref{halt_test}.
\end{proof}

\vspace{4mm}

\textbf{Acknowledgements.} This work was supported by EPSRC grant EP/L016516/1. I am grateful to Arno Pauly for discussions regarding Definition \ref{closed_under_search} and its use in Proposition \ref{search_prop}. Finally, I would like to thank Mohamed Nasser for generously sharing the code from \cite{liesen2017fast} for the computation of the capacity of finite unions of intervals.
\vspace{-2mm}

\linespread{1.17}

\bibliographystyle{spmpsci}
\bibliography{FOCM_bib}

\linespread{1.2}

\newpage
\appendix
\section{Routines for Computing Spectra}
\label{basic_spec22}
We describe the SCI-sharp $\Sigma_1^A$ algorithms in \cite{colb1} and \cite{colbrook3}, that are used in some of our proofs. In this section, we consider the problem functions $\Xi_1(A)=\mathrm{Sp}(A)$ and $\Xi_2(A)=\mathrm{Sp}_{\epsilon}(A)$, taking values in the space of non-empty compact subsets of $\mathbb{C}$ equipped with Hausdorff metric. The definitions of the classes $\Omega_g$ and $\Omega_f$ can be found in \S \ref{sec:rev_new_sec2}. As written, the outputs of the algorithms below may be empty for small $n$ (and hence not lie in the correct metric space). This does not affect the classifications and can be avoided by computing successive $\Gamma_n(A)$ and outputting $\Gamma_{m(n)}(A)$ where $m(n)\geq n$ is minimal with $\Gamma_{m(n)}(A)\neq\emptyset$. 

The methods in \cite{colb1} and \cite{colbrook3} use the function $f$ to approximate the function
\begin{equation}
\gamma_{n}(z;A)=\min\{\sigma_{\mathrm{inf}}((A-zI){|_{P_n\mathcal{H}}}),\sigma_{\mathrm{inf}}((A^*-\bar{z}I){|_{P_n\mathcal{H}}})\},
\end{equation}
where $P_m$ denotes the orthogonal projection onto the linear span of the first $m$ basis vectors and $\sigma_{\mathrm{inf}}$ denotes the injection modulus. As $n\rightarrow\infty$, the functions $\gamma_n$ converge uniformly on compact subsets down to the continuous function $\gamma(z;A)=\left\|R(z,A)\right\|^{-1}$, which we interpret as zero if the resolvent $R(z,A)=(A-zI)^{-1}$ does not exist as a bounded operator. The function $f$ and sequence $\{c_n\}$ allow us to approximate $\gamma_n$ to any given precision. To use this to compute the spectrum, we need some control on how the resolvent norm diverges near the spectrum and this is provided by the function $g$ satisfying \eqref{control_resolvent}. At various points in this paper, we have also made use of the related functions
\begin{equation}
\label{gojomnowknf}
\gamma_{n,m}(z;A)=\min\{\sigma_{\mathrm{inf}}(P_m(A-zI){|_{P_n\mathcal{H}}}),\sigma_{\mathrm{inf}}(P_m(A^*-\bar{z}I){|_{P_n\mathcal{H}}})\}.
\end{equation}
These can be computed from the rectangular matrices $P_m(A-zI)P_n,P_m(A-zI)^*P_n$ and converge uniformly on compact subsets of $\mathbb{C}$ to $\gamma_n$ as $m\rightarrow\infty$.

\vspace{2mm}

\begin{algorithm}[H]
  \SetKwInOut{Input}{Input}\SetKwInOut{Output}{Output}
\SetKwProg{Fn}{Function}{}{end} 
 \SetKwFunction{DistSpec}{DistSpec}
  \SetKwFunction{IsPosDef}{IsPosDef}
 \SetKwData{Left}{left}\SetKwData{This}{this}\SetKwData{Up}{up}
\SetKwFunction{Union}{Union}\SetKwFunction{FindCompress}{FindCompress}
 \Fn{\DistSpec{$A$,$n$,$z$,$f(n)$}}{
  \Input{$n \in \mathbb{N},$ $f(n) \in \mathbb{N}$, matrix $A$, $z \in \mathbb{C}$}
  \Output{$y \in \mathbb{R}_+$, an approximation to the function $z \mapsto \left\|R(z,A)\right\|^{-1}$}
 \BlankLine
 $B = (A-zI)(1:f(n),1:n)$; $C = (A-zI)^*(1:f(n),1:n)$\\
 $S = B^*B$; $T = C^*C$\\
 $\nu =  1$, $l = 0$\\
 \While{$\nu = 1$}{
 $l = l+1$\\
 $p = \IsPosDef(S - \frac{l^2}{n^2})$; $q = \IsPosDef(T - \frac{l^2}{n^2})$\\
 $\nu = \min(p,q)$
     }
 $y = \frac{l}{n}$
   }
	\caption{The subroutine \texttt{IsPosDef} checks whether a matrix is positive definite and is a standard routine that can be implemented in a myriad of ways. In practice, the while loop in \texttt{DistSpec} is replaced by a much more efficient interval bisection method. An alternative method for sparse matrices (which, however, does not rigorously guarantee an error bound on the smallest singular values) is to compute the smallest singular values of the rectangular matrices using iterative methods. See the supplementary material of \cite{colb1} for further discussion on efficient numerical computation. Note also that when evaluating \texttt{DistSpec} for different $z$, the computation can be done in parallel.}
\end{algorithm}

\vspace{2mm}

Throughout, we use that \texttt{DistSpec} requires only finitely many arithmetic operations and comparisons, as proven in \cite{colbrook3} (one can perform the \texttt{IsPosDef} routine using incomplete Cholesky decompositions). Furthermore, as outlined in Remark \ref{fdjojosr}, we can make all of the algorithms in this paper and those in this appendix work using $\Delta_1$-information and restricting to arithmetical operations over the rationals.

\vspace{2mm}

 \begin{algorithm}[H]
 \SetKwInOut{Input}{Input}\SetKwInOut{Output}{Output}
\SetKwProg{Fn}{Function}{}{end} 
 \SetKwFunction{CompInvg}{CompInvg}
 \SetKwData{Left}{left}\SetKwData{This}{this}\SetKwData{Up}{up}
 \Fn{\CompInvg{$n$,$y$,$g$}}{
  \Input{$n \in \mathbb{N},$ $y \in \mathbb{R}_+$, $g: \mathbb{R}_+ \rightarrow \mathbb{R}_+$}
  \Output{$m \in \mathbb{R}_+$, an approximation to $g^{-1}(y)$}
 \BlankLine
$m = \min\{k/n:k\in\mathbb{N},g(k/n)>y\}$ 
 \BlankLine
   }
 
   \SetKwInOut{Input}{Input}\SetKwInOut{Output}{Output}
 \BlankLine
\SetKwProg{Fn}{Function}{}{end} 
 \SetKwFunction{CompSpec}{CompSpec}
  \SetKwFunction{CompInvg}{CompInvg}
\SetKwFunction{Grid}{Grid}
 \SetKwData{Left}{left}\SetKwData{This}{this}\SetKwData{Up}{up}
  \SetKwFunction{DistSpec}{DistSpec}
\SetKwFunction{Union}{Union}\SetKwFunction{FindCompress}{FindCompress}
 \Fn{\CompSpec{$A$,$n$,$g$,$f(n)$,$c_n$}}{
  \Input{$n\in\mathbb{N}$, $f(n) \in \mathbb{N}$, $c_n\in\mathbb{R}_+$ (bound on dispersion), $g: \mathbb{R}_+ \rightarrow \mathbb{R}_+$, $A \in \Omega_f\cap\Omega_g$}
  \Output{$\Gamma_n(A) \subset \mathbb{C}$, an approximation to $\mathrm{Sp}(A)$, $E_n(A) \in \mathbb{R}_+$, the error estimate}
  \BlankLine
$G =\frac{1}{n}(\mathbb{Z}+i\mathbb{Z})\cap B_n(0)$\\

\For{$z \in G$}{$F(z) = \DistSpec{$A$,$n$,$z$,$f(n)$}$\\
 \eIf{$F(z) \leq (\left|z\right|^2+1)^{-1}$}{
 
\For{$w_j \in B_{\CompInvg{$n$,$F(z)$,$g$}}(z)\cap{}G=\{w_1,\ldots,w_k\}$}{ 
 $F_{j} = \DistSpec{$A$,$n$,$w_j$,$f(n)$}$
 }
 $M_z = \{w_{j} : F_{j} = \min_{q}\{F_{q}\}\}$
 }
{
$M_z = \emptyset$
}
 }
 $\Gamma_n(A) = \cup_{z \in G} M_z$\\
 $E_n(A) = \max_{z \in \Gamma_n(A)} \{\CompInvg{$n$,\DistSpec{$A$,$n$,$z$,$f(n)$}$+c_n$, $g$}\}$
 }
 \centering
\caption{The routine \texttt{CompSpec} computes spectra of bounded operators (see \cite{colbrook3} for extensions to unbounded operators) on $l^2(\mathbb{N})$ (or, more generally, graphs) using the subroutines \texttt{CompInvg} and \texttt{DistSpec} described above, and provides $\Sigma_1^A$ error control (without loss of generality by taking subsequences until the computed error is below a user specified tolerance).}\label{alg__1}
 \end{algorithm}

\vspace{2mm}

\begin{algorithm}[H]
 \SetKwInOut{Input}{Input}\SetKwInOut{Output}{Output}
 \BlankLine
\SetKwProg{Fn}{Function}{}{end} 
 \SetKwFunction{PseudoSpec}{PseudoSpec}
\SetKwFunction{Grid}{Grid}
 \SetKwData{Left}{left}\SetKwData{This}{this}\SetKwData{Up}{up}
  \SetKwFunction{DistSpec}{DistSpec}
\SetKwFunction{Union}{Union}\SetKwFunction{FindCompress}{FindCompress}
 \Fn{\PseudoSpec{$A$,$n$,$f(n)$,$c_n$, $\epsilon$}}{
  \Input{$n, f(n) \in \mathbb{N}$, $c_n \in \mathbb{R}_{+}^{\mathbb{N}}$, $A \in \Omega_f$, $\epsilon > 0$}
  \Output{$\Gamma \subset \mathbb{C}$, an approximation to $\mathrm{Sp}_{\epsilon}(A)$}
  \BlankLine
$G =$ \Grid{$n$}\\
$m  = \min\{k\geq n \, \vert \, c_k < \epsilon \}$\\
\For{$z \in G$}{

 $B = (A-zI)(1:f(m),1:m)$; $C = (A-zI)^*(1:f(m),1:m)$\\
 $S = B^*B$; $T = C^*C$\\
 $p = \IsPosDef(S - (\epsilon - c_m)^{2})$; $q = \IsPosDef(T - (\epsilon - c_m)^{2})$\\
 $\nu(z) = \min(p,q)$\\
 
 }
 $\Gamma = \bigcup \{z \in G \, \vert \nu(z)=0\}$
 }
\caption{\texttt{PseudoSpec} computes $\Gamma_n(A)\subset\mathrm{Sp}_\epsilon(A)$ with $\lim_{n\rightarrow\infty}\Gamma_n(A)=\mathrm{Sp}_\epsilon(A)$.}
\end{algorithm}

\newpage
\newgeometry{left=1.15in,right=1.15in,top=1.15in,bottom=0.6in}
\section{Examples of Computational Routines}
\label{compute_rout_appendix}

We provide short and simplified routines for some of the algorithms in this paper. For example, we have ignored issues like the rigorous approximation of the function $\gamma_{n,m}$ in \eqref{gojomnowknf} using arithmetical operations. For brevity, we stick to one domain $\Omega$ and the evaluation set $\Lambda_1$ (matrix values) for each problem function $\Xi$. In each case, we have chosen the non-trivial $\Omega$ with the simplest algorithm. For the different algorithms for different classes of operators, see the proofs. In general, different classes of operators and evaluation sets have different SCI classifications and different algorithms for the same problem function.

\subsection{Spectral radii, capacity and operator norms}

For the problem functions in \S \ref{sec:rev:spec_rad} - \ref{sec:rev:CAP}, we consider $\Omega_f$ (see \eqref{bd_disp2}) and $\Omega_f\cap\Omega_{\mathrm{SA}}$ for computing the capacity of the spectrum.

\vspace{2mm}

\begin{algorithm}[H]
 \SetKwInOut{Input}{Input}\SetKwInOut{Output}{Output}
 \BlankLine
\SetKwProg{Fn}{Function}{}{end} 
 \SetKwFunction{SpecRad}{SpecRad}
 \SetKwData{Left}{left}\SetKwData{This}{this}\SetKwData{Up}{up}
\SetKwFunction{Union}{Union}\SetKwFunction{FindCompress}{FindCompress}
 \Fn{\SpecRad{$n_1,n_2,f(n_1),c_{n_1},A$}}{
  \Input{$n_1,n_2,f(n_1)\in\mathbb{N}$, $c_{n_1}\in\mathbb{R}_{+}$, $A\in\Omega_f$}
  \Output{$\Gamma_{n_2,n_1}(A)$, a $\Pi_2^A$ approximation of $r(A)$}
  \BlankLine
	$S=\texttt{PseudoSpec}(A,n_1,f(n_1),c_{n_1}, n_2^{-1})=\{z_1,\ldots,z_m\}$\\
$\Gamma_{n_2,n_1}(A)=\sup_{1\leq j\leq m}|z_j|$}
\caption{\texttt{SpecRad} computes the spectral radius of operators in $\Omega_f$ using the algorithm for computing pseudospectra, \texttt{PseudoSpec}, which is parallelisable and provides $\Sigma_1^A$ error control.}
 \end{algorithm}

\vspace{2mm}

\begin{algorithm}[H]
 \SetKwInOut{Input}{Input}\SetKwInOut{Output}{Output}
 \BlankLine
\SetKwProg{Fn}{Function}{}{end} 
 \SetKwFunction{EssSpecRad}{EssSpecRad}
 \SetKwData{Left}{left}\SetKwData{This}{this}\SetKwData{Up}{up}
\SetKwFunction{Union}{Union}\SetKwFunction{FindCompress}{FindCompress}
 \Fn{\EssSpecRad{$n_1,n_2,f(n_1),c_{n_1},A$}}{
  \Input{$n_1,n_2,f(n_1)\in\mathbb{N}$, $c_{n_1}\in\mathbb{R}_{+}$, $A\in\Omega_f$}
  \Output{$\Gamma_{n_2,n_1}(A)$, a $\Pi_2^A$ approximation of $\Xi_{er}(A)$}
  \BlankLine
	$S=\texttt{EssSpec}(A,n_1,n_2,f(n_1),c_{n_1})=\cup_{j=1}^mR_j$\\
	NB: $R_j$ are rectangles with complex rational vertices.\\
$\Gamma_{n_2,n_1}(A)=\frac{1}{2^{n_2}}+\sup_{1\leq j\leq m}\max_{z\in R_j}|z|$}
\caption{\texttt{EssSpecRad} computes the essential spectral radius of operators in $\Omega_f$ using the algorithm for computing essential spectra, \texttt{EssSpec}, from \cite{ben2015can}.}
 \end{algorithm}

\vspace{2mm}

\begin{algorithm}[H]
 \SetKwInOut{Input}{Input}\SetKwInOut{Output}{Output}
 \BlankLine
\SetKwProg{Fn}{Function}{}{end} 
 \SetKwFunction{PolyNorm}{PolyNorm}
 \SetKwData{Left}{left}\SetKwData{This}{this}\SetKwData{Up}{up}
\SetKwFunction{Union}{Union}\SetKwFunction{FindCompress}{FindCompress}
 \Fn{\PolyNorm{$p,n,f(n),c_{n},A$}}{
  \Input{polynomial $p$, $n,f(n)\in\mathbb{N}$, $c_{n}\in\mathbb{R}_{+}$, $A\in\Omega_f$}
  \Output{$\Gamma_{n}(A)$, a $\Sigma_1^A$ approximation of $\|p(A)\|$}
  \BlankLine
	Compute $\hat B_n\approx B_n=P_np(A)P_n\in\mathbb{C}^{n\times n}$ using $f$ to compute matrix entries of powers of $A$.\\
Compute an upper bound $\delta_n$ of $\|\hat B_n-B_n\|$.\\
(Do the above so that $\delta_n$ is bounded by a null sequence.)\\
	$\Gamma_{n}(A)=\|\hat B_n\|-\delta_n$}
\caption{\texttt{PolyNorm} computes the operator norm of $p(A)$ for operators $A\in\Omega_f$ and polynomials $p$. The powers of $A$ can be computed through ``lazy evaluation'' (when one computes with infinite data structures, but defers the use of the information until needed) and the function $f$.}
 \end{algorithm}

\vspace{2mm}

\begin{algorithm}[H]
 \SetKwInOut{Input}{Input}\SetKwInOut{Output}{Output}
 \BlankLine
\SetKwProg{Fn}{Function}{}{end} 
 \SetKwFunction{CapSpec}{CapSpec}
 \SetKwData{Left}{left}\SetKwData{This}{this}\SetKwData{Up}{up}
\SetKwFunction{Union}{Union}\SetKwFunction{FindCompress}{FindCompress}
 \Fn{\CapSpec{$n_1,n_2,f(n_1),c_{n_1},A$}}{
  \Input{$n_1,n_2,f(n_1)\in\mathbb{N}$, $c_{n_1}\in\mathbb{R}_{+}$, $A\in\Omega_f\cap\Omega_{\mathrm{SA}}$}
  \Output{$\Gamma_{n_2,n_1}(A)$, a $\Pi_2^A$ approximation of the capacity, $\mathrm{cap}$, of $\mathrm{Sp}(A)$}
  \BlankLine
	Form a disjoint covering of $[-n_1,n_1]$ into intervals $I_j^{n_1,n_2}$, $j=1,\ldots,n_12^{n_2+1}$, of length $2^{-n_2}$\\
	Use Lemma \ref{halt_test} with $n=n_1$ to compute $\mathcal{I}_{n_1,n_2}\uparrow\{j:\mathrm{interior}(I_j^{n_1,n_2})\cap\mathrm{Sp}(A)\neq\emptyset\}$\\
	$\Gamma_{n_2,n_1}(A)=\mathrm{cap}\left(\cup_{j\in\mathcal{I}_{n_1,n_2}}\overline{I_j^{n_1,n_2}}\right)$}
\caption{\texttt{CapSpec} computes $\mathrm{cap}(\mathrm{Sp}(A))$ for operators $A\in\Omega_f\cap\Omega_{\mathrm{SA}}$. The capacity of a finite union of intervals can be computed using conformal mappings. The computation of $\mathcal{I}_{n_1,n_2}$ requires applications of \texttt{DistSpec} which can be performed in parallel.}
 \end{algorithm}

\subsection{Essential numerical range, gaps in essential spectra and detecting algorithm failure for finite section}

For the problems in \S \ref{fs_fails}, we consider $\Omega_\mathrm{B}$.

\vspace{2mm}

\begin{algorithm}[H]
 \SetKwInOut{Input}{Input}\SetKwInOut{Output}{Output}
 \BlankLine
\SetKwProg{Fn}{Function}{}{end} 
 \SetKwFunction{EssNumRange}{EssNumRange}
 \SetKwData{Left}{left}\SetKwData{This}{this}\SetKwData{Up}{up}
\SetKwFunction{Union}{Union}\SetKwFunction{FindCompress}{FindCompress}
 \Fn{\EssNumRange{$n_1,n_2,A$}}{
  \Input{$n_1,n_2\in\mathbb{N}$, $A\in\Omega_{\mathrm{B}}$}
  \Output{$\Gamma_{n_2,n_1}(A)$, a $\Pi_2^A$ approximation of $W_e(A)$}
  \BlankLine
	$B_{n_2,n_1}=(I-P_{n_2})P_{n_1+n_2}A|_{P_{n_1+n_2}(I-P_{n_2})\mathcal{H}}\in\mathbb{C}^{n_1\times n_1}$\\
	$\Gamma_{n_2,n_1}(A)=W(B_{n_2,n_1})$}
\caption{\texttt{EssNumRange} computes the essential numerical range for operators $A\in\Omega_{\mathrm{B}}$ (see \S \ref{append_UB_ENR} for unbounded operators). The numerical range of a finite square matrix can be approximated to arbitrary accuracy using finitely many arithmetic operations and comparisons. In practice, one can use the method of Johnson \cite{johnson1978numerical}, which reduces the computation of $\partial W(B)$ for $B\in\mathbb{C}^{n\times n}$ to a series of $n\times n$ Hermitian (extremal) eigenvalue problems.}
 \end{algorithm}

\vspace{2mm}

\begin{algorithm}[H]
 \SetKwInOut{Input}{Input}\SetKwInOut{Output}{Output}
 \BlankLine
\SetKwProg{Fn}{Function}{}{end} 
 \SetKwFunction{SpecPoll}{SpecPoll}
 \SetKwData{Left}{left}\SetKwData{This}{this}\SetKwData{Up}{up}
\SetKwFunction{Union}{Union}\SetKwFunction{FindCompress}{FindCompress}
 \Fn{\SpecPoll{$n_1,n_2,n_3,A,U$}}{
  \Input{$n_1,n_2,n_3,\in\mathbb{N}$, $A\in\Omega_{\mathrm{B}}$, open set $U$}
  \Output{$\Gamma_{n_3,n_2,n_1}(A,U)$, a $\Sigma_3^A$ approximation of $\Xi_{poll}^{\mathbb{C}}(A,U)$}
  \BlankLine
	$S_{n_2,n_1}=\texttt{EssNumRange}(n_1,n_2,A)=\{z_1,\ldots,z_m\}$\\
	NB: We use the version of \texttt{EssNumRange} that outputs a finite collection of points.\\
	$V_{n_1}=\cup_{j=1}^{n_1} U_j$\\
	$\Upsilon_{n_2,n_1}=\{z\in S_{n_2,n_1}:\mathrm{dist}(z,V_{n_1})<n_2^{-1}-n_1^{-1}\}$\\
	\eIf{$\Upsilon_{n_2,n_1}\neq\emptyset$}{
	$Q_{n_2,n_1}=\max_{z\in \Upsilon_{n_2,n_1}}\gamma_{n_2,n_1}(z;A)-n_1^{-1}$
 }
{
$Q_{n_2,n_1}=0$
}
\eIf{$Q_{n_2,n_1}\leq n_3^{-1}$}{$\Gamma_{n_3,n_2,n_1}(A,U)=0$
 }
{
$\Gamma_{n_3,n_2,n_1}(A,U)=1$
}
}
\caption{\texttt{SpecPoll} computes $\Xi_{poll}^{\mathbb{C}}(A,U)$ for operators $A\in\Omega_{\mathrm{B}}$ and open sets $U$ (given as a, possibly countably infinite, union of open balls $\{U_m\}$ with rational radii and centres). The function $\gamma_{n_2,n_1}$ is the same as in \eqref{gojomnowknf}.}
 \end{algorithm}

\subsection{Lebesgue measure}
\label{leb_num}

For the problems in \S \ref{Leb_sec}, we consider $\Omega_f$.

\begin{algorithm}[H]
 \SetKwInOut{Input}{Input}\SetKwInOut{Output}{Output}
 \BlankLine
\SetKwProg{Fn}{Function}{}{end} 
 \SetKwFunction{LebSpec}{LebSpec}
 \SetKwData{Left}{left}\SetKwData{This}{this}\SetKwData{Up}{up}
\SetKwFunction{Union}{Union}\SetKwFunction{FindCompress}{FindCompress}
 \Fn{\LebSpec{$n_1,n_2,f(n_1),c_{n_1},A$}}{
  \Input{$n_1,n_2,f(n_1)\in\mathbb{N}$, $c_{n_1}\in\mathbb{R}_{+}$, $A\in\Omega_f$}
  \Output{$\Gamma_{n_2,n_1}(A)$, a $\Pi_2^A$ approximation of $\mathrm{Leb}(\mathrm{Sp}(A))$}
  \BlankLine
$G=\frac{1}{2^{n_2}}\left(\mathbb{Z}+i\mathbb{Z}\right)\cap[-n_1,n_1]^2=\{z_1,\ldots,z_{m}\}$\\
\For{$z\in G$}{$F_{n_1}(z)=\texttt{DistSpec}(A,n_1,z,f(n_1))+c_{n_1}$\\

}
NB: WLOG we adapt $F_{n_1}$ to be non-increasing in $n_1$.\\
 $U(n_2,n_1,A)=[-n_1,n_1]^2\cap(\cup_{j=1}^{m}B(z_j,F_{n_1}(z_j)))$\\
$\Gamma_{n_2,n_1}(A)=4n_1^2-\mathrm{Leb}(U(n_2,n_1,A))$}
\caption{\texttt{LebSpec} computes $\mathrm{Leb}(\mathrm{Sp}(A))$ for operators $A\in\Omega_{f}$. It can be easily adapted to self-adjoint operators and computing the Lebesgue measure of the spectrum as a subset of the real line, by restricting the rectangles and balls to intervals. Again, the computation of \texttt{DistSpec} can be performed in parallel.}
 \end{algorithm}

\vspace{2mm}

\begin{algorithm}[H]
 \SetKwInOut{Input}{Input}\SetKwInOut{Output}{Output}
 \BlankLine
\SetKwProg{Fn}{Function}{}{end} 
 \SetKwFunction{LebPseudoSpec}{LebPseudoSpec}
 \SetKwData{Left}{left}\SetKwData{This}{this}\SetKwData{Up}{up}
\SetKwFunction{Union}{Union}\SetKwFunction{FindCompress}{FindCompress}
 \Fn{\LebPseudoSpec{$n,A,\epsilon$}}{
  \Input{$n\in\mathbb{N}$, $A\in\Omega_\epsilon^L$, $\epsilon>0$}
  \Output{$\Gamma_{n}(A)$, a $\Sigma_1^A$ approximation of $\mathrm{Leb}(\mathrm{Sp}_{\epsilon}(A))$}
  \BlankLine
	$G=\frac{1}{n}\left(\mathbb{Z}+i\mathbb{Z}\right)\cap[-n,n]^2=\{z_1,\ldots,z_{m}\}$\\
	\For{$z\in G$}{$F_n(z)=\texttt{DistSpec}(A,n,z,f(n))+c_{n}$}
	NB: WLOG we adapt $F_{n}$ to be non-increasing in $n$.\\
$S=\{z\in G:F_n(z)\leq\epsilon\}$\\
$\Gamma_n(A)=\mathrm{Leb}(\cup_{z\in S}D(z,\max\{0,\epsilon-F_n(z)\})$}
\caption{\texttt{LebPseudoSpec} computes $\mathrm{Leb}(\mathrm{Sp}_{\epsilon}(A))$ for operators $A\in\Omega_{f}$. It can be easily adapted to self-adjoint operators and computing the Lebesgue measure of the pseudospectrum restricted to the real line, by restricting the rectangles and balls to intervals. Again, the computation of \texttt{DistSpec} can be performed in parallel.}
 \end{algorithm}

\vspace{2mm}

\begin{algorithm}[H]
 \SetKwInOut{Input}{Input}\SetKwInOut{Output}{Output}
 \BlankLine
\SetKwProg{Fn}{Function}{}{end} 
 \SetKwFunction{NullLebSpec}{NullLebSpec}
 \SetKwData{Left}{left}\SetKwData{This}{this}\SetKwData{Up}{up}
\SetKwFunction{Union}{Union}\SetKwFunction{FindCompress}{FindCompress}
 \Fn{\NullLebSpec{$n_1,n_2,n_3,f(n_1),c_{n_1},A$}}{
  \Input{$n_1,n_2,n_3,f(n_1)\in\mathbb{N}$, $c_{n_1}\in\mathbb{R}_{+}$, $A\in\Omega_f$}
  \Output{$\Gamma_{n_3,n_2,n_1}(A)$, a $\Pi_3^A$ approximation of $\Xi_L^3(A)$}
  \BlankLine
\For{$j=1,\ldots,n_1$}{$t_j=\texttt{LebSpec}(j,n_2,f(j),c_{j},A)$
}
\eIf{$\max_{1\leq j\leq n_1}t_j\leq n_3^{-1}$}{$\Gamma_{n_3,n_2,n_1}(A)=1$
 }
{
$\Gamma_{n_3,n_2,n_1}(A)=0$
}
}
\caption{\texttt{NullLebSpec} computes $\Xi_L^3(A)$ (``Is $\mathrm{Leb}(\mathrm{Sp}(A))=0$?'') for operators $A\in\Omega_{f}$.  It can be easily adapted to self-adjoint operators, where the Lebesgue measure corresponds to that of the real line, by using the relevant adaptation of \texttt{LebSpec}.}
 \end{algorithm}

\subsection{Fractal dimensions}
\label{need_append_frac}

For the problems in \S \ref{frac_dims_sec}, we consider $\Omega_f^{BD}$ for the box-counting dimension and $\Omega_f\cap\Omega_{\mathrm{SA}}$ for the Hausdorff dimension.

\vspace{2mm}

\begin{algorithm}[H]
 \SetKwInOut{Input}{Input}\SetKwInOut{Output}{Output}
 \BlankLine
\SetKwProg{Fn}{Function}{}{end} 
 \SetKwFunction{BoxDim}{BoxDim}
 \SetKwData{Left}{left}\SetKwData{This}{this}\SetKwData{Up}{up}
\SetKwFunction{Union}{Union}\SetKwFunction{FindCompress}{FindCompress}
 \Fn{\BoxDim{$n_1,n_2,f(n_1),c_{n_1},A$}}{
  \Input{$n_1,n_2,f(n_1)\in\mathbb{N}$, $c_{n_1}\in\mathbb{R}_+$, $A\in\Omega^{BD}_f$}
  \Output{$\Gamma_{n_2,n_1}(A)$, a $\Pi_2^A$ approximation of $\mathrm{dim}_B(\mathrm{Sp}(A))$}
  \BlankLine
	\eIf{$n_1\geq n_2$}{
	\For{$l=1,\ldots,n_1$}{$S_l=\texttt{CompSpec}(A,l,f(l),c_{l},g:x\mapsto x)=\{z_{1,l},\ldots,z_{k_l,l}\}$\\
NB: WLOG we assume that $\mathrm{dist}(z_{j,l},\mathrm{Sp}(A))\leq 2^{-l}$.\\
\For{$j=1,\ldots,k_l$}{$I_{j,l}=[z_{j,l}-2^{-l},z_{j,l}+2^{-l}]$}
}
\For{$k\in\{n_2,n_2+1,\ldots,{n_1}\}$, $j\in\{1,2,\ldots,{n_1}\}$}{Let $\Upsilon_{k,j}$ be any union of $2^{-k}-$mesh intervals of minimal length $|\Upsilon_{k,j}|$ (where length is number of mesh intervals that make up the union) such that\\
\noindent$\Upsilon_{k,j}\cap I_{p,q}\neq \emptyset,\quad 1\leq q \leq j, \quad 1\leq p \leq k_q.$\\
\noindent{}Set $a_{k,j}=\frac{\log(\left|\Upsilon_{k,j}(A)\right|)}{k\log(2)}$}
$\Gamma_{n_2,n_1}(A)=\max\{a_{k,j}:n_2\leq k\leq n_1,1\leq j\leq n_1\}$ (max over empty set is zero).}
{$\Gamma_{n_2,n_1}(A)=0$}
 }
\caption{\texttt{BoxDimSpec} computes the box-counting dimension of the spectrum for operators $A\in\Omega_{f}^{BD}$. If the we enlarge the class to $\Omega_f\cap\Omega_{\mathrm{SA}}$, the result is a tower of algorithms that converges to a quantity $\Gamma(A)$ with $\underline{\mathrm{dim}}_B(\mathrm{Sp}(A))\leq \Gamma(A)\leq\overline{\mathrm{dim}}_B(\mathrm{Sp}(A))$.}
 \end{algorithm}

\vspace{2mm}

\begin{algorithm}[H]
 \SetKwInOut{Input}{Input}\SetKwInOut{Output}{Output}
 \BlankLine
\SetKwProg{Fn}{Function}{}{end} 
 \SetKwFunction{HausDimSpec}{HausDimSpec}
 \SetKwData{Left}{left}\SetKwData{This}{this}\SetKwData{Up}{up}
\SetKwFunction{Union}{Union}\SetKwFunction{FindCompress}{FindCompress}
 \Fn{\HausDimSpec{$n_1,n_2,n_3,A$}}{
  \Input{$n_1,n_2,n_3\in\mathbb{N}$, $A\in\Omega_f\cap\Omega_\mathrm{SA}$}
  \Output{$\Gamma_{n_3,n_2,n_1}(A)$, a $\Sigma_3^A$ approximation of $\mathrm{dim}_H(\mathrm{Sp}(A))$}
  \BlankLine
	Notation: $\rho_k$ denotes set of all closed intervals of form $[2^{-k}m,2^{-k}(m+1)]$, $m\in\mathbb{Z}$\\
	$S_{n_1,n_2}$ = union of all $S\in\rho_{n_2}$ with $S\subset[-n_1,n_1]$ and such that the algorithm discussed in Lemma \ref{halt_test} outputs ``Yes'' for the interior of $S$ and input parameter $n_1$.\\
$\mathcal{A}_{n_3,n_2,n_1}=\left\{\{U_i\}_{i\in I}:I\text{ is finite },S_{n_1,n_2}\subset\cup_{i\in I}U_i,U_i\in\cup_{n_3\leq l\leq n_2}\rho_l\right\}$\\
\For{$m\in\{1,\ldots,2^{n_3}\}$}{$
b_{m}=\inf\left\{\sum_i\mathrm{diam}(U_i)^{m/2^{n_3}}:\{U_i\}\in\mathcal{A}_{n_3,n_2,n_1}\right\}+n_2^{-1}$\\
 
}
$\Gamma_{n_3,n_2,n_1}(A)=\max\{m/{2^{n_3}}:b_j>1/2\text{ for }j=1,\ldots,m\}$ (max over empty set is zero).\\
 }
\caption{\texttt{HausDimSpec} computes the Hausdorff dimension of the spectrum for operators $A\in\Omega_f\cap\Omega_{\mathrm{SA}}$. An efficient way to compute the minimal covering is to use binary trees \cite{stewart2001towards}.}
 \end{algorithm}

\end{document}